\documentclass{amsart}
\usepackage{color, graphicx, enumerate, amssymb}
\usepackage{hyperref}
\usepackage{epsfig,wrapfig}
\usepackage{pxfonts}
\usepackage{graphicx}

\usepackage{eucal}

\usepackage[utf8]{inputenc}
\usepackage{ulem}

\usepackage{amsmath}

\usepackage[ulem]{changes}
\usepackage[normalem]{ulem}



\usepackage{geometry}
\usepackage{enumitem}

\newtheorem{theorem}{Theorem}[subsection] 
\newtheorem{lemma}[theorem]{Lemma}
\newtheorem{corollary}[theorem]{Corollary}

\newtheorem{definition}{Definition}[subsection]
\newtheorem{proposition}[theorem]{Proposition}
\newtheorem{remark}[theorem]{Remark}
\newtheorem{example}[theorem]{Example}

\newcommand{\C}{{\mathcal C}}
\newcommand{\F}{{\mathcal F}}

\title[Boundary Estimates For Guillemin Boundary PROBLEM]{Boundary Estimates For The Monge-Ampère Equation In The Polygons With Guillemin Boundary Conditions}

\author[M. Bayrami]{Masoud Bayrami}
\address[M. Bayrami]{Department of Mathematical Sciences, Sharif University of Technology, P.O. Box: 11365-9415, Tehran, Iran}
\email{masoud.bayrami1990@sharif.edu}

\author[R. Seyyedali]{Reza Seyyedali}
\address[R. Seyyedali]{School of Mathematics, Institute for Research in Fundamental Sciences (IPM), P.O. Box: 19395-5746, Tehran, Iran}
\email{rseyyedali@ipm.ir}

\author[M. Talebi]{Mohammad Talebi}
\address[M. Talebi]{School of Mathematics, Institute for Research in Fundamental Sciences (IPM), P.O. Box: 19395-5746, Tehran, Iran}
\email{mohammadtalebi@ipm.ir}

\date{\today}

\begin{document}

\begin{abstract}
We establish a Schauder-type boundary regularity result for a two-dimensional singular Monge--Amp\`ere equation on convex polytopes subject to the Guillemin boundary condition. Our result extends the work of Rubin and Huang to the case where the right-hand side is merely H\"older continuous. In particular, we obtain Euclidean \(C^{1,\alpha/2}\) regularity up to the boundary, including along edges and at the vertices of the polytope, and then refine this to the sharp Euclidean \(C^{1,\alpha}\) regularity. The analysis combines techniques introduced by Donaldson in his study of the Abreu equation with refined blow-up and localization arguments adapted to the degenerate boundary geometry.
\end{abstract}



\maketitle

\setcounter{tocdepth}{1}
\tableofcontents

\section{Introduction}
	
One of the fundamental questions in complex geometry concerns the existence and uniqueness of canonical metrics on compact Kähler manifolds, which include Kähler-Einstein metrics, constant scalar curvature Kähler metrics, and extremal metrics. In his early work, \cite{calabi1954the, calabi1957kahler}, Calabi shows that finding Kähler-Einstein metrics is equivalent to solving a complex Monge-Ampère equation. In later works, he introduces the concept of extremal metrics, which satisfy higher-order fully nonlinear elliptic equations (\cite{zbMATH03766648, zbMATH03917291}). Yau, in his seminal work \cite{Yau1978}, solves the complex Monge-Ampère equation with a smooth right-hand side on compact K\"ahler manifolds. As a result, he proves the existence of K\"ahler-Ricci flat metrics on compact manifolds with trivial canonical class. The existence problem for Kähler-Einstein metrics has been completely resolved  through the work of Yau, Aubin, Chen-Donaldson-Sun, and Tian, \cite{Yau1978, Aubin1978, ChenDonaldsonSun2015I, ChenDonaldsonSun2015II, ChenDonaldsonSun2015III, Tian1997}. However, the existence of constant scalar curvature K\"ahler metrics (cscK) and extremal metrics remains largely open. Even in the case of the complex Monge-Ampère equation, many questions regarding the regularity of solutions are still open. The celebrated result of Yau guarantees the existence of smooth solutions to the complex Monge-Ampère equation with a smooth right-hand side. However, for a less regular right-hand side, the regularity of solutions remains widely open. Most of these problems are highly challenging and far-reaching. One reason is that the regularity theory for the complex Monge-Ampère equation is inherently difficult and is still not well understood. For further information, we refer to references such as  \cite{Guedj2012, Kolodziej2005, GuedjZeriahi2005, Kolodziej2008}. On the other hand, the real Monge-Ampère equation has been extensively studied for decades. The regularity theory of the real Monge-Ampère equation is investigated in early works by Heinz, Pogorelov, Nirenberg, and others. References such as \cite{Heinz1959, Pogorelov1978, Nirenberg1953, MR4720871, zbMATH01614149, zbMATH06669881, zbMATH07788545} provide well-established results on the real Monge–Ampère equation. 
Caffarelli, in his seminal work \cite{zbMATH04154994}, provides a comprehensive analysis of the real Monge–Ampère equation from the perspective of regularity theory. His subsequent paper \cite{caffarelli1990localization} introduces the powerful localization technique, further advancing the understanding of interior regularity for solutions.
	
For toric Kähler manifolds, the existence problem for canonical metrics reduces to solving real Monge-Ampère or linearized Monge-Ampère equations. Therefore, it is natural to apply techniques from the real Monge-Ampère theory to establish the existence and regularity of canonical metrics on toric manifolds. Let \((X, \omega)\) be a compact toric manifold of dimension \( n \). The complex torus \((\mathbb{C}^*)^n\) acts on \(X\) with a dense free orbit. The action of the compact torus \((\mathbb{C}^*)^n\) on \((X, \omega)\) is Hamiltonian, which induces a moment map  
\(
\mu: X \to \mathbb{R}^n
\) 
that is unique up to translations. By a theorem of Atiyah-Bott \cite{atiyah1984moment} and Guillemin-Sternberg \cite{guillemin1982convexity}, the image \(P:=\mu(X) \subset \mathbb{R}^n\) is a Delzant polytope, independent of the choice of \(\omega \in [\omega]\) (see also \cite{guillemin-moment}). Furthermore, Guillemin \cite{guillemin1994kaehler} establishes the following theorem.
	
\begin{theorem}
\label{guillemin-result}
Let $l_1, \cdots, l_N$ be affine linear functions on $\mathbb{R}^n$ such that \begin{equation}
\label{polytope-set-def}
P=\{x \in \mathbb{R}^n \, : \, l_i(x) >0, \, i=1,\cdots, N \}.
\end{equation}
Then any toric K\"ahler metric on $(X,[\omega])$ corresponds to a strictly convex function $u:P \to \mathbb{R}$ that satisfies the following boundary condition:
\begin{equation*}
u - \sum_{i=1}^{N} l_i \log l_i \in C^{\infty}(\overline{P}),
\end{equation*}
which is called the "Guillemin boundary condition".
\end{theorem}
In the work of Abreu \cite{abreu1998kahler}, it is shown than the scalar curvature of the toric metric corresponding to a convex \(u\) is given by the following fourth-order formula:
\begin{equation}
\label{Abreu-eq-4th}
S(u)=-\sum u^{ij}_{ij}
\end{equation}
where $(u^{ij})$ is the inverse of the Hessian matrix $D^2u=(u_{ij})$.
For a prescribed $S=S(u)$, the above equation is called the Abreu equation, and \eqref{Abreu-eq-4th} can be written as a system of two second-order elliptic equations:
\begin{equation}
\label{Abreu-2-eq-v}
\begin{cases}
\det D^2 u =\varphi^{-1} \\
U^{ij}\varphi_{ij}=S(u),
\end{cases}
\end{equation}
where the latter one is the linearized Monge-Amp\'ere equation (studied by Caffarelli and Gutierrez in \cite{MR1439555}) and $U^{ij}$ denotes the co-factor matrix of the Hessian of $u$.
	
In a series of papers, \cite{MR1988506, donaldson2005interior, donaldson2008extremal, donaldson2009constant}, Donaldson studies the Abreu equation. His main motivation is to prove the existence of cscK metrics and extremal metrics on toric manifolds. In dimension \( n = 2 \), he proves the existence of cscK metrics on $K$-stable toric surfaces. To do this, he proves the regularity of the solutions of the Abreu equation where the function $S$ is constant. In \cite{MR3886172}, Chen, Li, and Sheng extend this result to the case of extremal K\"ahler metrics for toric surfaces which are the solutions of the Abreu equation with an affine linear function $S$.
The regularity of the solutions  for a  general function $S$ in any dimension remains widely open.
	
On the other hand, in a breakthrough work \cite{MR4301557, MR4301558}, Chen and Cheng prove a priori estimates for the solutions of cscK metrics (and more generally for bounded scalar curvature) on compact K\"ahler manifolds. A consequence of their results  is the existence of cscK metrics on compact toric manifolds of arbitrary dimension under a uniform stability condition. 
	
In general, the Abreu equation is quite complicated for a given function \( S \). The purpose of this paper is to focus on the first equation of the system version of the Abreu equation \eqref{Abreu-2-eq-v}, when the right-hand side function \( \varphi \) is specified. This particular question has been explored by Rubin in \cite{rubin2015monge} and Huang in \cite{huang2023guillemin} for the regular right-hand side function:
\begin{equation}
\label{num-den}
\varphi(x) = \frac{l_1(x)  \cdots  l_N(x)}{ H(x)}, \qquad \text{with} \quad H \in C^{\infty}(\overline{P}).
\end{equation}
Later, Huang and Shen, in \cite{huang2023monge}, extend this result to higher dimensions. Their approach is based on a purely PDE perspective.
	
The main goal of this paper is to study the case  where \( H \) is  only H\"older continuous. Due to some technicalities, we focus on the case of dimension \( n = 2 \). The higher-dimensional cases (\( n \geq 3 \)) may be explored in future research. 
	
Let $P \subset \mathbb{R}^n$ be the polytope defined by \eqref{polytope-set-def} and let $u: P \to \mathbb{R}$ be a strictly convex function that satisfies the real Monge-Amp\'ere equation and the Guillemin boundary condition, which is given by:
\begin{equation}
\label{MA-0000}
\begin{cases}
\det D^2 u=\dfrac{H}{l_1 \cdots l_N} \qquad & \text{in }P,\\
u - \displaystyle\sum_{i=1}^{N} l_i \log l_i \in C^{\infty}(\overline{P}).
\end{cases}
\end{equation}
	
The equation in \eqref{MA-0000} is singular along \( \partial P \). The main difficulty in solving \eqref{MA-0000} is to establish boundary regularity for
\(
u(x)-\sum_{i=1}^{N} l_i(x)\log l_i(x)
\)
by means of estimates that are uniform in \(u\). More precisely, our goal is to prove the following a priori bound:
\[
\left\|u-\sum_{i=1}^{N} l_i\log l_i\right\|_{C^{1,\alpha}(\overline{P})}\le C,
\]
where \(C\) depends only on the data of the problem—namely, the polytope \(P\), the right-hand side \(H\), and \(\|u\|_{C^0(\overline{P})}\)—and is otherwise independent of \(u\). We obtain this estimate in two steps. First, we derive a \(C^{1,\alpha/2}\) bound (see the proof of Theorem~\ref{Low-regularity} in Section~\ref{section--4}). Then, in Section~\ref{sec:boundary-regularity-completion}, we upgrade this regularity to the desired \(C^{1,\alpha}\) norm and also we prove that the second derivatives of
\(
u-\sum_{i=1}^{N} l_i\log l_i
\)
exhibit the optimal rate of blow-up near the boundary (see Theorem~\ref{main-theorem}).

First, it is known that the problem \eqref{MA-0000} is only well-defined for simple polytopes (see \cite[Proposition 2.1]{huang2023monge}). This issue only occurs in higher dimensions, specifically for \( n \geq 3 \), since all convex polygons are simple when in dimension \(  n=2 \).
Another important issue is that the boundary condition in equation \eqref{MA-0000} is unclear without explicitly prescribed boundary data. Explicit boundary values can only be specified at the vertices (corners) of \( P \). However, if the equation is restricted to the \((n-k)\)-face of \( P \), for \( k = 1, \cdots, n-1 \), the boundary values must satisfy a Monge-Ampère equation with Guillemin boundary conditions on this \((n-k)\)-face (see \cite[Lemma 2.2]{huang2023monge}). Therefore, the boundary condition for the problem \eqref{MA-0000} can be constructed by solving the Guillemin boundary value problem inductively, from edges to the \((n-1)\)-face of \( P \). In other words, the solvability of the problem \eqref{MA-0000} in dimension \( n \) depends on the solvability of these equations in dimensions \( 1, 2, \cdots, n-1 \). This discussion suggests that the Guillemin boundary value problem can be viewed as a Dirichlet boundary value problem with prescribed boundary data when the values of the solution at the vertices of \( P \) are specified.
	
From now on, we assume that \( H \) is a continuous function that satisfies  
\(
0 < a \leq H(x_1, x_2) \leq A
\),
for some constants \( a \) and \( A \).
The following result has been proven by Rubin in \cite{rubin2015monge} and by Huang in \cite{huang2023guillemin}. We will explain the details in Section~\ref{section--1}.

\begin{theorem}
\label{known-regular-solution}
Let $v_1,\cdots, v_N$ be the vertices of the polytope $P \subset \mathbb{R}^2$ and $a_1, \cdots, a_N$ be some given real numbers. If $H$ is smooth up to the boundary and satisfies a compatibility condition \eqref{eqeqeq-1}, to be specified later in Theorem \ref{Rubinthm1.1}; then there is a unique classical solution $u$ to the equation \eqref{MA-0000} with $u(v_i)=a_i$.
Moreover, the Guillemin boundary condition in \eqref{MA-0000} is satisfied uniformly, that is, $u- \sum_{i=1}^{N} l_i \log l_i$ is smooth up to the boundary, $\partial P$, in a uniform fashion with respect to $u$. In other words there is an a priori estimate.
\end{theorem}
	
In the first step, Rubin proves the existence issue of generalized solutions by adapting 
the method developed in \cite{MR437805}. He further proves that $u \in C^{0,\alpha} (\overline{P})$ for some $\alpha \in (0,1)$ and satisfies $ \|u\|_{C^{0,\alpha}(\overline{P})}\leq C $, for some positive constant  $C=C(P,\|u\|_{C^0 (\overline{P})}, a, A).$
	
A key observation in Rubin's work is that, if a solution \( u \) exists, its restriction to the edges \( e_i = \{ x \in \mathbb{R}^2 : l_i(x) = 0 \} \) satisfies a second-order ODE along the edge. This, along with the specified values of \( u \) at the vertices of the polygon, completely determines the restriction of \( u \) to the boundary of the polytope. Therefore, \( u \) can be obtained by solving the Monge–Ampère equation with the given Dirichlet boundary condition.
	
The main goal of this paper is to establish an analogous to Theorem \ref{known-regular-solution}, but under weaker regularity assumptions on the function \(H\). Our approach is based on deriving an a priori estimate for the solution \(u\).
We establish an a priori estimate comparing the Hessian of a solution $u$ with the model potential $\mathbf{u}_0$, where
\begin{equation}
\label{u_0-func}
\mathbf{u}_0(x)=\sum_{i=1}^{N} l_i(x)\log l_i(x).
\end{equation}
Additionally, we provide a Schauder-type estimate for
\begin{equation}
\label{reminder-eq}
v:= u-\mathbf{u}_0=u - \sum_{i=1}^{N} l_i(x)\log l_i(x),
\end{equation}
to characterize the behavior of $u$ near the edges and vertices of the polytope $P$, where the Monge–Ampère equation becomes degenerate. This estimate confirms that $D^2u$ is uniformly equivalent to $D^2\mathbf{u}_0$, ensuring uniform control of the solution's boundary regularity. We now present the main a priori estimate of this paper:

\begin{theorem}[Main Result]
\label{main-theorem}
Let \(u\) be a solution to \eqref{MA-0000}. There exists a positive constant \(C\), depending only on \(a\), \(A\), \(\|H\|_{C^{0,\alpha}(\overline{P})}\), \(\alpha\), and \(\|u\|_{C^0(\overline{P})}\), such that the following matrix inequality holds in the sense of symmetric matrices:
\begin{equation}
\label{eq-EQ-00}
C^{-1}D^2\mathbf{u}_0 \le D^2u \le C D^2\mathbf{u}_0
\qquad \text{in } P,
\end{equation}
where \(\mathbf{u}_0\) is the function defined in \eqref{u_0-func}. Furthermore, the function $v$ given by \eqref{reminder-eq} satisfies the regularity estimate
\begin{equation}
\label{eq-EQ-01}
\|v\|_{C^{1,\alpha}(\overline{P})} \le C, 
\end{equation}
as well as the pointwise bound
\begin{equation}
\label{eq-EQ-02}
\mathrm{dist}(x, \partial P)^{-\alpha+1} |D^2v(x)| \le C,
\end{equation}
where we assume $H$ is sufficiently smooth to guarantee that $D^2v$ exists pointwise a.e. for the validity of \eqref{eq-EQ-02}.
\end{theorem}

The uniform ellipticity-type estimate \eqref{eq-EQ-00} allows us to extend the results of Rubin and Huang to the case where the right-hand side \(H\) is only Hölder continuous. Consequently, our a priori estimate enables us to establish Theorem \ref{known-regular-solution} under these weaker assumptions.

\begin{theorem}[Weaker version of Main Result]
\label{Low-regularity}
Let \(v_1,\dots,v_N\) be the vertices of the polytope \(P\subset\mathbb{R}^2\), and let \(a_1,\dots,a_N\) be prescribed real numbers. Suppose that \(H\) is Hölder continuous up to the boundary with exponent \(\alpha\) and satisfies the compatibility condition \eqref{eqeqeq-1} (specified in Theorem~\ref{Rubinthm1.1}). 

Then there exists a unique classical solution \(u\) of \eqref{MA-0000} satisfying \(u(v_i)=a_i\). Moreover, \(u \in C^{2,\alpha}(P)\) and \(u\) satisfies the "weak Guillemin boundary condition":
\begin{equation}
\label{weak-Guillemin-bdr-c}
u-\sum_{i=1}^{N} l_i\log l_i \in C^{1,\alpha/2}(\overline{P}).
\end{equation}
\end{theorem}

In contrast of the weak Guillemin boundary estimate \eqref{weak-Guillemin-bdr-c} in Theorem~\ref{Low-regularity}, Theorem~\ref{main-theorem} yields the sharper boundary regularity
\(
v=u-\sum_{i=1}^{N} l_i\log l_i \in C^{1,\alpha}(\overline P)
\). For the sharpness, see Example~\ref{one-dim-ex}. Regarding regularity, the improvement from an exponent of $\alpha/2$ to $\alpha$ is not the most demanding aspect of our result. The primary challenge actually lies in establishing \eqref{weak-Guillemin-bdr-c}.

Before outlining the proofs of Theorems~\ref{main-theorem} and \ref{Low-regularity}, we place our results in the context of the existing literature. The Dirichlet problem for the Monge--Ampère equation was studied by Caffarelli, Nirenberg, and Spruck in \cite{zbMATH03964498}. They prove that, if $\Omega$ is a smooth bounded uniformly convex domain, then the homogeneous Dirichlet problem
\[
\begin{cases}
\det D^2u=f \quad & \text{in } \Omega, \\
u=0 \quad & \text{on } \partial \Omega,
\end{cases}
\]
admits a unique solution $u \in C^\infty(\overline{\Omega})$, provided that $0<c\le f \in C^\infty(\overline{\Omega})$ for some constant $c>0$. As part of their argument, they obtain \textit{a priori} estimates
\[
\|u\|_{C^{k+2,\alpha}(\overline{\Omega})}
\le C
=
C\bigl(k,\alpha,\Omega,\|f\|_{C^{k,\alpha}(\overline{\Omega})},\inf_{\overline{\Omega}} f\bigr),
\]
for every integer $k\ge 2$ and $\alpha\in(0,1)$. A central step in their proof is to bound $D^2u$ up to the boundary. Owing to the boundary condition and the uniform convexity of $\Omega$, the tangential second derivatives are easily controlled; the equation then reduces the problem to estimating the mixed second derivatives. To accomplish this, they differentiate the equation once and construct delicate barrier functions to control the mixed terms. In this way, they establish uniform ellipticity up to the boundary. Higher regularity then follows from the classical Schauder theory.

Later, Wang observes in \cite{zbMATH00870152} that Lipschitz continuity of $f$ is sufficient to bound the second derivatives up to the boundary. If, however, $f$ is only Hölder continuous, the boundary regularity problem becomes substantially more subtle. This question remained open for nearly two decades until the work of Trudinger and Wang \cite{zbMATH05578708}, who used localization techniques to control the geometry of boundary sections. By means of a sophisticated argument, they prove that $u\in C^{2,\alpha}(\overline{\Omega})$ whenever $0<c\le f\in C^{0,\alpha}(\overline{\Omega})$. In \cite{zbMATH06168117}, Savin further weakens the geometric assumptions by replacing uniform convexity of the boundary with a quadratic separation condition.

Two main difficulties arise in the study of \eqref{MA-0000}. First, the polygon $P$ is neither smooth nor uniformly convex. Consequently, unlike in the smooth uniformly convex setting, one does not have an \textit{a priori} gradient bound for $u$. Nevertheless, Rubin showed that the solution is Hölder continuous up to the boundary, with Hölder exponent depending on $P$; see, for instance, \cite{han2025globalc1alpharegularitymongeampere} and the references therein for recent developments. Second, the equation itself is singular near the boundary, and in particular the gradient of the solution blows up as one approaches $\partial P$.

The strategy in \cite{rubin2015monge}, \cite{huang2023guillemin}, and \cite{huang2023monge} is based on the linearization of the Monge--Ampère equation together with the construction of suitable barrier functions to control the second derivatives, modulo the explicit singular terms of the form $x\log x$, up to the boundary. In this respect, their approach parallels that of Caffarelli--Nirenberg--Spruck \cite{zbMATH03964498}. However, the toric setting presents additional difficulties that require new ideas. Most importantly, their method depends strongly on differentiability of the right-hand side. When the right-hand side is only Hölder continuous, one can no longer differentiate the equation, and it is natural to look instead to the ideas of Trudinger--Wang \cite{zbMATH05578708} and Savin \cite{zbMATH06168117}.

A natural strategy would therefore be to localize the equation near the boundary and study the corresponding boundary sections. In the present setting, however, this approach breaks down because the gradient of $u$ diverges at the boundary. To overcome this obstacle, we draw on ideas introduced by Donaldson in his study of the Abreu equation (\cite{donaldson2009constant}). Roughly speaking, we show that the geometry of the sections can be controlled uniformly near the boundary, independently of the distance to $\partial P$. As a consequence, we obtain the uniform ellipticity estimate
\[
C^{-1}D^2\mathbf{u}_0 \le D^2u \le C D^2\mathbf{u}_0
\qquad \text{in } P.
\]
Even with this estimate in hand, establishing the Guillemin boundary condition—namely, proving that $u-\mathbf{u}_0 \in C^{1,\alpha}(\overline{P})$—remains  extremely delicate.

The argument for the Guillemin boundary condition differs substantially between edges and vertices. Near an edge, we apply a partial Legendre transform and derive Schauder estimates for the resulting quasilinear equation in the Legendre plane. At a vertex, by contrast, the argument is more sophisticated; see Proposition~\ref{prop:vertex-sharp}. The main difficulty is that Huang’s barrier construction depends crucially on $C^2$ regularity of the right-hand side, and therefore does not directly apply in the Hölder setting. To overcome this issue, we adapt the method of Trudinger and Wang. More precisely, in Lemma~\ref{lem:product}, We show that Huang’s barrier argument continues to work for Hölder continuous functions of separated form,
\[
H(x_1,x_2)=F(x_1)G(x_2).
\]
We then approximate a general Hölder continuous function $H$ by a family of functions $H_\epsilon$ satisfying this separation property in a neighborhood of the vertex.

Once the weak Guillemin boundary condition \eqref{weak-Guillemin-bdr-c} has been established, we show that the associated complex Monge--Ampère equation has H\"older continuous right-hand side. Then, we apply the regularity theory for the complex Monge--Amp\`ere equation (c.f. \cite{TWWY}) to conclude that the complex potential is  $C^{2,\alpha}$. Using this we can control the blow-up rate of the second derivatives of $v$ and improve the H\"older exponent of $D v$; see Proposition~\ref{por:edge-C1a} and Proposition~\ref{por:vertex-C1a}.

The paper is organized as follows, and the proof of Theorem~\ref{main-theorem} proceeds in several stages.

In section~\ref{section--1}, we begin by collecting a number of preliminary results concerning convex solutions, sections of solutions to the Monge--Amp\`ere equation, and the geometry associated with the Guillemin potential. We also review results of Rubin and Huang that will be used throughout the proof. We first establish the interior estimate, Theorem~\ref{somehow-standard}. In dimension $n=2$, a classical theorem of Heinz yields a bound on the modulus of strict convexity in the interior. For the reader's convenience, we provide a self-contained proof; see also \cite{Chen2014}.

The main part of the paper is devoted to the boundary analysis, which is divided into several steps. In Section~\ref{section--2}, we derive an a priori estimate near an edge, away from the vertices (Proposition~\ref{proposition-edges}). In Section~\ref{section--3}, we develop the blow-up and localization arguments needed for the local analysis near a vertex. Section~\ref{section--5} is devoted to the remaining vertex analysis, where we obtain the corresponding a priori estimates using the approximation procedure of Trudinger and Wang \cite{zbMATH05578708}. In Section~\ref{section--4}, we prove the low-regularity boundary estimate (Theorem~\ref{Low-regularity}). We then establish the sharp boundary regularity in Section~\ref{sec:boundary-regularity-completion}, where the regularity exponent is improved to the $C^{1,\alpha}$ norm and finally complete the proof of Theorem~\ref{main-theorem}.

To preserve the flow of the main arguments, the technical estimates and auxiliary computations utilized in the boundary analysis are relegated to Appendices~\ref{app:Schauder} and~\ref{sec:calc}.

Our approach may also be adaptable to higher dimensions and to more general equations, including the complex Monge--Amp\`ere equation on toric manifolds.

\section{Preliminaries}
\label{section--1}
	
Let $P \subset \mathbb{R}^n$ be the polytope defined by \eqref{polytope-set-def} and let $u: P \to \mathbb{R}$ be a strictly convex function satisfying the real Monge-Amp\'ere equation and the Guillemin boundary condition in \eqref{MA-0000}. We assume that $H$ is continuous and satisfies
$ 0 < a \leq H(x_1, x_2) \leq A $,
for constants $a$ and $A$.
We focus on the case of dimension $n=2$. 
Let \( n_i = (n_i^x, n_i^y) \) be the unit inward-pointing normal to the edge \( e_i = \{ x \in \mathbb{R}^2 : l_i(x) = 0 \} \) of the polytope \( P \), with vertices \( v_i \) for \( 1 \leq i \leq N \), where \( v_i \) is the intersection of the edges \( e_{i-1} \) and \( e_i \). Note that we use the notations \( l_0 = l_N \), \( e_0 = e_N \), and \( v_0 = v_N \).
Set \( l_i(x) = n_i \cdot x - \lambda_i \), where \( n_i \) is the unit normal vector, and let \( T_i \) be the tangent vector to the face, which is the result of a \( -90^\circ \) rotation of \( n_i \). Any two vectors \( n_i \) and \( n_k \) define a matrix \( n_i^{\alpha} n_k^{\beta} \), the determinant of which gives the area of the parallelogram spanned by \( n_i \) and \( n_k \).
Now, let \( u: P \to \mathbb{R} \) be a function that satisfies equation \eqref{MA-0000}. Given arbitrary real numbers \( a_1, \cdots, a_N \), we seek to find a function \( u \) such that \( u(v_i) = a_i \) for each \( i \).
Rubin, in \cite{rubin2015monge}, proved the existence of a unique solution that is globally \( C^{0, \alpha} \) and \( C^{\infty} \) at all points up to the boundary, except at the vertices of the polytope. The H\"older exponent \( \alpha \) depends only on the number of vertices of the polytope.
	
\begin{theorem}[Rubin \cite{rubin2015monge}]
\label{Rubinthm1.1}
Let $P$ be a convex polytope in $\mathbb{R}^2$, and consider the problem \eqref{MA-0000},
where $0<H \in C^{\infty}(\overline{P})$.
		
\begin{enumerate}
\item[(i)]
If the equation \eqref{MA-0000} admits a solution $u$ that is convex in $P$ and satisfies the Guillemin boundary condition in \eqref{MA-0000}, then the given function $H(x)$ must satisfy
\begin{equation}
\label{eqeqeq-1}
H(v_i)= \left( n_{i-1}^x n_i^y - n_{i}^x n_{i-1}^y \right)^2 \prod_{j \neq i-1, i} l_j(v_i).
\end{equation}
\item[(ii)]
Conversely, assume that the given function $H(x)$ satisfies \eqref{eqeqeq-1}. Then there exists $\alpha > 0$ depending only on $N$ such that for each choice of values $\{a_i\}_{i=1}^N$, $a_i \in \mathbb{R}$, there is a unique solution $u \in C^{0,\alpha}(\overline{P})$ to the equation in \eqref{MA-0000}, satisfying the following boundary condition:
\begin{equation}
\label{eqeqeq-2}
u - \sum_{i=1}^{N} l_i \log l_i \in C^{\infty} \left(\overline{P} \setminus \{v_1, \cdots, v_N\} \right),
\end{equation}
and $u(v_i)=a_i$.
\end{enumerate}
\end{theorem}

The condition \eqref{eqeqeq-2} is nearly identical to the Guillemin boundary condition in \eqref{MA-0000}, except at the vertices of the polygon.
	
Rubin observed that if a solution \( u \) exists, then its restriction to the edge \( e_i \) satisfies the following second-order ODE along the edge \( e_i \), i.e.:
\begin{equation}
\label{eqeqeq-3}
\partial^2_{T_i} u = \frac{|n_i|^2}{\varphi_{n_i}}
\end{equation}
where \( \varphi \) is given by \eqref{num-den}. Combined with the assigned values of \( u \) at the vertices \( v_i \), this equation completely determines the restriction of \( u \) to the boundary \( \partial P \) of the polytope. Thus, \( u \) can be obtained by solving the Monge–Ampère equation with this Dirichlet condition to find a generalized solution.
More precisely, in \cite[Lemma 2.3]{rubin2015monge}, it is shown that if the \( i \)-th edge \( e_i = \{ l_i(x) = 0 \} \) is parametrized by \( x = v_i + tT_i \), then for the restriction function \( u \in C^2 \left((0,L) \right) \cap C^0([0,L]) \) that solves
\begin{equation}
\label{l-r-en}
u_{tt} = \frac{h(t)}{t(L-t)}
\end{equation}
where \( 0 < h \in C^{\infty}\left([0,L] \right) \), we have:
\begin{equation}
\label{l-r-en-1}
u(t) = h(0) t \log t + h(L) (L - t) \log (L - t) + v(t).
\end{equation}
Here \( v \) is smooth on \( [0, L] \). The function \( u(t) \) is determined by its boundary values \( u(0) \) and \( u(L) \).

It is worth mentioning that the boundary condition \eqref{eqeqeq-3} also appeared in the work of Le and Savin \cite{MR3125549}, in connection with the variational approach to Abreu’s equation \eqref{Abreu-eq-4th}. In that context, this relation arises from the Euler–Lagrange equation satisfied by a minimizer. However, in Rubin's context, it follows directly from the Guillemin boundary conditions through some relatively simple computations.
	
The remaining important issue, we must emphasize again, is regularity. The \( C^{0,\alpha} \) regularity on \( \overline{P} \) is established by constructing appropriate barrier functions. To address the lack of strict convexity of the domain, a lower barrier is derived from a power of the product of the \( l_i \)'s. The regularity and asymptotic expansion at the edges are modeled on the following problems:
\(
\det D^2u(x_1, x_2) = \frac{1}{x_2}
\)
near the interior of an edge \( \{x_2 = 0\} \), and  
\(
\det D^2u(x_1, x_2) = \frac{1}{x_1 x_2}
\)
near a vertex, i.e. in this model problem the corner is located at \( \mathbf{0} = (0, 0) \).
	
Rubin proved the existence of an Alexandrov (equivalently viscosity) solution \( u \in C^{0,\alpha}(\overline{P}) \) for the Monge–Ampère equation with suitable boundary conditions for some \( \alpha \in (0, 1) \). By applying the partial Legendre transform (as in \cite{MR2494810}), reducing the problem to a quasi-linear equation, and using perturbation arguments, he established the \( C^{\infty} \) regularity of \( u(x) - \sum_{i=1}^N l_i(x) \log l_i(x) \) up to the boundary of \( P \), except at the vertices.  
He demonstrated that solutions corresponding to positive, smooth boundary functions \( H \) exhibit the same expansion as solutions of the model problem:
\begin{equation}
\label{MA-model-face}
\det D^2 u(x,y)=\frac{1}{y}
\end{equation} 
in neighborhoods of boundary points away from the vertices. For \eqref{MA-model-face}, Rubin derived expansions in terms of \( y^n \) and \( y^n \log y \) for the partial Legendre transform of the solution (and thus for the solution itself) within a half-ball. These expansions follow from the fact that the solution satisfies a simple degenerate elliptic linear equation.
In other words, Rubin established the existence of solutions to the Monge–Ampère equation with the boundary condition:   
\[
u - \sum_{i=1}^N l_i \log l_i \in C^{\infty}\left(\overline{P} \setminus \{v_1, \cdots, v_N\} \right) \cap C^{0,\alpha} (\overline{P}).
\]  
The regularity of the solution at the vertices \( v_i \) remained unresolved until Huang provided a definitive answer in \cite{huang2023guillemin}, proving higher regularity up to the vertices:
$ u - \sum_{i=1}^{N} l_i \log l_i \in C^{\infty}(\overline{P})$.
	
This is a surprising phenomenon when compared to the recent results in \cite{zbMATH07653694, huang2023monge} for the uniformly elliptic case on cones and polygonal domains, highlighting the crucial role of the Guillemin boundary condition in \eqref{MA-0000}. Notably, the gradient of \( u \) blows up on \( \partial P \), ensuring the strict convexity of \( u \) and interior regularity in higher dimensions \( n \geq 3 \). This contrasts with the two-dimensional case due to Pogorelov's well-known counterexample (see for example \cite{zbMATH07568023}). Specifically, in \cite{huang2023monge}, Huang and Shen addressed the higher-dimensional case (for example, see the proof of Lemma 3.2 and equation (3.9) in their work), where the effect of the Guillemin boundary condition in \eqref{MA-0000} is evident. In this argument, they utilized Caffarelli's localization theorem \cite[Theorem 1]{caffarelli1990localization}. However, these results are somewhat well-known; for the sake of completeness, we provide a detailed proof in Theorem~\ref{somehow-standard}.
	
Besides the importance of the Guillemin boundary condition in \eqref{MA-0000}, the polytope structure of the problem is also crucial, as noted by Donaldson in \cite{donaldson2005interior} and also by Rubin in \cite{rubin2015monge}. Specifically, a naive formulation of the problem on strictly convex domains \( D \subset \mathbb{R}^2 \), with the boundary function \( d(x) := \mathrm{dist}(x, \partial D) \), given by:   
\[
\begin{cases}
\det D^2 u(x) = \mathrm{O}(d^{-1}), & \quad \text{in }D, \\
u - d \log d \in C^3 (\overline{D}),
\end{cases}
\]  
would admit no solution. The boundary asymptotics for \( u \) imply that \( \det D^2 u(x) \sim d^{-1} \log d^{-1} \) near \( \partial D \), which is not compatible (\cite{donaldson2005interior}). Therefore, the geometric features of the polytope are essential to the problem.

Since problem is local in nature, without loss of generality, let's consider the following degenerate/singular Monge-Ampère equation:
$$
\det D^2 u(x) = \frac{H(x)}{x_1 x_2}
$$
defined within a cube, say,
$$ Q_1 := \left\{x = (x_1, x_2) \in \mathbb{R}^2 \, : \, 0 < x_1, x_2 < 1 \right\}, $$
where the singularity on the right-hand side arises along the edges $x_1 = 0$, $x_2 = 0$, and specifically at the vertex $(0, 0)$. We show that, under mild assumptions—specifically, \( 0 < H \in C^{0,\alpha}(\overline{Q_1}) \) for some \( \alpha \in (0,1) \), and the explicit bounds   
\[
0 < a \leq H(x_1,x_2) \leq A < +\infty,
\]  
the weak Guillemin boundary condition  \[
u - x_1 \log x_1 - x_2 \log x_2 \in C^{1,\alpha/2}(\overline{Q_1}),
\]  
is satisfied.
	
Before we finish this section, we address the interior estimate. As mentioned before, in dimension two, a well-known theorem by Heinz (\cite{heinz1959differentialungleichung}) guarantees that we can bound the modulus of strict convexity in the interior. Indeed, for \( n = 2 \), the strict convexity of the solution arises directly from the positivity of the Monge-Ampère measure within the interior. However, in any dimension, the following interior estimate is a direct consequence of Caffarelli's "no line segment" theorem (\cite[Theorem 1]{caffarelli1990localization}). For the reader's convenience, we provide a detailed proof.
	
\begin{theorem}
\label{somehow-standard}
Let $u$ solves the problem \eqref{MA-0000}. Then, for any positive $\delta>0$, there exists a constant $C=C(P, a, A, \|H\|_{C^{0,\alpha}(\overline{P})}, \alpha, \|u\|_{{C^0} (\overline{P})},\delta)$ such that 
$$ \|u\|_{C^{2,\alpha}\left(P_{\delta}\right)} \leq C, $$
where \( P_{\delta} = \left\{ p \in P \, : \, \mathrm{dist}\left(p, \partial P \right) \geq \delta \right\} \) denotes the \( \delta \)-neighborhood of the polytope \( P \).
\end{theorem}
	
\begin{proof}
In this proof, we do not restrict \( n \) to be \( 2 \). For a given positive \( \delta \) and any solution \( u \), the modulus of strict convexity of \( u \) on \( P_{2\delta} \subset P_{\delta} \) is defined as follows:
\[
M(u,\delta) := \inf \left\{ \left(u - l_{p,u} \right)(x) \, : \, p \in P_{2\delta}, \, x \in \partial P_{\delta} \right\}.
\]
Here, \( l_{p, u} \) denotes the linearization of \( u \) at \( p \). By the celebrated theorem of Caffarelli, \cite[Theorem 1]{caffarelli1990localization}, we obtain the estimate:
\[
\|u\|_{C^{2,\alpha}\left(P_{2\delta}\right)} \leq C = C\left(\|u\|_{L^{\infty}\left(P_{\delta}\right)},\|F\|_{C^{0,\alpha}\left(P_{\delta}\right)}, \inf_{P_{\delta}} F,  M(u,\delta)\right)
\]
where 
\[
F= \varphi^{-1}= \frac{H}{l_1  \cdots  l_N}.
\] 
Clearly, \( \|F\|_{C^{0,\alpha}\left(P_{\delta}\right)} \) and \( \inf_{P_{\delta}} F \) depend only on \( \delta \), \( \alpha \) and \( \|H\|_{C^{0,\alpha}(\overline{P})} \). Hence, to establish the interior estimate, it suffices to show that \( M(u,\delta) \) is bounded from below.
		
Suppose, for contradiction, that the interior estimate fails for some positive \( \delta \). Then there exists a sequence of functions \( u_j \) and \( H_j \) such that 
\[
\left\|u_j \right\|_{L^\infty\left(P_{\delta}\right)}, \left\|H_j\right\|_{C^{0,\alpha}\left(P_{\delta}\right)} \leq C,
\]
but 
\[
M(u_j,\delta) \to 0, \qquad \text{as} \quad j \to + \infty.
\] 
This implies that there exist sequences of points \( p_j \in P_{2\delta} \) and \( x_j \in \partial P_{\delta} \) such that 
\[
\left(u_j - l_{p_j, u_j} \right)(x_j) \to 0.
\] 
Passing to a subsequence, we may assume that \( u_j \to u_{\infty} \) locally uniformly in \(P\), \( x_j \to x_{\infty} \in \partial P \), and \( p_j \to p_{\infty} \in P_{\delta} \). Since we have a uniform bound on \( u_j \), on the compact sets of \(P\), we also obtain a uniform gradient bound:
\[
\left\| \nabla u_j \right\|_{L^{\infty}(P_{\delta})} \leq C=C(\delta).
\] 
This ensures that the linearizations \( l_{p_j, u_j} \) converge locally uniformly to an affine function \( l_{\infty} \). Consequently,
\[
u_j - l_{p_j, u_j} \to u_{\infty} - l_{\infty} \qquad \text{locally uniformly}
\]
as well. In particular, \( u_{\infty} - l_{\infty} \) vanishes at \( p_{\infty} \) and \( x_{\infty} \). 
By Caffarelli's result, \cite[Theorem 1]{caffarelli1990localization}, \( u_{\infty} - l_{\infty} \) must vanish along a line segment that intersects the boundary of the polytope. However, one can easily obtain a contradiction with the boundary condition for \( u_{\infty} \), following a similar line of argument as in \cite{Chen2014} or \cite[Lemma 3.2]{huang2023monge}. In fact, the set \(\{ u_{\infty} = l_{\infty} \}\) contains a ray \( L \) with its endpoint on \( \partial P \). Without loss of generality, assume that this endpoint is \(\textbf{0}\) and that the edge of the polytope \( P \) containing \(\textbf{0}\) is given by \( x_1 = 0 \). Furthermore, assume that 
\[
L = \left\{ x = (x_1, x_2, \dots, x_n) \in \overline{P} \, : \, x_1 = t, \, x_2 = c_1 t, \, \dots, \, x_n = c_{n-1} t \right\}
\]
for some range of \( t \) and for some constants \( c_1, c_2, \dots, c_{n-1} \). Then, along \( L \), we have:
\[
\begin{aligned}  
& 0 = \left| u_{\infty}(t, c_1t, \dots, c_{n-1}t) - u_{\infty}(\textbf{0}) \right|  \\  
& \geq \left| u_{\infty}(t, c_1t, \dots, c_{n-1}t) - u_{\infty}(0, c_1t, \dots, c_{n-1}t) \right| - \left| u_{\infty}(0, c_1t, \dots, c_{n-1}t) - u_{\infty}(\textbf{0}) \right|,  \\  
&\geq \mathrm{O}\left(t \left|\log t \right|\right) - \mathrm{O}\left(t^2\right) > 0,
\end{aligned}
\]
for sufficiently small \( t \), due to the Guillemin boundary condition and the behavior of solutions along the edge of the polytope (see, e.g., \eqref{l-r-en} and \eqref{l-r-en-1}). This contradiction completes the proof.
\end{proof}

\section{Regularity near the Edge}
\label{section--2}
	
The goal of this section is to establish Theorem~\ref{main-theorem} away from the vertices. More precisely, we prove the following result.
	
\begin{proposition}
\label{proposition-edges}
For any positive \( \delta \), there exists a constant 
$
C = C(P, a, A, \|H\|_{C^{0,\alpha}(\overline{P})}, \delta, \alpha, \|u\|_{C^0 (\overline{P})})
$
such that
\[
C^{-1} D^2 \textbf{u}_0 \leq D^2 u \leq C D^2 \textbf{u}_0, \qquad \text{in }P \setminus N_{\delta},
\]  
for any solution $u$ to the problem \eqref{MA-0000}. Here, \( N_{\delta} \) denotes the union of balls of radius \( \delta \) centered at the vertices of \( P \).
\end{proposition}
	
We will provide a sketch of the proof before going into the details.
For a given \( \delta > 0 \), let us consider a point \( p = (p_1, p_2) \) that is close to an edge, say the \( x_1 \)-axis, and is at least a distance of \(\delta\) away from the other edges of the polytope \( P \). This implies that \( p \) lies within a small neighborhood of the \( x_1 \)-axis but is still far from the vertices and other edges of \( P \), ensuring that the local geometry around \( p \) can be analyzed without interference from corner singularities. The goal is to show that 
\[
D^2u(p_1,p_2) = \begin{bmatrix}
u_{11}(p_1,p_2) & u_{12}(p_1,p_2) \\
u_{21}(p_1,p_2) & u_{22}(p_1,p_2)
\end{bmatrix}
\sim 
\begin{bmatrix}
1 & 0 \\
0 & \dfrac{1}{p_2}
\end{bmatrix}
\]
uniformly (depending on $\delta$).
First, we observe that the tangential second derivative on the boundary has a lower and upper bound depending on $H$ and $\delta$. Therefore, we would like to use the equation to extend this estimate to a neighborhood of the edge. A natural way to do that is to localize the solution at the boundary. Unfortunately, we can not do that due to the blow up of the derivative at the boundary. In order to overcome this issue, we follow a clever idea introduced by Donaldson in \cite{donaldson2009constant}.
Instead of estimating the second tangential derivative, we try to show that the second normal derivative close to the boundary behaves uniformly like $d^{-1}$, where $d$ is the distance to the boundary. More precisely, we want to show that 
$$
u_{22}(p_1,p_2)p_2 \leq C=C\left(\delta, \|u\|_{\infty}, H \right).
$$ 
It is difficult to directly prove such an estimate. Instead, we prove a coarse version of this estimate (Theorem~\ref{thm1}).
Namely, we are bounding the quantity $D$ introduced by Donaldson in \cite{donaldson2009constant} (c.f. Definition \ref{DDD-12}).  The next step is to use Caffarelli's celebrated result (\cite{zbMATH04154994}) to obtain a $C^{2,\alpha}$ estimate. It is important to note that in order to bound $D$, we only use upper and lower bounds on $H$. However, to obtain an estimate on the second derivative, we use the H\"older continuity of $H.$ 
	
\begin{definition}[Normalization of solutions]
For a solution $u$ and a point $p = (t, s) \in P$,  we define
$$ u^*(x) := u(x) - l_p(x), $$
where 
$$ l_p(x) := u(p) + \nabla u(p) \cdot (x - p) $$
being the supporting hyperplane of $u$ at $p$. Alternatively, we can write
$$ u^*(x_1, x_2) := u(x_1, x_2) - u(t, s) - u_1(t, s)(x_1 - t) - u_2(t, s)(x_2 - s), $$
which represents the normalization of the solution $u$ at $p = (t, s)$. Note that $u^*$ attains its minimum value of $0$ at $p$, and the supporting hyperplane of $u^*$ at $p$ is the trivial zero function.
\end{definition}

\begin{definition}[Donaldson's notion of $D$ of $u$ at $(t, s)$] 
\label{DDD-12}
Suppose the point $p$ is close to an edge, say $x_2=0$.
For the solution $u$ and the normalization point $p = (t, s)$, we define the constant
$$ D \left(u; (t, s) \right) := \frac{u^*(t, 0)}{s} = \frac{u(t, 0) - u(t, s) + u_2(t, s) s}{s} $$
which depends on $u$ and $(t, s)$. This constant is referred to as the $D$ for the normalized solution $u^*$ at the normalization point $p = (t, s)$.
\end{definition}
	
To prove the preceding proposition, it suffices to focus on a neighborhood of the edges. Without loss of generality, we assume the edge is  
\[
\left\{(x_1,0) \, : \, |x_1| \leq 2 \right\}
\]
and consider the rectangle  
\[
R = \left\{(x_1, x_2) \, : \, |x_1| \leq 1, \, 0 < x_2 \leq 1 \right\}.
\]
Let $P = [-2, 2] \times (0, 2] \subset \mathbb{R}^2$, and consider a smooth, strictly convex function $u : P \to \mathbb{R}$, where $u(x) = u(x_1, x_2)$, satisfies the following Guillemin boundary condition:
\begin{equation}
\label{Guilliman}
u(x_1, x_2) - x_2 \log x_2 \in C^{\infty} (\overline{P}).
\end{equation}
Additionally, assume that $u$ satisfies the following Monge-Ampère equation:
\begin{equation}
\label{eq1}
\det D^2 u(x_1, x_2) = \frac{H(x_1, x_2)}{x_2}
\end{equation}
where $H : \overline{P} \to \mathbb{R}$ is continuous and satisfies $ 0 < a \leq H(x_1, x_2) \leq A $, for some constants $a$ and $A$. Furthermore, assume that 
$u(x_1, 0)$ is smooth and strictly convex for $-2 \leq x_1 \leq 2$; 
then, the equation \eqref{eq1} implies that 
$$ u_{11}(x_1, 0) = H(x_1, 0). $$
Considering $u = u(x_1, x_2)$ as described above, Rubin in \cite[Theorem 3.1]{rubin2015monge}, used Perron’s method to solve the Dirichlet problems and demonstrated that there exists a solution $u$ in $P$ which is H\"older continuous up to the boundary. More precisely, there exists $\alpha \in (0,1)$ such that
\begin{equation}
\label{first-bound}
\|u\|_{C^{0,\alpha}(\overline{P})} \leq C,
\end{equation}
where $C = C(a, A)$ is a universal constant, i.e. independent of all solutions $u$ of \eqref{eq1}. Specifically, this is also a uniform $L^{\infty}$ bound for the solutions.
	
The following theorem demonstrates that $D$ is universally bounded with respect to the normalization points $(t, s) \in R$, regardless of how close $s$ is to zero. This bound is uniform in $u$ as well.
	
\begin{theorem}
\label{thm1}
For any positive \( \delta \), there exists a universal positive constant \( \tilde{D} = \tilde{D}(a, A, \delta) \) such that  
\[
D \left( u; (t,s) \right) \leq \tilde{D},
\]  
for any \( (t,s) \in \overline{P} \setminus N_{\delta} \), where \( N_{\delta} \) denotes the union of balls of radius \( \delta \) centered at the vertices of the polytope \( P \).
\end{theorem}
In order to prove Theorem~\ref{thm1}, we follow the same strategy as in \cite{donaldson2009constant}.
Arguing by contradiction we assume that such an estimate fails. Therefore, we have a solution $u$ and a point $p$ such that $D(u;p)$ is very large. We may assume that $p$ is the "worst point" in the sense that it almost maximizes $D$ (cf. \eqref{worst point}). By rescaling $u$, we construct a solution for a rescaled equation that is defined on a large subset of upper half-plane (see Definition \ref{D-constant-def}). Moreover, the rescaled function is normalized at $(0,1)$ and is uniformly bounded on compact subsets of upper half-plane. Now assume that there exists a sequence of solutions and sequence of points close to the edge whose $D$ goes to $+\infty$. By rescaling this sequence, we obtain a subsequence of rescaled solutions that converges locally uniformly to a nowhere strictly convex function which is affine linear on a half-plane on the boundary. On the other hand the boundary condition gives a uniform strict convexity of rescaled functions on the boundary. This leads to a contradiction.
	
In order to prove Theorem~\ref{thm1}, we need the following lemma.
	
\begin{lemma}
\label{zero-D}
In a uniform fashion with respect to any $u$, and any $t$ such that $(t, s) \in R$, we have:
$$ s^{1-\gamma} D \left(u; (t, s) \right) \to 0 $$
as $s \to 0$, for any fixed $\gamma<\alpha$.
\end{lemma}
	
\begin{proof}
Without loss of generality, we assume that $t=0$ and  $u$ is normalized at $p=(0,1)$. Define $g(x_2) := u(0, x_2)$. By convexity, $g'(x_2) = u_2(0, x_2)$ is increasing. Since $g'(1) = 0$, we know that $|g'(x_2)|$ is decreasing on $(0,1]$. Therefore, for any $s \in (0,1)$, there exists $\xi \in (0, s)$, depending on $s$, such that
$$ g'(\xi) = \frac{g(s) - g(0)}{s}. $$
Thus,
\begin{equation}
\label{eq10}
|u_2(0, s)|=|g'(s)| \leq |g'(\xi)| = \frac{|g(s) - g(0)|}{s} \leq s^{-1+\alpha} \|u \|_{C^{0, \alpha}(\overline{P})} \leq C s^{-1+\alpha} 
\end{equation}
where we used the estimate \eqref{first-bound}.
By the definition of the $D$ of $u$ at $(0, s)$, we have:
$$ s^{1-\gamma} D \left(u; (0, s) \right) = s^{-\gamma} \left( u(0, 0) - u(0, s) \right) + u_{2}(0, s) s^{1-\gamma}. $$ 
Thus, by using the uniform bound \eqref{first-bound} and  \eqref{eq10}, we obtain:
$$ s^{1-\gamma} D \left(u; (0, s) \right) \leq s^{-\gamma} |u(0, 0) - u(0, s)| + |u_{2}(0, s)| s^{1-\gamma} \leq C s^{\alpha-\gamma} $$
which completes the proof.
\end{proof}

\subsection{Proof of Theorem \ref{thm1} and its consequences}
	
In this subsection, we will prove Theorem~\ref{thm1} and discuss its consequences.
The proof will be divided into several steps:
	
\medskip
\noindent {\bf STEP I (preparing the rescaling of solutions):} 
\noindent
Firstly, for any point $p = (t, s) \in R$, define
$$ \Lambda(p) = \Lambda(t, s) := \min\{1 - t, 1 + t\} $$
which represents the distance of $p$ to the lines $x_1= -1$ and $x_1=1$. 
Also, for any solution $u$, let
\begin{equation}
\label{eq1.5}
\mu := \max_{p \in R} \Lambda(p)^{\theta} D \left(u; p \right)
\end{equation}
where $\theta>\frac{2}{\alpha}-2$ is a fixed number, and $\alpha$ is the H\"older exponent in \eqref{first-bound}.
If $p_0 = (t_0, s_0)$ is the point where this maximum is attained over $R$, i.e.,
\begin{equation}\label{worst point} \Lambda(p)^{\theta} D \left(u; p \right) \leq \mu = \Lambda(p_0)^{\theta} D \left(u; p_0 \right) \qquad \forall p \in R;\end{equation}
then by the Lemma~\ref{zero-D}, one have:
$$ s_0^{-\alpha+1} \leq \frac{C}{\mu} \Lambda^{\theta}(p_0) $$
for a universal constant $C=C(a,A)$.
Moreover, one can easily check that for any fixed $L>0$, the following estimate holds:
\begin{equation}
\label{eq2}
D \left(u; (t,s) \right) \leq D \left(u; p_0 \right) \frac{1}{1-LC \mu^{-\frac{1}{\theta}}}
\end{equation}
for any $(t,s) \in R$ with $|t-t_{0}| \leq L \sqrt{s_0 D(u; p_0)}$, provided that $s_0$ is sufficiently small.
	
Secondly, we start by rescaling the normalization of the solution $u$, denoted $u^*$, as follows:
	
\begin{definition}[Definition of the scaling $\tilde{u}$]
\label{D-constant-def}
Let $u^*$ be the normalization of $u$ at $p_0 = (t_0, s_0)$. We then define the following rescaling:
$$ \tilde{u}(x_1, x_2) := \frac{1}{s_0 D \left(u; p_0 \right)} u^* \left(t_0 + \sqrt{s_0 D \left(u; p_0 \right)} x_1, s_0 x_2 \right). $$
For simplicity in notation, we will occasionally denote \(\lambda_0 := \sqrt{s_0 D(u; p_0)}\), which will serve as our scaling parameter.
\end{definition}

\begin{remark}
\label{lambda_0}
Our choice of \( \lambda_0 \) is inspired by Donaldson's approach (see, e.g., equation (27) in \cite{donaldson2009constant}). Specifically, we choose \( \lambda = \lambda_0 \) so that  
\begin{equation}
\label{R-R-12-100}
\lambda \int_{t_0-\lambda}^{t_0+\lambda} u_{11}(t,0) \, dt = s_0 D(u; p_0).
\end{equation}
Since \( u_{11}(t,0) \sim 1 \) away from the vertex, this suggests considering:  
\[
\lambda_0 = \sqrt{s_0 D(u; p_0)}.
\]  
\end{remark}
	
Notice that, according to Lemma~\ref{zero-D}, as $s_0$ approaches zero (i.e., as the point $p_0$ moves towards the edge of $R$), $\lambda_0$ approaches zero. 
	
The rescaled function $\tilde{u}$ is then defined, on the rectangles
$$ \tilde{R} = \left(-\frac{\Lambda(t_0)}{2\lambda_0}, \frac{\Lambda(t_0)}{2\lambda_0}\right) \times \left(0, \frac{1}{s_0}\right), $$
and it satisfies $\tilde{u}(0,0) = 1$. 
Furthermore, a simple computation gives the following PDE for $\tilde{u}$:
\begin{equation}
\label{eq2.5}
\det D^2 \tilde{u}(x_1, x_2) = \frac{s_0^2}{\lambda_0^2} \det D^2 u(t_0 + \lambda_0 x_1, s_0 x_2) = \frac{H(t_0 + \lambda_0 x_1, s_0 x_2)}{D \left(u; p_0 \right) x_2} \qquad \text{in }\tilde{R}.
\end{equation}

\medskip
\noindent {\bf STEP II (local boundedness of the rescaled solutions):}
\noindent
In this step, we aim to prove that the rescaled solutions $\tilde{u}$ are locally bounded on $\tilde{R}$ uniformly with respect to the scaling parameter $s_0$ (and also, obviously, with respect to $\lambda_0$, due to Lemma~\ref{zero-D}).

First, we make the following simple observations about the restrictions of $\tilde{u}$ to the $x_2$ and $x_1$-axes.

\begin{lemma}
\label{UPER-LOWER-bound}
We have: 
$$
\tilde{u}(0,x_2) 
\leq 
\begin{cases}
1, & \qquad 0 \leq x_2 \leq 1, \\
k x_{2} \log x_{2},  &  \qquad 1 \leq x_2,
\end{cases}
$$
where $k>1$ can be chosen close enough to $1$, provided that $s_0$ is sufficiently small.
\end{lemma}
	
\begin{lemma}
\label{UPER-LOWER-bound-2}
We have:
$$ \tilde{u}(x_1,0) \leq 1+C|x_1|+\frac{A}{2} x_1^2 \qquad \forall x_1  \in \left(-\frac{\Lambda(t_0)}{2 \lambda_0},\frac{\Lambda(t_0)}{2 \lambda_0} \right), $$
for some universal constant $C$.
\end{lemma}

\begin{proof}[Proof of Lemma \ref{UPER-LOWER-bound}]
Due to \eqref{eq2}, we can assume
\begin{equation}
\label{eq3}
D \left(\tilde{u}; (t, s) \right) \leq k,
\end{equation}
for any $(t, s)$ in compact subsets of the upper half-plane $\mathbb{H} = \{ (x_1, x_2) \in \mathbb{R}^2 \,:\, x_2 > 0 \}$, and besides we can suppose $k$ is as close to $1$ as we please.

For simplicity, define $g(x_2) := \tilde{u}(0, x_2)$. Then, \eqref{eq3} immediately leads to the following differential inequality:
$$ x_2 g'(x_2) - g(x_2) + 1 \leq k x_2 $$
where, by the construction of $\tilde{u}$, we know that $g(1) = g'(1) = 0$. Solving for $g$ results in the following bounds:
$$ g(x_2) \leq k x_2 \log x_2 + (1 - x_2), \qquad \forall x_2 \in [1, +\infty), $$
and
\begin{equation}
\label{eq4}
g(x_2) \geq k x_2 \log x_2 + (1 - x_2), \qquad \forall x_2 \in (0,1].
\end{equation}
On the other hand, since $g(0) = 1$, $g(1) = 0$, and more importantly, $g$ is convex (i.e., $g'$ is increasing), we have $g(x_2) \leq 1$ for any $0 \leq x_2 \leq 1$.
\end{proof}

\begin{proof}[Proof of Lemma \ref{UPER-LOWER-bound-2}]
For simplicity, let $h(x_1) = \tilde{u}(x_1, 0)$. Then, $h$ is a convex function satisfying:
$$ h''(x_1) = u_{11}(t_0 + \lambda_0 x_1, 0) = H(t_0 + \lambda_0 x_1, 0) \leq A, $$
for all $x_1 \in \left(-\frac{\Lambda(t_0)}{2\lambda_0}, \frac{\Lambda(t_0)}{2\lambda_0}\right)$. Also, $h(0)=1$. Therefore, we have:
$$ h(x_1) - 1 - h'(0) x_1 = h(x_1) - h(0) - h'(0) x_1 = \int_{0}^{x_1} \int_{0}^{s} h''(t) \, dt \, ds \leq \frac{A}{2} x_1^2 \qquad \forall x_1 \in \left(-\frac{\Lambda(t_0)}{2\lambda_0}, \frac{\Lambda(t_0)}{2\lambda_0}\right). $$
Noting that for $s_0$ small enough $\frac{\Lambda(t_0)}{2\lambda_0}>1$, and since
$ \int_{-1}^{1} {\tilde{u}}_{11}(t,0) \, dt $ 
is universally bounded, it easily follows that $h'(0)$ is universally bounded as well. This completes the proof.
\end{proof}
	
Now, using the convexity of $\tilde{u}$ and leveraging Lemmas \ref{UPER-LOWER-bound} and \ref{UPER-LOWER-bound-2}, we obtain a uniform bound on the rectangle
$$ \left(-\frac{\Lambda(t_0)}{2\lambda_0}, \frac{\Lambda(t_0)}{2\lambda_0} \right) \times \left(0, \frac{1}{s_0} \right). $$
Since, according to Lemma~\ref{zero-D}, this rectangle exhausts the upper half-plane $\mathbb{H}$ as $s_0$ approaches zero, we conclude that on every compact subset of $\mathbb{H}$, the rescaled solutions $\tilde{u}$ are uniformly bounded for any scaling parameter $s_0$.

\medskip
\noindent {\bf STEP III (final contradiction by getting the limit solution):}
\noindent
Now, we are ready to prove Theorem~\ref{thm1} by contradiction. 
	
Assuming that the positive constant $\tilde{D}$ does not exist, there is a sequence of solutions $u_n$ to \eqref{eq1}, with $H = H_n$, and the sequence $H_n$ satisfies (uniformly) the aforementioned assumptions on $H$, i.e.,
$$ 0 < a \leq H_n(x_1, x_2) \leq A, $$
and the $u_n$'s are uniformly bounded in $\overline{P}$. If
$$ \sup_{x \in R} D \left(u_n; x \right) \to +\infty, $$
then, $\lambda_0 \to +\infty$, and after passing to a subsequence, the rescaled functions $\tilde{u}_n$ will converge locally uniformly to a convex function $u_{\infty}$ on $\mathbb{H}$.
This follows easily from the standard Arzelà-Ascoli compactness theorem. Specifically, it is sufficient to show the equicontinuity of $\tilde{u}_n$'s in any compact subset of $\mathbb{H}$. This equicontinuity is guaranteed by the local uniform bound on $\tilde{u}_n$'s and the convexity of $\tilde{u}_n$'s, which immediately leads to local uniform Lipschitz estimates for $\tilde{u}_n$'s.
	
On the other hand, since the solutions $\tilde{u}_n$ can also be interpreted as viscosity solutions, it follows from equation \eqref{eq2.5} that $u_{\infty}$ is nowhere strictly convex on $\mathbb{H}$. More precisely, suppose the contrary; then there exists a ball $B_r(x)$ and an affine-linear function $l$ such that $u_{\infty} - l$ vanishes at $x$, but is strictly positive on the boundary of $B_r(x)$. The functions $\tilde{u}_n$ satisfy equation \eqref{eq2.5}, with the right-hand side tending to zero. Also, $\tilde{u}_n(x) - l_{x}(x) \to 0$ and $\tilde{u}_n - l_{x} \geq \delta > 0$, say, on $\partial B_r(x)$. Now, by applying the Alexandrov maximum principle, we obtain: 
$$ \left| \tilde{u}_n - l_{x} \right|^n \leq C(r) \int_{B_r(x)} \det D^2 \tilde{u}_n \to 0, $$
which implies that $u_{\infty} - l_{x}$ becomes zero at some point on $\partial B_r(x)$, leading to a contradiction and completing the proof.
	
\begin{lemma}
There exists a non-zero number $\eta$ such that $u_{\infty}$ vanishes on the half-plane $x_2 = \eta x_1+1$.
\end{lemma}

\begin{proof}
Let's define
$$ Z := \left\{ (x_1, x_2) \in \mathbb{H} \, : \, u_{\infty}(x_1, x_2) = 0 \right\}. $$
Given that the functions $\tilde{u}_n$'s are normalized such that their minimum occurs at the point $p_0 = (0, 1)$, it can be concluded that $Z$ is a convex set containing $p_0$.
		
Since the function \( u_{\infty} \) is nowhere strictly convex on \( \mathbb{H} \), \( Z \) cannot have extreme points. This implies that there must exist a line passing through \( p_0 \) whose intersection with \( \mathbb{H} \) is contained in \( Z \).
		
From equation \eqref{eq4}, we know that the functions \( \tilde{u}_n(0, x_2) \) are bounded below by \( x_2 \log x_2 - x_2 + 1 \) for \( 0 < x_2 < 1 \). Due to this lower bound, \( Z \) cannot contain the line segment \(\{x_1 = 0, x_2 > 0\}\). Therefore, \( Z \) must contain the line
$$ \Gamma_{\eta} := \{(x_1, x_2) \, : \, x_2 = \eta x_1 + 1, x_2 > 0\}, $$
for some \(\eta \in \mathbb{R}\).
		
To show that \(\eta \neq 0\), we proceed as follows:
		
We normalize the functions so that the \(x_1\) derivative of \(\tilde{u}_n\) differs by 1 at the points \((\pm 1, 0)\). As a result, there is a constant \(c\) such that \(\tilde{u}_n(c, 0) > 2\). Let's assume that \(Z\) contains the line \(\Gamma_0\), which means that for \(x_2 = 1\), we would have \(\tilde{u}_n(c, 1) \to 0\) as \(n \to +\infty\). Therefore, there exists a sequence \(c_n\) converging to 1 such that the \(x_2\) derivatives of \(\tilde{u}_n\) at \((c, c_n)\) approaches $0$. This, however, implies that the estimate provided by \eqref{eq3} would be violated for any \(k < 2\), which contradicts the assumption that \(D(u; (t, s))\) is uniformly bounded. Hence, \(\eta \neq 0\).
\end{proof}

\begin{proof}[Proof of Theorem \ref{thm1}]
Without lose of generality, we can assume that $u_{\infty}$ vanishes along the half-line $x_2 = -x_1 + 1$.
		
We claim that $\tilde{u}_n(x_1,0)$ converges locally uniformly to $-x_1 + 1$ on $(-\infty, 1]$.
		
To demonstrate this, consider a given $x_1 < 1$. We can find a sequence $d_n$ converging to $-x_1 + 1$ such that both $\tilde{u}_n$ and $\frac{\partial \tilde{u}_n}{\partial x_2}$, evaluated at $(x_1, d_n)$, converge to $0$ as $n \to +\infty$. Consequently, from equation \eqref{eq3}, we obtain
\begin{equation}
\label{linear-bound}
\tilde{u}_n (x_1,0) \leq k_n (-x_1+1) + \epsilon_n
\end{equation}
where $\epsilon_n \to 0$ and $k_n \to 1$.
		
Moreover, by construction, $\tilde{u}_n (0,0) = 1$, and due to convexity, $\tilde{u}_n (x_1,0)$ converges to $-x_1 + 1$ uniformly for each $x_1$ within any compact subset of $(-\infty, 1)$. Additionally, since the solutions are bounded in a neighborhood around the point $(1,0)$, convexity guarantees that this convergence is uniform up to $x_1 = 1$.
		
On the other hand, we have:
$$ \tilde{u}_{n}(x_1,0) = \tilde{u}_{n}(0,0) + \tilde{u}_{n}^{\prime}(0,0) x_1 + \int_{0}^{x_1} \int_{0}^{s} H_n(t_0 + \lambda_0 t, 0) \, dt ds. $$
After passing to a subsequence, we can assume that $H_n(t_0 + \lambda_0 t, 0) \to \alpha \geq a$ and $\tilde{u}_{n}^{\prime}(0,0) \to \beta$ as $n \to +\infty$. Therefore,
\begin{equation}
\label{quadratic-bound}
\tilde{u}_{n}(x_1,0) \to 1 + \beta x_1 + \frac{\alpha}{2} x_1^2.
\end{equation}
However, \eqref{quadratic-bound} clearly contradicts \eqref{linear-bound} for any $x_1$ close to $1$.
\end{proof}
	
\subsection{Proof of Proposition \ref{proposition-edges} } 
In this subsection, we complete the proof of Proposition~\ref{proposition-edges} using the upper bound on $D$. We begin by proving that the solution satisfies a uniform modulus of strict convexity (Theorem~\ref{G-O-Sec}). However, since we are in dimension $n=2$, Heinz’s theorem (\cite{heinz1959differentialungleichung}) can be directly applied to obtain such an estimate. Nonetheless, we believe that the argument in Theorem~\ref{G-O-Sec} can be adapted to achieve a similar result in higher dimensions. Therefore, we present its line of arguments here only for interested readers.
	
Let fix a positive number \( \epsilon \) and without loss of generality let \( p =(t,s)= (0, \epsilon) \). We assume that \( u \) is normalized at \( p = (0, \epsilon) \). The goal is to study the geometry of the sections near the edge.
	
For $\delta>0$ and the point $q=(q_1,q_2)$ away from vertices, define:
$$ \Delta(\delta, q):= \int_{-\delta+q_1}^{\delta+q_1} u_{11}(t, q_2) \, dt = u_1(\delta+q_1, q_2) - u_1(-\delta+q_1, q_2). $$
	
\begin{lemma}
\label{lem-lower}
There exists $c_1, c_2>0$ such that if \( \Delta(\delta, q) \leq c_1 \frac{q_2}{\delta} \),
then
$$ \det D^2 u \left(q_1, \frac{q_2}{2} \right) \leq \frac{c_2}{\delta^2}$$ for any point $q=(q_1,q_2)$ away from vertices with $q_2$ small enough.
\end{lemma}
	
\begin{proof}
Without loss of generality, we may assume that $q=(0,\epsilon)$ and \( u \) is normalized at \( (0,\epsilon) \). Define 
$$ \tilde{u}(x_1, x_2) = \frac{1}{\epsilon} u^*(\delta x_1, \epsilon x_2). $$ 
Then, by the assumption of lemma, we have: 
$$ \int_{-1}^{1} \tilde{u}_{11}(x, 1) \, dx = \epsilon^{-1} \delta^2 \int_{-1}^{1} u_{11}(\delta x, \epsilon) \, dx = \epsilon^{-1} \delta \int_{-\delta}^{\delta} u_{11}(t, \epsilon) \, dt \leq c_1. $$
This implies that $ \tilde{u}_{1}(1, 1) - \tilde{u}_{1}(-1, 1) \sim 1 $, and consequently we obtain \( |\tilde{u}_{1}(x_1,1) | \) is universally bounded, say, on $[-\frac{3}{4},\frac{3}{4}]$. On the other hand, by the estimates in Lemma~\ref{UPER-LOWER-bound}, we have \( \tilde{u}(0, x_2) \leq 1 \) on \( [0,1] \). Therefore, \( \tilde{u}_2(0, x_2) \) is universally bounded, say, on \( [\frac{1}{4},\frac{3}{4}] \). Now, by convexity, it is evident that \( |\nabla \tilde{u}| \) remains bounded within a small ball of radius \( \frac{1}{4} \) centered at \( (0, \frac{1}{2}) \). Consequently, according to the definition of Alexandrov solutions, the integral 
\(
\int_{B_{\frac{1}{4}}(0,\frac{1}{2})} \det D^2 \tilde{u}(x_1, x_2) 
\)
which equals to \( |\nabla \tilde{u}(B_{\frac{1}{4}}(0,\frac{1}{2}))| \) is also bounded. Note that on \( B_{\frac{1}{4}}(0,\frac{1}{2}) \), we have: 
$$ 2a \frac{\delta^2}{\epsilon} \leq \det D^2 \tilde{u}(x_1, x_2) = \delta^2 \det D^2 u (\delta x_1, \epsilon x_2) = \frac{\delta^2 H(\delta x_1, \epsilon x_2)}{\epsilon x_2} \leq 2A \frac{\delta^2}{\epsilon}. $$ 
This together with the boundedness of the integral implies that $\frac{\delta^2}{\epsilon}$ is universally bounded; and the proof is complete.
\end{proof}

\begin{corollary}
\label{C-lower-bound}
There exists a constant \( M = M(c_2, a) > 0 \), with \( c_2 \) defined in Lemma~\ref{lem-lower}, such that for any fixed \( a_2 > M \), there exists \( a_1 > 0 \) ensuring that 
\begin{equation}
\label{new-coarse-es}
\Delta(a_2 \sqrt{\epsilon}, \epsilon)  \geq a_1 \sqrt{\epsilon}
\end{equation}
for sufficiently small \( \epsilon > 0 \) depending on \( a_2 \).
\end{corollary}
\begin{proof}
We proceed by contradiction. Assume that the inequality \eqref{new-coarse-es} does not hold for any \( a_1 \) when \( \epsilon \) is sufficiently small.
Thus, for \( \delta = a_2 \sqrt{\epsilon} \), we obtain:
\[
\Delta(\delta, \epsilon)  < a_1 \sqrt{\epsilon} = a_1 \frac{a_2 \epsilon}{\delta} = a_1 \frac{a_2}{c_1} \frac{c_1 \epsilon}{\delta}
\]
for any \( a_1 > 0 \), with \( c_1 \) defined in Lemma~\ref{lem-lower}.
Choosing \( a_1 \) such that \( a_1 a_2 < c_1 \) leads to 
\(
\Delta(\delta, \epsilon) \leq \frac{c_1 \epsilon}{\delta}.
\)
By applying the result from Lemma~\ref{lem-lower}, this implies:
$$ \frac{a}{\epsilon} \leq \frac{H \left(0, \frac{\epsilon}{2} \right)}{\epsilon}=\det D^2u \left(0, \frac{\epsilon}{2} \right) \leq \frac{c_2}{\delta^2}=\frac{c_2}{a_2^2 \epsilon} $$
or, $a a_2^2 \leq c_2$.
Now, by selecting \( a_2 \) such that \( a a_2^2 > c_2 \), we arrive at a contradiction.
\end{proof}

\begin{theorem}
\label{G-O-Sec}
There exists a universal constant \( C = C(a, A) \) such that 
$$ \{ u < h \} \Subset \left\{ x_2 > \frac{\epsilon}{2} \right\}, $$
for any \( h < C \epsilon \).
\end{theorem}
	
\begin{proof}
\medskip
\noindent {\bf STEP I:}
\noindent 
Let \( u \) be normalized at \( (0, \epsilon) \), and assume that \( h \) is the smallest positive number such that \( \{ u \leq h \epsilon \} \) intersects the line \( y = \frac{\epsilon}{2} \) at the point \( (x, \frac{\epsilon}{2}) \). We claim that there exists a uniform positive constant \( c \) such that \( |x| \leq c \sqrt{\epsilon} \). Without loss of generality, we assume that \( x \geq 0 \).
We proceed by contradiction. Assume that for any \( a_2 > M \) (where \( M \) is as defined in Corollary~\ref{C-lower-bound}), we have  
\begin{equation}
\label{x>>!}
x \geq a_2 \sqrt{\epsilon}.
\end{equation}  
Since we have  
\[
u\left(0, \frac{\epsilon}{2}\right) - h \epsilon = u\left(0, \frac{\epsilon}{2}\right) - u\left(x, \frac{\epsilon}{2}\right) = \int_0^x \int_0^t u''\left(s, \frac{\epsilon}{2}\right) \, ds \, dt,
\]  
applying Corollary~\ref{C-lower-bound}, we obtain:  
\begin{align*}  
u\left(0, \frac{\epsilon}{2}\right) - h \epsilon &\geq  \int_{2M\sqrt{\epsilon}}^x \int_0^t u''\left(s, \frac{\epsilon}{2}\right) \, ds \, dt \\  
&= \int_{2M\sqrt{\epsilon}}^x \left( u'\left(t, \frac{\epsilon}{2}\right) - u'\left(0, \frac{\epsilon}{2}\right) \right) \, dt \\  
&\geq \left(x - 2M\sqrt{\epsilon} \right) \Delta\left(M \sqrt{\epsilon}, \frac{\epsilon}{2}\right) \\  
&\geq a_1 \sqrt{\epsilon} \left(x - 2M\sqrt{\epsilon} \right).
\end{align*}  
On the other hand, due to the convexity of \( u \) and Theorem~\ref{thm1}, we know that  
\[
u\left(0, \frac{\epsilon}{2}\right) \leq u(0, 0) = \epsilon D\left(u; (0, \epsilon)\right) \leq \epsilon \tilde{D}.
\]  
Combining this with the previous estimate, we obtain  
\[
a_1 \sqrt{\epsilon} x \leq \epsilon \tilde{D} - \epsilon h + 2Ma_1\epsilon \leq \epsilon \left( \tilde{D} + 2Ma_1 \right).
\]  
This, together with \eqref{x>>!}, implies that  
\[
a_2 \leq \frac{2Ma_1 + \tilde{D}}{a_1},
\]  
which contradicts the assumption that \( a_2 \) can be arbitrarily larger than \( M \).
		
Now, we proceed with the remainder of the proof using the following rescaling:  
\[
\tilde{u}(x_1, x_2) = \frac{1}{\epsilon} u^*(\sqrt{\epsilon} x_1, \epsilon x_2),
\]  
applied to the solution \( u \). Note that the section \( \{\tilde{u} \leq h \} \) intersects the line \( y = \frac{1}{2} \) at the point \( (\tilde{x}, \frac{1}{2}) \). We have already established that \( |\tilde{x}| \leq C \) uniformly.
		
\medskip
\noindent {\bf STEP II:}
\noindent    
Using the uniform bound established in STEP I, we proceed with the proof of the theorem. We argue by contradiction. Suppose there exists a sequence \(\{u_n\}\), \(\{\epsilon_n\}\), and \(\{h_n\}\) such that the set \(\{u_n < h_n \epsilon_n\}\) intersects the line \(y = \frac{\epsilon_n}{2}\) at the point \((x_n, \frac{\epsilon_n}{2})\), where \(h_n, \epsilon_n \to 0\).  
Applying the scaling from STEP I, we obtain a sequence of rescaled functions \(\tilde{u}_n\) and corresponding points \((\tilde{x}_n,\frac{1}{2})\).  
First, observe that \(\tilde{u}_n(t, \frac{1}{2})\) attains a minimum at \( t = \tilde{x}_n \), satisfying the following conditions:  
\begin{itemize}
\item \( \tilde{u}_n(\tilde{x}_n, \frac{1}{2}) = h_n \);
\item \(\tilde{u}_n'(\tilde{x}_n, \frac{1}{2}) = 0 \);
\item \( \tilde{u}_n(t,\frac{1}{2})\geq h_n\).
\end{itemize}
After passing to a subsequence, we may assume that \(\tilde{x}_n \to \tilde{x}_{\infty}\). Moreover, by the definition of \(x_n\), we obtain:  
\[
\tilde{u}_n\left(x_n,\frac{1}{2}\right) = h_n \to 0.
\]
On the other hand, since  
\begin{equation}
\label{eq5}
\tilde{u}_n(0,0) = \epsilon_n D \left(u_n; (0,\epsilon) \right) \leq \epsilon_n \tilde{D} \to 0,
\end{equation}
a similar argument to those in Lemmas \ref{UPER-LOWER-bound} and \ref{UPER-LOWER-bound-2} yields a uniform bound for \(\tilde{u}_n\) on the set \(\left[-1,1\right] \times \left[\frac{1}{4}, 1\right]\), independent of \(\epsilon_n\).
		
Additionally, since \(\tilde{u}\) is normalized at \((0,1)\), and  
\[
\det D^2 \tilde{u}(x_1, x_2) = \epsilon \det D^2 u(\sqrt{\epsilon} x_1, \epsilon x_2) = \frac{H(\sqrt{\epsilon} x_1, \epsilon x_2)}{x_2},
\]  
we obtain the inequality  
$$ 4a \leq \det D^2 \tilde{u}(x_1, x_2) \leq 4A \qquad \text{on} \quad \left(\frac{-1}{\sqrt{\epsilon}}, \frac{1}{\sqrt{\epsilon}}\right) \times \left(\frac{1}{4}, \frac{1}{\epsilon}\right). $$  
Since \(\tilde{u}_n\) converges locally uniformly to \(u_{\infty}\) on \([-1,1] \times [\frac{1}{2},1]\) as \(n \to +\infty\), we conclude that  
\[
4a \leq \det D^2 u_{\infty} \leq 4A \qquad \text{in }(-1,1) \times \left(\frac{1}{2},1\right).
\]  
However, \(u_{\infty}\) is identically zero on the line segment connecting the points \((0,1)\) and \(({\tilde{x}}_{\infty}, \frac{1}{2})\), which contradicts the strict convexity of \(u_{\infty}\), as established by a well-known theorem in dimension \(n=2\) by Heinz (\cite{heinz1959differentialungleichung}).  
This contradiction completes the proof of the claim.
\end{proof}
	
Now, we can apply Caffarelli's celebrated regularity theory (\cite{zbMATH04154994}) to establish Proposition~\ref{proposition-edges}.

\begin{proof}[Proof of Proposition \ref{proposition-edges}]
Assume that \(u\) is normalized at the point \((t,\epsilon)\) and define
\[
\tilde{u}(x_1,x_2)
=
\frac{1}{\epsilon}\,
u^*(t+\sqrt{\epsilon}\,x_1,\epsilon x_2).
\]
Then \(\tilde u\) is defined on the ball \(B_{1/2}((0,1))\). Moreover,
\(\tilde u\) is uniformly bounded and strictly convex. The strict convexity
follows from Heinz's theorem \cite{heinz1959differentialungleichung}. Hence
we may apply Caffarelli's interior regularity theory
(\cite{zbMATH04154994}; see also \cite[Lemma 4.41]{zbMATH06669881}).

Under this scaling, we have
\[
\det D^2\tilde u(x_1,x_2) = \frac{H(t+\sqrt{\epsilon}x_1,\epsilon x_2)}{x_1x_2} \in C^{0,\alpha}\left(B_{1/2}((0,1))\right).
\]
Applying Caffarelli’s estimate to the Hessian at the point $(0,1)$,
\[
D^2\tilde u(0,1) = 
\begin{bmatrix}
u_{11}(t,\epsilon) & \sqrt{\epsilon}\,u_{12}(t,\epsilon)\\
\sqrt{\epsilon}\,u_{12}(t,\epsilon) & \epsilon\,u_{22}(t,\epsilon)
\end{bmatrix},
\]
we obtain the following bounds:
\[
C^{-1} I \leq D^2 \tilde u(0,1) \leq C I,
\]
where $I$ denotes the identity matrix, and the constant $C = C\big(P, a, A, \|H\|_{C^{0,\alpha}(\overline P)}, \alpha, \|u\|_{C^0(\overline P)}\big)$ is independent of $\epsilon$.

Returning to the original variables, where 
\[
D^2u(t,\epsilon) = \begin{bmatrix} u_{11}(t,\epsilon) & u_{12}(t,\epsilon) \\ u_{12}(t,\epsilon) & u_{22}(t,\epsilon) \end{bmatrix},
\]
we deduce that
\[
C^{-1} \begin{bmatrix} 1 & 0 \\ 0 & \epsilon^{-1} \end{bmatrix} \leq D^2 u(t,\epsilon) \leq C \begin{bmatrix} 1 & 0 \\ 0 & \epsilon^{-1} \end{bmatrix}.
\]
Finally, noting that $D^2\mathbf{u}_0(t,\epsilon) \sim \operatorname{diag}(1, \epsilon^{-1})$ completes the proof.
\end{proof}

\begin{remark}
We do not need to have a priori control on the modulus of strict convexity in the preceding argument, as we are in dimension $n=2$. However, in higher dimensions, one needs to have such an a priori estimate in order to apply Caffarelli's result (\cite{zbMATH04154994}).
\end{remark}
	
\section{Regularity near the Vertex I}
\label{section--3}
	
In this section we apply similar ideas to extend the regularity of solutions to the vertices. This allows us to complete the proof of Theorem~\ref{main-theorem}. The main goal of this section is proving the follwing.

\begin{proposition}
\label{proposition-vertices}
Let \( u \) be a solution to \eqref{MA-0000}. There exists a constant 
$
C = C(P, a, A, \|H\|_{C^{0,\alpha}(\overline{P})},  \alpha, \|u\|_{C^0(\overline{P})})
$
such that
\[
C^{-1} D^2 \textbf{u}_0 \leq D^2 u \leq C D^2 \textbf{u}_0, \qquad \text{in }N,
\]  
where \( N \) is sufficiently small neighborhood of the vertices of \( P \).
\end{proposition}
	
Our strategy is similar to the one outlined in Section \ref{section--2}. Roughly speaking, We would like to prove Theorem~\ref{thm1} with a constant that is independent of the distance to the vertices. Without loss of generality, we may assume that we are close to the vertex $(0,0)$ where the polytope $P$ around the vertex $(0,0)$ is locally given by $x_1,x_2>0$. The first step is to estimate $D$ on the diagonal (Theorem~\ref{thm1V}). Once, we have done that, we conclude that $D$ is uniformly bounded for any point $(p_1,p_2)$ such that $p_1 \sim p_2 \sim 0$. The next step is to show that $D$ is bounded where $p_1 \gg p_2$ (Theorem~\ref{thm1V2}).
	
\subsection{Regularity along the diagonal near the vertices}
	
Define $Q:=Q_2=\{x = (x_1, x_2) \in \mathbb{R}^2 \, : \, 0 < x_1, x_2 < 2 \}$ and assume that $u:Q \to  \mathbb{R}$, $u=u(x_1,x_2)$, is a smooth strictly convex function satisfying the Guillemin boundary condition $$ u(x_{1},x_{2})-x_{1} \log x_{1} -x_{2}\log x_{2} \in C^{\infty}(\overline{Q}). $$
Also, assume that $u$ satisfies again the following Monge-Amp\'ere equation 
\begin{equation}
\label{eq1V}
\det D^2u \left(x_1,x_2 \right) = \frac{H(x_1,x_2)}{x_1 x_{2}}
\end{equation}
where $H :\overline{Q} \to \mathbb{R}$ is continuous and satisfying 
$ 0 < a \leq H(x_{1},x_{2}) \leq A$,
for some constants $a,A$.
Then, the equation \eqref{eq1V} implies that:
\begin{equation}
\label{restriction-x_1-x_2}
u_{11}(x_{1},0)=\frac{H(x_{1},0)}{x_1} \quad \text{and} \quad u_{22}(0,x_{2})=\frac{H(0,x_{2})}{x_2}.
\end{equation}
Moreover, for  
$$ v(x_1,x_2):=u(x_1,x_2)-x_1 \log x_1 -x_2 \log x_2 $$  
if necessary, by subtracting an affine function, one can confirm that \( v \) satisfies the following Dirichlet boundary value problem:  
\begin{equation}
\label{v-equation-1}
\begin{cases}
\left( x_1 v_{11} +1 \right) \left( x_2 v_{22} + 1 \right) - x_1x_2v_{12}^2 = H(x_1,x_2), \quad & \text{in }Q, \\  
v(0,0) = |\nabla v(0,0)| = 0, \quad  
v_{11}(x_1,0) = \dfrac{H(x_1,0)-1}{x_1}, \quad  
v_{22}(0,x_2) = \dfrac{H(0,x_2)-1}{x_2}.
\end{cases}
\end{equation}
As Rubin demonstrated in \cite{rubin2015monge}, equation \eqref{v-equation-1} admits a solution \( v \in C^{\infty}\left(\overline{Q}\setminus\{\textbf{0}\}\right) \cap C^{0,\theta}(\overline{Q}) \) for some \( \theta \in (0,1) \). Building on this, Huang in \cite{huang2023guillemin} constructed barrier functions to establish that \( v \in C^{0,1}(\overline{Q}) \). However, the proof in \cite[Lemma 2.4]{huang2023guillemin} relies on the assumption that \( H \in C^1(\overline{Q_1}) \). We adapt this approach with a slight modification to the barrier functions, ensuring that the proof remains valid even for functions \( H \) that are only  Hölder continuous. This will be the subject of Lemma~\ref{Barrier-modified-Huang}.

Before proceeding with our main objectives, we will first present Definition \ref{Donaldson-E-constant} by Donaldson (\cite{donaldson2009constant}). The result of Lemma~\ref{Barrier-modified-Huang} (which shows the Lipschitz continuity of solutions to \eqref{v-equation-1} at the boundary), given the Hölder regularity of the right-hand side function \( H \), will immediately ensure the uniform boundedness of the \( E \) for the solutions \( u \) of \eqref{eq1V} with the same right-hand side function \(H\).
	
Near the vertex point \((0, 0)\), for $\epsilon>0$, we normalize our solutions at the point \((\epsilon, \epsilon)\), i.e.
$$ 
u^*(x_{1},x_{2}) = u(x_{1},x_{2}) - u(\epsilon,\epsilon) - u_{1}(\epsilon,\epsilon)(x_{1} - \epsilon) - u_{2}(\epsilon,\epsilon)(x_{2} - \epsilon).
$$
	
\begin{definition}[Donaldson's notion of $E$ of $u$ at $(\epsilon, \epsilon)$]
\label{Donaldson-E-constant} 
Let $\epsilon$ be a positive number and assume that $u$ is normalized at $p=(\epsilon,\epsilon)$. For the solution $u$, define: 
$$ E\left(u; (\epsilon,\epsilon) \right):=E(u;\epsilon):= \frac{u^*(2\epsilon,0)+u^*(0,2\epsilon)}{\epsilon}=\frac{u(2\epsilon,0)+u(0,2\epsilon)-2u(\epsilon,\epsilon)}{\epsilon}.$$
\end{definition}

First of all, we have the following theorem:
	
\begin{theorem}
\label{thm1V}
There exists a universal positive constant $\tilde{E}=\tilde{E}(a,A)$ such that 
$$ E(u;\epsilon) \leq \tilde{E}, $$
for any small positive $\epsilon$.
\end{theorem}
	
This theorem is a simple corollary of the following lemma.
	
\begin{lemma}
\label{Barrier-modified-Huang}
Suppose that \( H \in C^{0,\alpha}(\overline{Q}) \) and \( v \) is a solution of the Dirichlet problem \eqref{v-equation-1}. Then, the following estimates hold:  
\[
|v(x_1,x_2)-v(x_1,0)| \leq Cx_2,
\]
\[
|v(x_1,x_2)-v(0,x_2)| \leq Cx_1,
\]
where the constant \( C \) depends on \( \|H\|_{C^{0,\alpha}(\overline{Q})} \), \( \|v\|_{C^0(\overline{Q})} \), and \( \|Dv\|_{C^0\left(\overline{Q\setminus Q_{\delta}}\right)} \), while \( \delta \) depends on \( a \), \( \alpha \), and \( \|H\|_{C^{0,\alpha}(\overline{Q})} \).
\end{lemma}
	
\begin{proof}
We follow the proof of Huang \cite[Lemma 2.4]{huang2023guillemin}, with a slight modification of the barrier functions. The proof is very similar to the one in \cite[Lemma 2.4]{huang2023guillemin}. For the reader convenience, we provide the details.
Let
$$
U_+(x_1,x_2) = A_+ x_2 - \frac{B_+}{\alpha(1+\alpha)} x_2^{1+\alpha} + x_1 \log x_1 + x_2 \log x_2 + v(x_1,0),
$$
where the two positive constants $A_+$ and $B_+$ will be determined later. In the following, we select appropriate values for $r_+$, $A_+$, and $B_+$ to ensure that $U_+(x_1,x_2)$ serves as an upper barrier in $Q_{r_+}$.
		
Then, in $Q_{r_+}$, we have:
$$
\det D^2 U_+ =
\left(v_{x_1 x_1} (x_1,0) + \dfrac{1}{x_1}
\right) \left( \dfrac{1}{x_2} - B_+x_2^{\alpha-1} \right)
= \frac{H(x_1,0)}{x_1} \left( \frac{1}{x_2} - B_+x_2^{\alpha-1} \right).
$$
Then, for \( u = v + x_1 \log x_1 + x_2 \log x_2 \), we obtain:
\begin{align} \nonumber
\det D^2 u - \det D^2 U_+ 
& = \frac{x_2^{\alpha-1}}{x_1} \left( B_+ H(x_1,0) + \frac{H(x_1,x_2) - H(x_1,0)}{x_2^{\alpha}} \right) \\ \label{diff-bar}
& 
\geq \frac{x_2^{\alpha-1}}{x_1} \left( a B_+  - \|H\|_{C^{0,\alpha}(\overline{Q})} \right) > 0, \qquad \text{in }Q_{r_+}.
\end{align}
In deriving the last inequality in \eqref{diff-bar}, we fix: 
\[ B_+ = \frac{2}{a} \|H\|_{C^{0,\alpha}(\overline{Q})} \qquad \text{and} \qquad r_+ = \max \left\{ \left( \frac{a}{4 \left(\|H\|_{C^{0,\alpha}(\overline{Q})} + 1 \right)} \right)^{\frac{1}{\alpha}}, \left( \frac{\alpha(\alpha+1)}{2} \right)^{\frac{1}{\alpha}} \right\}. \]
With these choices, we have:
$$
U_{+,11}(x_1,x_2) = \frac{H(x_1,0)}{x_1}, \qquad U_{+,22}(x_1,x_2) = \frac{1}{x_2} - B_+x_2^{\alpha-1} > 0, \qquad U_{+,12}(x_1,x_2) = 0,
$$
which also implies that \( U_+ \) is convex in \( Q_{r_+} \).
		
Next, we compare the boundary values of \( u \) and \( U_+ \). Let \( A_+ \geq B_+ r_+ \). Then, we have:
$$
A_+ x_2 - \frac{B_+}{\alpha(1+\alpha)} x_2^{1+\alpha} \geq \frac{A_+}{2} x_2 \qquad 0 \leq x_2 \leq r_+.
$$
\begin{itemize}
\item On \( 0 \leq x_1 \leq r_+ \), with \( x_2 = 0 \), we have \( u = U_+ \).
			
\item On \( x_1 = 0 \), with \( 0 \leq x_2 \leq r_+ \), we have:
$$
U_+ - u \geq \frac{A_+}{2} x_2 - v(0, x_2) > 0
$$
for \( A_+ \geq 2 \max_{0 \leq x_2 \leq 1} |v_{x_2}(0,x_2)| \).
			
\item On \( x_1 = r_+ \), with \( 0 \leq x_2 \leq r_+ \), we have:
$$
U_+ - u \geq \frac{A_+}{2} x_2 - \left(v(r_+,x_2) - v(r_+,0) \right) > 0
$$
for \( A_+ \geq 2 \max_{0 \leq x_2 \leq 1} |v_{x_2}(r_+,x_2)| \).
			
\item On \( 0 \leq x_1 \leq r_+ \), with \( x_2 = r_+ \), we have
$$
U_+ - u \geq \frac{A_+}{2} r_+ + v(x_1,0) - v(x_1,r_+) > 0
$$
for
$$
A_+ \geq 4 \frac{\max_{Q} |v(x_1,x_2)|}{r_+}.
$$
\end{itemize}
The argument above shows that \( U_+(x_1,x_2) \) is an upper barrier for \( u \) on \( \partial Q_{r_+} \), provided \( A_+ \) is sufficiently large. From the standard maximum principle, it follows that:
$$
u(x_1,x_2) \leq U_+(x_1,x_2), \qquad \forall (x_1,x_2) \in Q_{r_+}.
$$
A similar argument also gives:
$$
u(x_1,x_2) \geq U_-(x_1,x_2), \qquad \forall (x_1,x_2) \in Q_{r_-}
$$
where
$$
U_-(x_1,x_2) = -A_- x_2 + \frac{B_-}{\alpha(1+\alpha)} x_2^{1+\alpha} + x_1 \log x_1 + x_2 \log x_2 + v(x_1,0).
$$
This proves that \( |v(x_1,x_2) - v(x_1,0)| \leq C x_2 \). By choosing similar barrier functions
$$
\hat{U}_\pm(x_1,x_2) = \pm \hat{A}_\pm x_1 \mp \frac{\hat{B}_\pm}{\alpha(1+\alpha)} x_1^{1+\alpha} + x_1 \log x_1 + x_2 \log x_2 + v(0,x_2)
$$
one also obtains the second inequality of the present lemma.
\end{proof}

Now, we proceed similarly to the proof of Proposition~\ref{proposition-edges} as follows:  
	
Suppose that \( u \) is normalized at the point \( (\epsilon, \epsilon) \), and consider the following scaling:  
\[
\tilde{u}(x_1, x_2) = \frac{1}{\epsilon} u^*\left(\epsilon x_1, \epsilon x_2\right).
\]
Then, the function \( \tilde{u} \) is defined on a ball of radius \( \frac{1}{2} \) centered at \( (0,1) \). It is important to note that \( \tilde{u} \) is uniformly bounded by Theorem~\ref{thm1V} and has a bounded modulus of strict convexity. This allows us to apply Caffarelli’s regularity result (\cite{zbMATH04154994}). However, although this argument effectively ensures regularity along the diagonal near the vertices, our proof of Proposition~\ref{proposition-vertices} essentially covers this case as well. The proof of this proposition will be presented in the next subsection.

\subsection{Regularity away from the diagonal near the vertices}
	
In the preceding subsection, we proved the regularity of solutions along the diagonal near the vertices. Near the vertex point \((0, 0)\), for a point $p=(p_1,p_2)$, we normalize our solutions at the point $(p_1,p_2)$ as follows:
$$ 
u^*(x_{1},x_{2}) = u(x_{1},x_{2}) -u(p_1,p_2) - u_{1}(p_1,p_2)(x_{1} - p_1) - u_{2}(p_1,p_2)(x_{2} - p_2).
$$
Recalling Definition \ref{DDD-12}, we also need to define the following constants.
\begin{definition}[Donaldson's notion of $D_i$ of $u$ at $(p_1,p_2)$] 
Let $(p_1,p_2)$ be close to the origin and assume that $u^*$ is normalized at $p=(p_1,p_2)$ for the solution $u$. We define: 
$$ D_{1}\left(u; p \right):= \frac{u^*(p_1,0)}{p_2}=\frac{u(p_1,0)-u(p_1,p_2)+p_2 u_{2}(p_1,p_2)}{p_2}.$$
Similarly, we can define $D_2\left(u; p \right)$. 
\end{definition}
Analogous to the result of Theorem~\ref{thm1}, we establish the following theorem:
	
\begin{theorem}
\label{thm1V2}
There exists a universal positive constant $\tilde{D}=\tilde{D}(a,A)$ such that 
$$ D_1(u;(p_1,p_2)) \leq \tilde{D}, $$
for any point $ (p_1,p_2).$
\end{theorem}
	
The proof is very similar to the proof of Theorem~\ref{thm1}. First, note that the result follows from Theorem~\ref{thm1V} if \( p_1 \sim p_2 \). Therefore, without loss of generality, we may assume that \( p_1 \gg p_2 \) and prove that \( D_1 = D_1 \left( u; (p_1, p_2) \right) \) is uniformly bounded. Suppose the maximum of \( D_1 \) occurs at \( p = (p_1, p_2) \). As in Section \ref{section--2}, we demonstrate that if \( D_1 \) is large, a sequence of solutions can be obtained that converges to a nowhere strictly convex function. Thus, to derive a contradiction with the affine linearity of the limit on the \( x_1 \)-axis, we must demonstrate that the scaled solutions, when restricted to the \( x_1 \)-axis, remain uniformly strictly convex.
	
\subsection{Proof of Theorem \ref{thm1V2}}

Let \( u \) be a solution for which \( D_1 \) is large near a vertex. We assume that the vertex is at the origin and that the edges coincide with the \( x_1 \)- and \( x_2 \)-axes. Suppose the maximum of \( D_1 \) occurs at \( p = (p_1, \epsilon_0) \), and \( u \) is normalized at \( p = (p_1, \epsilon_0) \). Let \( D_1 = D_1 \left(u; p \right) \), and rescale the function \( u \) as follows:  
$$  
\tilde{u}(x_1, x_2) := \frac{1}{\epsilon_0 D_1 \left(u; p \right)} u^* \left(p_1x_1, \epsilon_0 x_2 \right).  
$$  
One can easily verify that:  
$$  
\det D^2\tilde{u} (x_{1},x_{2}) = \frac{p_1^2\epsilon_0^2}{D_1^2\epsilon_0^2} \det D^2u (p_1x_{1},\epsilon_{0}x_{2}) = \frac{p_1}{D_1^2\epsilon_0} \frac{H(p_1x_{1},\epsilon_{0}x_{2})}{x_1x_2} \qquad \text{in }\left(0, \frac{1}{p_1} \right) \times \left(0, \frac{1}{\epsilon_0} \right).  
$$  
	
\begin{proof}[Proof of Theorem \ref{thm1V2}]
Suppose we have a sequence \( \{u_n\} \) and points \( (p_n, \epsilon_n) \) such that  
$$ D_1=D_1^{(n)}=D_1 \left(u_n; (p_n, \epsilon_n) \right) \to +\infty. $$  
We may assume that \( u_n \) is normalized at \( (p_n, \epsilon_n) \) and define \( \tilde{u}_n \) as follows:  
$$  
\tilde{u}_n(x_1, x_2) = \frac{1}{\epsilon_n D_1 \left(u_n; (p_n, \epsilon_n) \right)} u_n^* \left(p_n x_1, \epsilon_n x_2 \right).  
$$
Since the sequence \( \tilde{u}_n \) is locally uniformly bounded, it converges locally uniformly to a convex function \( u_{\infty} \) on \( \{(x_1,x_2) \, : \, x_1>0, \, x_2>0\} \).  
		
On the other hand, we have  
$$  
\det D^2\tilde{u}_n (x_1, x_2) \sim \frac{p_n}{D_1^2\epsilon_n} \frac{1}{x_1 x_2} \to 0,  
$$  
on any compact subset of \( \{(x_1,x_2) \, : \, x_1>0, \, x_2>0\} \). Therefore, \( u_{\infty} \) is nowhere strictly convex. This implies that the set \( \{u_{\infty} = 0\} \) contains a line, a half-plane, or a line segment \( l \) that intersects the boundary of \( \{(x_1,x_2) \, : \, x_1 > 0, \, x_2 > 0\} \). It is easy to check that this must be a half-plane intersecting the \( x_1 \)-axis. The same argument as in Section \ref{section--2} shows that the restriction of \( u_{\infty} \) to the \( x_1 \)-axis is affine linear.  
		
On the other hand, for any \( \tilde{u} = \tilde{u}_n \), we have:
\begin{align*}
\tilde{u}_1(2,0) - \tilde{u}_1(1,0) = \int_1^2 \tilde{u}_{11}(t,0) \, dt = \frac{p_1^2}{D_1 \epsilon_0} \int_1^2 u_{11}(p_1 t, 0) \, dt & = \frac{p_1^2}{D_1 \epsilon_0} \int_1^2 \frac{H(p_1 t, 0)}{p_1 t} \, dt \\
&\sim \frac{p_1}{D_1 \epsilon_0} \int_{p_1}^{2 p_1} \frac{1}{t} \, dt \\
&= \frac{p_1}{D_1 \epsilon_0} \log 2.
\end{align*}  
This implies that \( \tilde{u}_n \) is strictly convex on the \( x_1 \)-axis uniformly in \( n \), provided there is a uniform upper bound on \( \frac{\epsilon_n D_1}{p_n} \). Therefore, such a bound would complete the proof. We prove this in Lemma~\ref{lemmascale}, which is a direct consequence of Theorem~\ref{thm1V}.
\end{proof}
	
\begin{lemma}
\label{lemmascale}
There exists a universal constant \( C > 0 \) such that  
\[
\frac{p_2 D_1}{p_1} \leq C.
\]
\end{lemma}
	
\begin{proof}
Let \( u^* \) be the normalization of \( u \) at \( (p_1, p_1) \), and define  
\[
\tilde{u}(x_1, x_2) = \frac{1}{p_1} u^*(p_1 x_1, p_1 x_2).
\]  
Since \( E\left(u; (p_1, p_1) \right) \leq \tilde{E} \), by Theorem~\ref{thm1V}, and using similar arguments as those in Lemmas \ref{UPER-LOWER-bound} and \ref{UPER-LOWER-bound-2}, one can easily show that \( \tilde{u} \) is uniformly bounded on any compact subset of \( \{x_1 > 0, x_2 > 0\} \).  

On the other hand,  
\[
D_1 \left(u; (p_1, p_2) \right) = D_1 \left(\tilde{u}; (1, \epsilon) \right),
\]  
where \( \epsilon = \frac{p_2}{p_1} \). Hence, the uniform boundedness of \( \tilde{u} \) on compact sets implies that \( x_2 D_1 \left( \tilde{u}; (x_1, x_2) \right) \leq C \) for any \( (x_1, x_2) \sim (1, 0) \).  
		
This leads to the conclusion that  
\[
\frac{p_2 D_1}{p_1} = \epsilon D_1 = \epsilon D_1 \left( \tilde{u}; (1, \epsilon) \right) \leq C.
		\]  
\end{proof}
	
Finally, we present the proof of the up to corner counterpart of Proposition~\ref{proposition-edges}. Before proceeding, we recall Remark \ref{lambda_0}. Similar to that remark, we choose scaling \( \lambda \) such that  
\begin{equation}
\label{R-R-12-10000}
\lambda \int_{p_1-\lambda}^{p_1+\lambda} u_{11}(t,0) \, dt = p_1 D_1
\end{equation}
where \( D_1 \) is the same as in the previous arguments and is known to be universally bounded due to Theorem~\ref{thm1V2}. Since \( u_{11}(t,0) \sim t^{-1} \) near the vertex, it implies that \( \lambda \sim \sqrt{p_1p_2} \). This will be established in the next lemma.
	
\begin{lemma}  
\label{LLL-LLL-10}  
There exists a universal positive constant \( c \) such that:  
\[
0< c^{-1} \leq \frac{\lambda}{\sqrt{p_1p_2}} \leq c.
\]  
\end{lemma}  
	
\begin{proof}  
By referencing the definition of \( \lambda \) in identity \eqref{R-R-12-10000}, along with \eqref{restriction-x_1-x_2}, we derive:  
\[
\lambda \log \left( \frac{p_1+\lambda}{p_1-\lambda} \right) \sim p_2
\]
since Theorem~\ref{thm1V2} ensures that \( D_1 \) remains universally bounded. Defining \( \zeta := \frac{\lambda}{p_1} \), we can rewrite the expression as:  
\begin{equation}
\label{Map-axillary}
F(\zeta):=\zeta \log \left( \frac{1 + \zeta}{1 - \zeta} \right) \sim \frac{p_2}{p_1}.
\end{equation}  
Given our assumption that \( \frac{p_2}{p_1} \leq 1 \), it follows that \( \zeta \) remains universally bounded, i.e.,  
\[
0 < c^{-1} \leq \zeta \leq c.
\]
On the other hand, since  
\[
\frac{\lambda}{\sqrt{p_1p_2}} = \zeta \sqrt{\frac{p_1}{p_2}},
\]  
if \( p_1 \ll p_2 \), then by \eqref{Map-axillary}, \( \frac{p_2}{p_1} \) cannot belong to the image of the mapping \( \zeta \mapsto F(\zeta) \). This concludes the proof.
\end{proof}

\begin{proof}[Proof of Proposition \ref{proposition-vertices}]
Suppose that \(u\) is normalized at the point \((p_1,p_2)\) and assume
\(p_1\gg p_2\). Consider the rescaling
\[
\tilde u(x_1,x_2)
=
\frac{1}{p_2}\,
u^*(p_1+\lambda x_1,\, p_2 x_2),
\]
where, by Lemma~\ref{LLL-LLL-10}, we have \(\lambda\sim\sqrt{p_1p_2}\).
Under this transformation, \(\tilde u\) is defined on the ball
\(B_{1/2}((0,1))\).

One verifies that \(\tilde u\) is uniformly bounded and strictly convex.
As in the proof of Proposition~\ref{proposition-edges}, we may therefore
apply Caffarelli’s interior regularity theorem
(\cite{zbMATH04154994}; see also \cite[Lemma 4.41]{zbMATH06669881}).
The strict convexity follows from Heinz’s theorem
\cite{heinz1959differentialungleichung}.

Under this scaling, we compute
\[
\det D^2\tilde u(x_1,x_2) = \lambda^2 \frac{H(p_1+\lambda x_1, p_2 x_2)}{(p_1+\lambda x_1) p_2 x_2}.
\]
Given that $\lambda^2 \sim p_1 p_2$, the right-hand side belongs to $C^{0,\alpha}(B_{1/2}(0,1))$, with a norm bounded independently of $p_1$ and $p_2$. Consequently, Caffarelli’s estimate implies
\[
C^{-1} I \leq D^2\tilde u(0,1) = \begin{bmatrix}
\dfrac{\lambda^2}{p_2}u_{11}(p_1,p_2) & \lambda u_{12}(p_1,p_2) \\[6pt]
\lambda u_{12}(p_1,p_2) & p_2 u_{22}(p_1,p_2)
\end{bmatrix} \leq C I,
\]
where the constant $C = C\big(P, a, A, \|H\|_{C^{0,\alpha}(\overline P)}, \alpha, \|u\|_{C^0(\overline P)}\big)$ is uniform with respect to $p_1$ and $p_2$.

Using the scaling relation $\lambda \sim \sqrt{p_1 p_2}$ to invert the transformation, we deduce the following estimate for the Hessian of the original function:
\[
D^2u(p_1,p_2) = \begin{bmatrix}
u_{11}(p_1,p_2) & u_{12}(p_1,p_2) \\
u_{12}(p_1,p_2) & u_{22}(p_1,p_2)
\end{bmatrix} \sim \begin{bmatrix}
p_1^{-1} & 0 \\
0 & p_2^{-1}
\end{bmatrix} = D^2\mathbf{u}_0(p_1,p_2).
\]
\end{proof}

\section{Regularity near the Vertex II}
\label{section--5}

The regularity of solutions near the vertices of the polygon is the most delicate part of the analysis, since the singularities arising from the two adjacent boundary edges coalesce there. This issue was recently addressed by Huang \cite{huang2023guillemin} and by Huang and Shen \cite{huang2023monge}, who established boundary estimates for the Monge–Ampère equation on cones and polygonal domains under the Guillemin boundary condition. Building on the classical boundary regularity theory of Trudinger and Wang \cite{zbMATH05578708}, in this section we extend these results to the setting in which the right-hand side \(H\) is only Hölder continuous. The goal of this section is to prove Proposition \ref{prop:vertex-sharp}. In the first step, we prove that the function
\[
v=u-x_1\log x_1-x_2\log x_2
\]
is \(C^1\) up to the corner and satisfies a sharp pointwise \(C^{1,\alpha}\) modulus of continuity at the origin (lemma \ref{lem:C1-v}). This is essentially proven by Huang (\cite{huang2023guillemin}). For the reader's convenience, we will give the proof.
In the second step, we show that Huang's barries work for
the H\"older continuous function of the form $H(x_1,x_2)=F(x_1)G(x_2)$ (Lemma \ref{lem:product}).
Finally, we approximate a H\"older continuous function $H$ with a sequence $H_{\epsilon}$ which are of the form $H_{\epsilon}(x_1,x_2)=F_{\epsilon}(x_1)G_{\epsilon}(x_2)$ near the origin. 
\subsection{Huang's result revisited}

Throughout this subsection, we assume that \(H(x_1,x_2) \in C^{\alpha}(\overline{Q_1})\) and \(H(0,0)=1\). Let \(u=x_1\log x_1+x_2\log x_2+v(x_1,x_2)\) be a convex solution of the Monge–Ampère equation
\begin{equation}\label{eq:MA-vertex}
\det D^2u(x_1,x_2)=\frac{H(x_1,x_2)}{x_1x_2} \quad \text{in } Q_1,
\end{equation}
which satisfies the boundary conditions
\begin{equation}\label{eq:vertex-bdr-conditions}
v_{11}(x_1,0)=\frac{H(x_1,0)-1}{x_1}, \quad
v_{22}(0,x_2)=\frac{H(0,x_2)-1}{x_2},
\end{equation}
with the normalization \(v(0,0)=0\) and \(Dv(0,0)=0\). Note that equation \eqref{eq:MA-vertex} can be written in terms of the regular part \(v\) as:
\begin{equation}\label{eq:v-eq}
\bigl(x_1v_{11}+1\bigr)\bigl(x_2 v_{22}+1\bigr)-x_1x_2 v_{12}^2=H.
\end{equation}

The first step is to establish that \(v\) is \(C^1\) up to the corner with a controlled modulus of continuity at the origin.

\begin{lemma}
\label{lem:C1-v}
The function \(v\) belongs to \(C^1(\overline{Q_1})\). There exists a constant \(C\), depending only on \(\|H\|_{C^{\alpha}(\overline{Q_1})}\) and \(\|v\|_{L^{\infty}(Q_1)}\), such that
\[
\|v\|_{C^1(\overline{Q_1})} \le C.
\]
Moreover, for every \(\epsilon > 0\), there exists a \(\delta = \delta(\epsilon, \|H\|_{C^{\alpha}}, \|v\|_{\infty}) > 0\) such that
\[
|Dv(p)| \le \epsilon \qquad \text{whenever } |p| \le \delta.
\]
In particular, \(Dv\) is continuous at the origin with \(Dv(0) = \mathbf{0}\).
\end{lemma}

\begin{proof}
Since \(H \in C^{0,\alpha}(\overline{Q_1})\) and \(H(0) = 1\), we have the H\"older estimates at the origin:
\[
|H(x_1,0) - 1| \le C x_1^\alpha \qquad \text{and} \qquad |H(0,x_2) - 1| \le C x_2^\alpha.
\]
By the boundary conditions \eqref{eq:vertex-bdr-conditions}, these estimates imply
\[
|v_{11}(x_1,0)| \le C x_1^{-1+\alpha} \qquad \text{and} \qquad |v_{22}(0,x_2)| \le C x_2^{-1+\alpha}.
\]
Integrating these inequalities along the coordinate axes starting from the origin, and utilizing the normalization \(Dv(0) = \mathbf{0}\), we obtain
\[
|v_1(x_1,0)| \le C x_1^\alpha \qquad \text{and} \qquad |v_2(0,x_2)| \le C x_2^\alpha,
\]
which yields
\[
|v(x_1,0)| \le C x_1^{1+\alpha} \qquad \text{and} \qquad |v(0,x_2)| \le C x_2^{1+\alpha}.
\]

To establish the full \(C^1\) regularity up to the corner, we proceed by contradiction. Suppose there exist \(\varepsilon_0 > 0\) and a sequence of points \(p_n \in Q_1\) converging to the origin such that
\[
|Dv(p_n)| \ge \varepsilon_0.
\]
Let \(\delta_n := |p_n|\) and define the rescaled functions
\[
\tilde{v}_n(x) := \delta_n^{-1} v(\delta_n x).
\]
Then, each \(\tilde{v}_n\) satisfies the rescaled Monge-Amp\`ere type equation
\[
\big(x_1 (\tilde{v}_n)_{11} + 1\big) \big(x_2 (\tilde{v}_n)_{22} + 1\big) - x_1 x_2 (\tilde{v}_n)_{12}^2 = H(\delta_n x)
\]
on the expanding domains \(\Omega_n := \delta_n^{-1} Q_1\). 

The uniform convexity of \(v\) together with interior and edge regularity estimates provide uniform \(C^{2,\beta}\) bounds on compact subsets away from the coordinate axes. Meanwhile, the boundary estimates established above control the behavior of \(\tilde{v}_n\) near the axes. Consequently, by Arzel\`a-Ascoli, we can pass to a subsequence (still denoted by \(\tilde{v}_n\)) that converges locally uniformly, along with its gradient on compact subsets of \(\overline{Q_\infty} \setminus \{\mathbf{0}\}\), to a convex limit \(v_\infty\).

The limit function \(v_\infty\) solves the homogeneous equation
\[
(x_1 v_{\infty,11} + 1)(x_2 v_{\infty,22} + 1) - x_1 x_2 v_{\infty,12}^2 = 1 \qquad \text{in } Q_\infty = \{(x_1,x_2) \, : \, x_1 > 0, \, x_2 > 0\},
\]
subject to the homogeneous boundary conditions
\[
v_\infty(x_1,0) = 0, \quad \partial_1 v_\infty(x_1,0) = 0 \qquad \text{and} \qquad v_\infty(0,x_2) = 0, \quad \partial_2 v_\infty(0,x_2) = 0.
\]
By the uniqueness of the corresponding homogeneous limit problem (see Lemma~2.4 in \cite{huang2023guillemin}, where Lemma~\ref{Barrier-modified-Huang} guarantees that the required growth estimate \(|v_\infty(x)| = o(|x| \log|x|)\) holds as \(|x| \to +\infty\)), we conclude that \(v_\infty \equiv 0\). This implies \(\lim_{n\to\infty} |D\tilde{v}_n(p_n/\delta_n)| = |Dv_\infty(p_\infty)| = 0\) (where \(|p_n/\delta_n| = 1\)), which contradicts our assumption that \(|Dv(p_n)| = |D\tilde{v}_n(p_n/\delta_n)| \ge \varepsilon_0\). This contradiction completes the proof.
\end{proof}

\begin{lemma}\label{lem:product}
Let 
\[
H(x_1,x_2)=F(x_1)G(x_2),
\qquad F,G\in C^{0,\alpha}([0,1]), \qquad F(0)=G(0)=1.
\]
Let $v$ denote the remainder term in the representation
\[
u(x)=x_1\log x_1 + x_2\log x_2 + v(x).
\]
Then there exists a constant 
\[
C=C\!\left(\|F\|_{C^{0,\alpha}([0,1])},\,\|G\|_{C^{0,\alpha}([0,1])}\right)
\]
such that
\[
|v(x)| \le C\,|x|^{1+\alpha}
\qquad\text{for all } x\in Q_1.
\]
\end{lemma}

\begin{proof}
The proof is essentially the same as in Huang's~\cite[Lemma~3.1]{huang2023guillemin}. For the completeness, we give the proof here.
Since $F(0)=G(0)=1$, we have
\[
H(x_1,x_2)=F(x_1)G(x_2)=H(x_1,0)\,H(0,x_2).
\]
We first show that, near the origin,
\[
|v(x_1,x_2)-v(x_1,0)-v(0,x_2)|\le C_0 (x_1x_2)^2
\]
for some constant $C_0$. Let $\delta,r\in(0,1)$ be small constants to be chosen later, and define
\[
\overline B(x)
=
x_1\log x_1+x_2\log x_2+v(x_1,0)+v(0,x_2)
+x_1x_2\left(\frac{\delta}{r}-\frac{\delta^2}{2r^2}(x_1+x_2)\right).
\]
We claim that $\overline B$ is an upper barrier in $Q_r$.

A direct computation, exactly as in Huang's~\cite[Lemma~3.1]{huang2023guillemin}, gives
\[
x_1x_2\det D^2\overline B
\le H(x_1,0)H(0,x_2)-\frac{\delta^2}{r^2}x_1x_2
+\frac{\delta^4}{r^4}x_1^2x_2^2.
\]
Hence, if $\delta\le \frac{1}{\sqrt2}$, then
\[
x_1x_2\det D^2\overline B
\le H(x_1,0)H(0,x_2)-\frac{\delta^2}{2r^2}x_1x_2.
\]
Since in the product case
\[
H(x_1,0)H(0,x_2)=H(x_1,x_2),
\]
it follows that
\[
\det D^2\overline B\le \det D^2u
\qquad\text{in }Q_r,
\]
where $u=x_1\log x_1+x_2\log x_2+v$ is the full solution.

Next, as in~\cite[Lemma~3.1]{huang2023guillemin}, we compute
\[
(\overline B)_{11}
=
\frac{1}{x_1}+\partial_{11}v(x_1,0)-\frac{\delta^2}{r^2}x_2
=
\frac{H(x_1,0)-\frac{\delta^2}{r^2}x_1x_2}{x_1}.
\]
Since $F(0)=1$ and $F\in C^{0,\alpha}([0,1])$, we may choose $r>0$ sufficiently small so that
\[
H(x_1,0)=F(x_1)\ge \frac12
\qquad\text{for }0\le x_1\le r.
\]
Thus
\[
(\overline B)_{11}\ge \left(\frac12-\delta^2\right)x_1^{-1}\ge0
\]
in $Q_r$, provided that $\delta\le \frac12$. Similarly,
\[
(\overline B)_{22}\ge0
\]
in $Q_r$ for sufficiently small $r$. Therefore, taking $\delta=\frac14$ and $r$ small enough, the function $\overline B$ is convex in $Q_r$.

Now compare boundary values on $\partial Q_r$. Let
\[
\varepsilon(r)=\max_{Q_r}|Dv|.
\]
By Lemma~\ref{lem:C1-v}, we know that $\varepsilon(r)\to0$ as $r\to0$. Hence, after decreasing $r$ if necessary, we may assume
\[
\varepsilon(r)\le \frac{\delta-\delta^2}{2}.
\]
Then:

\begin{itemize}
\item On $x_1=0$ or $x_2=0$, we have by definition
\[
u=\overline B.
\]

\item On $x_1=r$ and $0\le x_2\le r$, we compute
\[
\overline B(r,x_2)-u(r,x_2)
=
v(r,0)+v(0,x_2)
+x_2\left(\delta-\frac{\delta^2}{2r}(r+x_2)\right)
-v(r,x_2),
\]
so
\[
\overline B(r,x_2)-u(r,x_2)
\ge (\delta-\delta^2-2\varepsilon(r))x_2\ge0.
\]

\item On $x_2=r$ and $0\le x_1\le r$, the same argument gives
\[
\overline B(x_1,r)\ge u(x_1,r).
\]
\end{itemize}

Therefore, $\overline B$ is an upper barrier for $u$ in $Q_r$. By the same argument,
\[
\underline B(x)
=
x_1\log x_1+x_2\log x_2+v(x_1,0)+v(0,x_2)
-x_1x_2\left(\frac{\delta}{r}-\frac{\delta^2}{2r^2}(x_1+x_2)\right)
\]
is a lower barrier for $u$ in $Q_r$. Consequently,
\[
|v(x_1,x_2)-v(x_1,0)-v(0,x_2)|\le C_0(x_1x_2)^2
\qquad\text{for all }x\in Q_r.
\]

It remains to estimate the boundary terms. Along the $x_1$-axis, the equation reduces to
\[
\frac{d}{dt}v(t,0)=\frac{F(t)-1}{t}.
\]
Since $F\in C^{0,\alpha}([0,1])$ and $F(0)=1$, we have
\[
|F(t)-1|\le [F]_{C^{0,\alpha}}t^\alpha,
\]
and therefore
\[
\left|\frac{d}{dt}v(t,0)\right|
\le [F]_{C^{0,\alpha}}t^{\alpha-1}.
\]
Integrating from $0$ to $t$ and using $v(0,0)=0$, we obtain
\[
|v(t,0)|\le Ct^{1+\alpha}.
\]
Similarly,
\[
|v(0,t)|\le Ct^{1+\alpha}.
\]
Combining these estimates yields
\[
|v(x)|\le C|x|^{1+\alpha}
\qquad\text{for all }x\in Q_r,
\]
after possibly adjusting $C$.
\end{proof}

\subsection{ Approximation }

To extend the vertex estimates of Lemma~\ref{lem:product} to general non-product right-hand sides \(H \in C^{0,\alpha}\), we construct a family of approximating problems with a localized product structure near the corner.

Let \(\phi \in C^\infty(\mathbb{R})\) be a standard cut-off function such that \(\phi \equiv 1\) on \([0,1]\), \(\phi \equiv 0\) on \([2,\infty)\), and \(0 \le \phi \le 1\). For \(\epsilon > 0\), we define the cut-off function on the quadrant as
\[
\phi_{\epsilon}(x_1, x_2) = \phi\left(\frac{x_1 + x_2}{\epsilon}\right).
\]
Let \(G(x_1, x_2) := H(x_1, 0) H(0, x_2)\) be the product approximation and define the modified right-hand side
\[
H_{\epsilon} := \phi_{\epsilon} G + (1 - \phi_{\epsilon}) H.
\]
By construction, \(H_{\epsilon}\) coincides with \(H\) on the coordinate axes and satisfies \(H_{\epsilon}(0) = 1\). Since \(H \in C^{0,\alpha}(\overline{Q_1})\), we have
\[
\begin{aligned}
|H(x_1, x_2) - G(x_1, x_2)| &\le |H(x_1, x_2) - H(x_1, 0)| + |H(x_1, 0) - H(0,0)H(0, x_2)| \\
&\le |H(x_1, x_2) - H(x_1, 0)| + |H(0, x_2)|\,|H(x_1, 0) - H(0, 0)| \\
&\le C(x_2^\alpha + x_1^\alpha).
\end{aligned}
\]
Since \(\phi_\epsilon\) is supported on the set where \(x_1 + x_2 \le 2\epsilon\), we obtain the uniform approximation bound:
\begin{equation}\label{eq:H-approx-bound}
|H - H_{\epsilon}| = \phi_{\epsilon} |H - G| \le C \epsilon^{\alpha}.
\end{equation}

Let \(u_{\epsilon}\) be the unique convex solution to the Dirichlet problem
\[
\det D^2 u_{\epsilon} = \frac{H_{\epsilon}}{x_1 x_2} \quad \text{in } \Delta = \{(x_1,x_2) \, : \, x_1 > 0, \, x_2 > 0, \, x_1 + x_2 < 1\},
\]
with the same Guillemin boundary conditions as \(u\) on \(\partial \Delta\).

\begin{lemma}
\label{lem:approx}
There exist constants \(A, B > 0\), depending only on the data and independent of \(\epsilon\), such that
\[
|u(x) - u_{\epsilon}(x)| \le B \epsilon^{\alpha}(x_1 + x_2) \qquad \text{whenever } x_1 + x_2 \le \frac{\epsilon}{2}.
\]
\end{lemma}

\begin{proof}
We construct an upper barrier for \(u_{\epsilon}\) of the form
\[
\tilde{u} = u + A(x_1^{1+\alpha} + x_2^{1+\alpha}) - B \epsilon^{\alpha}(x_1 + x_2).
\]
The Hessian of the perturbation \(w = \tilde{u} - u\) is given by
\[
D^2 w = A \alpha(1+\alpha) \begin{pmatrix} x_1^{\alpha-1} & 0 \\ 0 & x_2^{\alpha-1} \end{pmatrix}.
\]
Recall the identity for the determinant of the sum of two \(2 \times 2\) symmetric matrices \(M\) and \(N\):
\[
\det(M + N) = \det M + \operatorname{tr}(\operatorname{cof}(M)N) + \det N \ge \det M + \operatorname{tr}(\operatorname{cof}(M)N).
\]
Multiplying by \(x_1 x_2\) and applying this to \(D^2 \tilde{u} = D^2 u + D^2 w\), we have
\[
x_1 x_2 \det D^2 \tilde{u} \ge H + x_1 x_2 \operatorname{tr}(\operatorname{cof}(D^2 u) D^2 w) = H + x_1 x_2 (u_{22} w_{11} + u_{11} w_{22} - 2 u_{12} w_{12}).
\]
Using the lower bounds \(u_{11} \ge c_0/x_1\) and \(u_{22} \ge c_0/x_2\) from Theorem~\ref{main-theorem}, it follows that
\[
x_1 x_2 \det D^2 \tilde{u} \ge H + c c_0 A (x_1^{\alpha} + x_2^{\alpha}),
\]
where \(c = \alpha(1+\alpha) > 0\). On the domain \(\Delta_{\epsilon} = \{(x_1,x_2) \, : \, x_1 + x_2 \le \epsilon\}\), the bound \eqref{eq:H-approx-bound} implies \(H_{\epsilon} \le H + C_1 \epsilon^{\alpha}\). By choosing \(A\) sufficiently large such that \(c c_0 A \ge 2C_1\), we ensure
\[
x_1 x_2 \det D^2 \tilde{u} \ge H_{\epsilon} = x_1 x_2 \det D^2 u_{\epsilon} \quad \text{in } \Delta_{\epsilon}.
\]
Regarding the boundary \(\partial \Delta_{\epsilon}\), on the coordinate axes (\(x_1 = 0\) or \(x_2 = 0\)), we choose \(B \ge A\) to ensure \(\tilde{u} \le u = u_\epsilon\). On the remaining face \(\Gamma_{\epsilon} = \{(x_1,x_2) \, : \, x_1 + x_2 = \epsilon\}\), the difference \(|u-u_\epsilon|\) is bounded by \(C\epsilon^\alpha\) (by the global stability estimate for Alexandrov solutions of the Monge--Amp\`ere equation with respect to perturbations of the right-hand side; see \cite{DePhilippisFigalliSavin, zbMATH01614149}). Thus, by taking \(B\) sufficiently large relative to this stability constant, we obtain \(\tilde{u} \le u_{\epsilon}\) on \(\Gamma_{\epsilon}\). 

The comparison principle then implies \(\tilde{u} \le u_{\epsilon}\) in \(\Delta_{\epsilon}\). An analogous construction for a lower barrier yields the reverse inequality, completing the proof.
\end{proof}

\subsection{Sharp regularity at the vertex}

We now establish the sharp pointwise decay for the regular component \(v\).

\begin{proposition}[Sharp regularity at the vertex]
\label{prop:vertex-sharp}
Assume \(H \in C^{0,\alpha}(\overline{Q_1})\) with \(H(0) = 1\), and let \(u = x_1 \log x_1 + x_2 \log x_2 + v\) be the solution to \eqref{eq:MA-vertex}. Then \(Dv(0) = 0\), and there exists a constant \(C\), depending only on \(\|H\|_{C^{0,\alpha}}\) and \(\|v\|_{L^\infty}\), such that
\[
|v(x)| \le C|x|^{1+\alpha} \qquad \text{for all } x \in Q_{1/2}.
\]
In particular, \(v\) exhibits \(C^{1,\alpha}\) regularity at the vertex in the pointwise sense.
\end{proposition}

\begin{proof}
Fix \(x \in Q_{1/2}\) and set \(\epsilon = 2|x|\). We construct the approximating solution 
\[ u_{\epsilon} = x_1 \log x_1 + x_2 \log x_2 + v_{\epsilon}, \] corresponding to the modified right-hand side \(H_{\epsilon}\) (as defined in the previous subsection). Since \(H_{\epsilon}\) possesses a product structure in \(\Delta_{\epsilon/2}\), Lemma~\ref{lem:product} implies that its regular part satisfies
\[
|v_{\epsilon}(x)| \le C|x|^{1+\alpha},
\]
where \(C\) is a constant independent of \(\epsilon\). By Lemma~\ref{lem:approx}, we can estimate the difference between the actual solution and the approximation:
\[
|v(x)| \le |v(x) - v_{\epsilon}(x)| + |v_{\epsilon}(x)| \le B\epsilon^{\alpha}|x| + C|x|^{1+\alpha}.
\]
Substituting \(\epsilon = 2|x|\), we obtain \(|v(x)| \le C' |x|^{1+\alpha}\), which yields the desired pointwise decay.
\end{proof}

\begin{proposition}[Euclidean regularity near the vertex]
\label{prop:vertex-euclidean}
Assume that \(H \in C^{0,\alpha}(\overline{Q_1})\) with \(H(0) = 1\), and let \(u = x_1 \log x_1 + x_2 \log x_2 + v\) solve 
\[ \det D^2u = \frac{H}{x_1 x_2} \qquad \text{in } Q_1. \] 
Then there exists a constant \(C\), depending only on \(\|H\|_{C^{0,\alpha}(\overline{Q_1})}\) and \(\|v\|_{L^\infty(Q_1)}\), such that
\[
\|v\|_{C^{1,\alpha/2}(\overline{Q_{1/2}})} \le C.
\]
In particular, \(v\) enjoys \(C^{1,\alpha/2}\) regularity near the vertex.
\end{proposition}

\begin{proof}
By Proposition~\ref{prop:vertex-sharp}, we have the growth bound
\begin{equation}
\label{eq:v-bound-vertex}
|v(x)| \le C |x|^{1+\alpha} \qquad \text{for all } x \in Q_{1/2}.
\end{equation}
This implies \(v(0) = 0\), and combined with the weighted second-derivative bounds from Theorem~\ref{main-theorem} and Appendix~\ref{app:Schauder}, it ensures \(Dv(0) = \mathbf{0}\) via Lemma~\ref{lem:ode-C1}.

Let \(p = (x_1, x_2)\) and \(q = (y_1, y_2)\) be points in \(Q_{1/2}\). By fixing \(x_2\) and considering the function \(f(t) = v(t, x_2)\), the hypotheses of Lemma~\ref{lem:ode-optimal} are satisfied due to \eqref{eq:v-bound-vertex} and the weighted estimates. This implies
\begin{equation}
\label{eq:deriv1-holder}
|\partial_1 v(x_1, x_2) - \partial_1 v(y_1, x_2)| \le C |x_1 - y_1|^{\alpha/2}.
\end{equation}
Similarly, fixing \(y_1\) and applying the same lemma to \(g(t) = v(y_1, t)\) yields
\begin{equation}
\label{eq:deriv2-holder}
|\partial_2 v(y_1, x_2) - \partial_2 v(y_1, y_2)| \le C |x_2 - y_2|^{\alpha/2}.
\end{equation}
Using the triangle inequality, the total variation of the gradient is bounded by
\[
|Dv(p) - Dv(q)| \le |Dv(x_1, x_2) - Dv(y_1, x_2)| + |Dv(y_1, x_2) - Dv(y_1, y_2)| \le C |p - q|^{\alpha/2}.
\]
Thus, we conclude that \(\|v\|_{C^{1,\alpha/2}(\overline{Q_{1/2}})} \le C\).
\end{proof}

\section{Proof of Theorem \ref{Low-regularity}: A Weaker Version of the Main Result}
\label{section--4}

In this section, we establish a priori estimates for solutions to the Monge--Amp\`ere equation \eqref{MA-0000} under the assumption that the right-hand side \(H\) belongs to the H\"older space \(C^{0,\alpha}(\overline{P})\); i.e. we prove that the solution satisfies the weak Guillemin boundary condition $u - \sum_{i=1}^{N} l_i \log l_i \in C^{1,\alpha/2}(\overline{P})$.

For simplicity of exposition, we assume the existence of solutions to the boundary value problem \eqref{MA-0000} when the function \(H\) is smooth up to the boundary \(\overline{P}\). In this case, we invoke the existence results stated in Theorem~\ref{known-regular-solution}. However, this assumption is not essential; one could alternatively construct solutions directly by approximating \(H\) with a sequence of smooth functions \(H_n\) converging to \(H\) in \(C^{0,\alpha}(\overline{P})\) and applying the classical continuity method. Indeed, the existence results of Theorem~\ref{known-regular-solution}, originally established by Rubin~\cite{rubin2015monge} and Huang~\cite{huang2023guillemin}, can be recovered from the key a priori estimate established in Theorem~\ref{main-theorem} for smooth right-hand sides, combined with standard elliptic regularity and bootstrap arguments. A similar regularization approach is discussed in Donaldson's work~\cite{donaldson2008extremal}, where analogous ideas are applied to the fourth-order Abreu equation \eqref{Abreu-eq-4th}.

For the reader's convenience, we restate the theorem to be proved in this section.

\begin{theorem}[A priori estimate for $C^{0,\alpha}$ right-hand side]
\label{main-a-priori-estimate}
Let \(P = \{x \in \mathbb{R}^2 : l_i(x) > 0, \, i = 1, \dots, N\}\) be a polytope in \(\mathbb{R}^2\), where each \(l_i(x)\) is an affine linear function. Let \(\{H_n\}\) be a sequence of smooth functions on \(\overline{P}\) satisfying the uniform bounds
\(
0 < a \le H_n(x) \le A
\).
Assume that \(H_n \to H\) in \(C^{0,\alpha}(\overline{P})\), and let \(\{u_n\}\) be the corresponding solutions of the Monge--Amp\`ere equation
\[
\det D^2 u_n = \frac{H_n}{l_1 \cdots l_N} \qquad \text{in } P,
\]
subject to the Guillemin boundary condition
\[
u_n - \sum_{i=1}^{N} l_i \log l_i \in C^\infty(\overline{P}).
\]
Then there exists a convex function $u$ in the class $C^{2,\alpha}(P)$ on \(\overline{P}\) solving
\[
\det D^2 u(x) = \frac{H(x)}{l_1(x) \cdots l_N(x)} \qquad \text{in } P,
\]
and \(u\) satisfies the weak Guillemin boundary condition
\[
u - \sum_{i=1}^{N} l_i \log l_i \in C^{1,\alpha/2}(\overline{P}).
\]
\end{theorem}

\begin{proof}
Choose a sequence \(\{H_n\}\subset C^\infty(\overline P)\) such that
\( H_n \to H \) in \( C^{0,\alpha}(\overline P) \)
as \(n \to \infty\), satisfying
\[
\sup_n \|H_n\|_{C^{0,\alpha}(\overline P)} < \infty
\qquad\text{and}\qquad
H_n \ge c_0 > 0
\]
for some constant \(c_0\) independent of \(n\). For each \(n\), let \(u_n\) be the unique smooth, strictly convex solution of the Dirichlet problem
\[
\det D^2u_n=\frac{H_n}{\prod_{i=1}^N l_i}\qquad\text{in }P,
\]
subject to the Guillemin boundary condition
\[
u_n-\mathbf u_0 =: v_n \in C^\infty(\overline P),
\qquad\text{where }
\mathbf u_0:=\sum_{i=1}^N l_i\log l_i.
\]
The existence and uniqueness of such a solution \(u_n\) is guaranteed by Theorem~\ref{known-regular-solution}.

\medskip
\noindent
\textbf{STEP 1: Uniform interior estimates and local compactness.}

By the smooth Hessian estimate established in Theorem~\ref{main-theorem}, there exists a constant \(C\ge 1\), depending only on \(P\), \(c_0\), and \(\sup_n\|H_n\|_{C^{0,\alpha}(\overline P)}\), such that
\[
C^{-1}D^2\mathbf u_0 \le D^2u_n \le C D^2\mathbf u_0
\qquad\text{in }P
\]
holds for all \(n\). Let \(K\Subset P\) be an arbitrary compact subset. Since \(\mathbf u_0\) is smooth in the interior of \(P\), there exists a constant \(C_K>0\) such that \(\|D^2\mathbf u_0\|_{L^\infty(K)}\le C_K\). Hence, we obtain the uniform bound
\(
\|D^2u_n\|_{L^\infty(K)}\le C\,C_K
\).
Using the relation \(v_n = u_n - \mathbf u_0\), this interior Hessian control yields
\(
\|v_n\|_{C^2(K)}\le C_K'
\)
for a constant \(C_K'\) independent of \(n\). By the Arzel\`a--Ascoli theorem and a standard diagonal argument, there exists a subsequence (still denoted by \(\{v_n\}\)) and a function \(v \in C^1_{\mathrm{loc}}(P)\) such that
\[
v_n \to v
\quad\text{and}\quad
u_n \to u := \mathbf u_0 + v
\quad\text{in } C^1_{\mathrm{loc}}(P).
\]
Since \(H_n\to H\) in \(C^{0,\alpha}(\overline P)\), the limit \(u \in C^{2,\alpha}(P)\) is strictly convex and satisfies the Monge--Amp\`ere equation
\[
\det D^2u=\frac{H}{\prod_{i=1}^N l_i}
\qquad\text{in }P.
\]

\medskip
\noindent
\textbf{STEP 2: Uniform boundary estimates near the edges.}

Let \(p\in \partial P\) be a point belonging to the relative interior of an edge of \(P\). We choose a neighborhood \(U_p\) of \(p\) and affine coordinates mapping this neighborhood to the normal form detailed in Appendix~\ref{app:Schauder}. Applying Corollary~\ref{cor:global}, we obtain a constant \(C_p > 0\), depending only on the local geometry of \(P\), \(\alpha\), \(c_0\), and \(\sup_n\|H_n\|_{C^{0,\alpha}(\overline P)}\), such that
\[
\|v_n\|_{C^{1,\alpha/2}(U_p\cap \overline P)}\le C_p
\]
for all \(n\). Passing to the subsequential limit, we obtain
\[
\|v\|_{C^{1,\alpha/2}(U_p\cap \overline P)}\le C_p.
\]

\medskip
\noindent
\textbf{STEP 3: Uniform boundary estimates near the vertices.}

Let \(q\) be a vertex of \(P\) and \(U_q\) be a neighborhood of \(q\). By Proposition~\ref{prop:vertex-euclidean}, there exists a constant \(C_q > 0\), depending only on the local geometry of \(P\), \(\alpha\), \(c_0\), and \(\sup_n\|H_n\|_{C^{0,\alpha}(\overline P)}\), such that
\[
\|v_n\|_{C^{1,\alpha/2}(U_q\cap \overline P)}\le C_q
\]
for all \(n\). Passing to the subsequential limit, we obtain
\[
\|v\|_{C^{1,\alpha/2}(U_q\cap \overline P)}\le C_q.
\]

\medskip
\noindent
\textbf{STEP 4: Globalization and passage to the limit.}

Since the boundary \(\partial P\) is compact, we can cover it with finitely many neighborhoods \(U_1, \dots, U_m\) of edge points and vertices where the estimates from STEP 2 and STEP 3 hold. Let
\[
K := \overline P \setminus \bigcup_{j=1}^m U_j.
\]
Because \(K\) is a compact subset lying strictly in the interior of \(P\), the uniform \(C^2(K)\) bound on \(v_n\) from STEP 1 implies a uniform \(C^{1,\alpha/2}(K)\) bound. 

Combining this interior estimate with the local boundary estimates from STEP 2 and STEP 3 via a partition of unity, we conclude that
\[
\|v_n\|_{C^{1,\alpha/2}(\overline P)}\le C
\]
for a constant \(C > 0\) depending only on \(P\), \(\alpha\), \(c_0\), and \(\sup_n\|H_n\|_{C^{0,\alpha}(\overline P)}\). By the compact embedding \(C^{1,\alpha/2}(\overline P) \hookrightarrow C^{1,\beta}(\overline P)\) for \(0 < \beta < \alpha/2\), the subsequence converges to \(v\) in \(C^{1,\beta}(\overline P)\). The limit function \(v\) inherits the bound
\[
\|v\|_{C^{1,\alpha/2}(\overline P)}\le C,
\]
which implies that \(u - \mathbf u_0 \in C^{1,\alpha/2}(\overline P)\). Finally, since the limit \(u\) solves the Monge--Amp\`ere equation in \(P\) and satisfies the weak Guillemin boundary condition, uniqueness of such a weak solution guarantees that the entire original sequence \(\{u_n\}\) converges to \(u\).
\end{proof}

Having completed the proof of Theorem~\ref{Low-regularity}, we turn to Section~\ref{sec:boundary-regularity-completion}, where we upgrade the regularity exponent from $\alpha/2$ to $\alpha$ using the algebraic properties of the Monge--Ampère equation near the boundary. This improvement is secondary, however, as the crux of the result resides in establishing \(
u - \sum_{i=1}^{N} l_i \log l_i \in C^{1,\alpha/2}(\overline{P})\).

\section{Boundary regularity and the Completion of our main Result Theorem~\ref{main-theorem}}
\label{sec:boundary-regularity-completion}

In this section we complete the proof of Theorem~\ref{main-theorem}. Recall 
\[
\mathbf u_0=\sum_{i=1}^N l_i\log l_i.
\]
Proposition \ref{proposition-edges}, Proposition \ref{proposition-vertices} and  Theorem \ref{main-a-priori-estimate} imply that
\[
D^2u\sim D^2\mathbf u_0
\qquad \text{in }P,
\]
and
\[
u-\mathbf u_0\in C^{1,\alpha/2}(\overline P).
\]
Thus it remains to improve the boundary H\"older exponent from $\alpha/2$ to $\alpha$, and to establish the sharp weighted second-derivative bounds near $\partial P$. All arguments in this section are carried out within the a priori setting, assuming $u-\textbf{u}_0$ is smooth up to the boundary.

Since the statements are local, we work in a small neighborhood of a boundary point. Around an arbitrary boundary point, we pass to toric complex coordinates and use the complex Monge--Amp\`ere equation satisfied by the corresponding K\"ahler potential. Since the real Hessian of $u$ is uniformly comparable to that of the Hessian of the Guillemin model $\mathbf u_0$, the associated K\"ahler metric is uniformly equivalent to the model toric metric associated to $\mathbf{u}_0$. This allows us to invoke the local $C^{2,\alpha}$ theory for the complex Monge--Amp\`ere equation.  We then transfer the resulting regularity back to the symplectic side and combine it with calculus Lemma~\ref{lem:ode-generalized} to bootstrap the boundary exponent to the sharp value $\alpha$ and sharp blow up rate for $D^2(u-\textbf{u}_0)$ (Propositions \ref{por:edge-C1a} and \ref{por:vertex-C1a}).

\subsection{Toric complex coordinates and the complex Monge--Amp\`ere equation}
\label{subsec:toric-complex-coord}

We begin by recalling the toric setup. Let $\phi$ denote the Legendre transform of $u$, namely
\[
\phi(t)=x\cdot t-u(x),
\qquad
x_i=\partial_{t_i}\phi(t),
\qquad
t_i=\partial_{x_i}u(x).
\]
Then the Hessians are related by
\[
D^2u(x)=\bigl(D^2\phi(t)\bigr)^{-1},
\qquad
\det D^2u(x)=\frac{1}{\det D^2\phi(t)}.
\]
On the dense torus $(\mathbb C^*)^2\subset X_P$, we introduce complex coordinates $z=(z_1, z_2)$ via
\[
z_i=e^{\frac{t_i+\sqrt{-1}\theta_i}{2}},
\qquad
t_i=\log|z_i|^2,
\]
and define the potential
\(
\Psi(z)=\phi(t)
\).
Then $\Psi$ is a torus-invariant K\"ahler potential, and the corresponding K\"ahler form is given by $\omega=\sqrt{-1}\,\partial\bar\partial\Psi$.

By a direct calculation in logarithmic coordinates, the determinant of the complex Hessian satisfies:
\begin{equation}
\label{eq:complex-density-local}
\log\det(\Psi_{z_i\bar z_j})
=
-\sum_{i=1}^2 \log |z_i|^2 -\log\det D^2\phi(t).
\end{equation}
Recalling that $t_i = u_{x_i}$ and using the relation between the determinants of $D^2u$ and $D^2\phi$, we may rewrite the right-hand side using \eqref{MA-0000} as:
\[
-\sum_{i=1}^2 u_{x_i}
-\log H
+\sum_{k=1}^N \log l_k(x).
\]
The crucial point of this transformation is that, after passing to a toric chart adapted either to an edge or to a vertex, the logarithmic singularities $\log l_k$ cancel exactly against the singular terms coming from the derivatives of the Guillemin potential (the $u_{x_i}$ terms). Therefore, the right-hand side of \eqref{eq:complex-density-local} extends as a $C^{0,\alpha}$ function in the local chart.

On the other hand, by \eqref{eq-EQ-00}, the metric $\omega$ is uniformly equivalent to the model toric metric $\omega_0$ induced by the canonical potential $\mathbf u_0$. Since $\omega_0$ extends smoothly across the toric divisor in each fixed chart, it follows that in these local complex coordinates the Hermitian matrix $(\Psi_{z_i\bar z_j})$ is uniformly elliptic and bounded. Consequently, the local complex Monge--Amp\`ere equation for $\Psi$ has a $C^{0,\alpha}$ right-hand side and uniformly elliptic coefficients.

To apply the local $C^{2,\alpha}$ theory, it remains to justify that the potential $\Psi$ is locally bounded in the relevant toric chart. This is provided in the edge and vertex normalizations below. Once boundedness is established, we apply the local $C^{2,\alpha}$ estimates for the complex Monge--Amp\`ere equation (as developed, for instance, by Tosatti--Wang--Weinkove--Yang \cite{TWWY}) to conclude that, in each toric chart intersecting the boundary:
\begin{equation}
\label{eq:Psi-C2alpha}
\Psi\in C^{2,\alpha}_{\mathrm{loc}}.
\end{equation}
We now explain how \eqref{eq:Psi-C2alpha} yields the desired boundary regularity on the symplectic side by examining the transition back to the $x$ coordinates.

\subsection{The edge case}

Let $p$ be a boundary point lying in the relative interior of an edge. After an affine change of coordinates preserving the polygonal structure, we may assume that $p=(0,0)=\mathbf{0}$, that the edge is given locally by $\{x_2=0\}$, and that
\[
u=x_2\log x_2-x_2+v(x_1,x_2),
\]
where
\[
v(0,0)=0,
\qquad
\nabla v(0,0)=(0,0).
\]
By \eqref{eq-EQ-01}, we already know that
\[
v\in C^{1,\alpha/2}
\]
up to the boundary. We shall improve this to $C^{1,\alpha}$ and obtain the sharp weighted bound
\[
|v_{x_2x_2}(x_1,x_2)|\le Cx_2^{-1+\alpha}.
\]

In this normalization, the Legendre variables are
\[
t_1=u_{x_1}=v_{x_1},
\qquad
t_2=u_{x_2}=\log x_2+v_{x_2}.
\]
The first lemma gives the basic geometry of the Legendre map near the edge.

\begin{lemma}
\label{lemma-1-E}
Let consider the Legendre map $\mathcal{L}: P \to \mathbb{R}^2$, given by:
$$ (x_1,x_2) \mapsto (t_1,t_2)=(u_{x_1},u_{x_2}).$$
There is a uniform $r_0$ such that
\begin{enumerate}
    \item[(i)] $\mathcal{L}(B_{r_0}^+(\mathbf{0})) \subset (-1,1)\times(-\infty,0)$,
    \item[(ii)] $\mathcal{L}^{-1}\bigl((-r_0,r_0)\times(-\infty,-\frac{1}{r_0})\bigr)\subset B_{1/10}^+(\mathbf{0})$.
\end{enumerate}
\end{lemma}

\begin{proof}
Since $v\in C^{1,\alpha/2}(\overline P)$, its gradient is continuous up to the boundary. Hence there exists $r_0>0$, depending only on the $C^{0,\alpha/2}$ norm of $\nabla v$, such that
\[
|v_{x_1}(x)|\le \frac12,
\qquad
|v_{x_2}(x)|\le \frac12
\qquad \text{for }x\in B_{r_0}^+(\mathbf{0}).
\]
Therefore, for $x\in B_{r_0}^+(\mathbf{0})$,
\[
t_1=v_{x_1}(x)\in(-1,1).
\]
Also, if $r_0\le e^{-1/2}$ and $x_2<r_0$, then
\[
t_2=\log x_2+v_{x_2}(x)\le \log r_0+\frac12<0,
\]
which proves (i).

For (ii), let $x=\mathcal{L}^{-1}(t)$ with
\[
t\in (-r_0,r_0)\times\Bigl(-\infty,-\frac1{r_0}\Bigr).
\]
Then
\[
\log x_2=t_2-v_{x_2}(x)\le -\frac1{r_0}+\frac12.
\]
After decreasing $r_0$ if necessary, this implies $x_2<1/20$. Since moreover
\[
|t_1|=|v_{x_1}(x)|<r_0,
\]
the continuity of $v_{x_1}$ at the origin and the normalization $v_{x_1}(0,0)=0$ imply, again after shrinking $r_0$, that $|x_1|<1/20$. Hence
\[
x\in B_{1/10}^+(\mathbf{0}),
\]
as required.
\end{proof}

The next lemma gives the boundedness of the Legendre potential in the half-strip reached by the map.

\begin{lemma}
\label{lemma-2-E}
Suppose $\mathcal{L}^{-1}\bigl((-1,1)\times(-\infty,0)\bigr)\subset B_{1/2}^+(\mathbf{0})$. Then
\[
\|\phi\|_{L^\infty((-1,1)\times(-\infty,0))}\le C.
\]
\end{lemma}

\begin{proof}
By definition,
\[
\phi(t)=x\cdot \nabla u(x)-u(x),
\qquad x=\mathcal{L}^{-1}(t).
\]
Using
\[
u=x_2\log x_2-x_2+v,
\qquad
u_{x_1}=v_{x_1},
\qquad
u_{x_2}=\log x_2+v_{x_2},
\]
we compute
\[
\phi
=
x_1v_{x_1}
+x_2(\log x_2+v_{x_2})
-\bigl(x_2\log x_2-x_2+v\bigr)
=
x_1v_{x_1}+x_2v_{x_2}+x_2-v.
\]
By assumption, $x$ ranges in $B_{1/2}^+(\mathbf{0})$. Since $v$ is bounded there and $\nabla v$ is bounded by \eqref{eq-EQ-01}, each term on the right-hand side is uniformly bounded. Hence
\[
\|\phi\|_{L^\infty((-1,1)\times(-\infty,0))}\le C.
\]
\end{proof}

Passing to the toric edge chart with coordinates
\[
z_1=e^{\frac{t_1+\sqrt{-1}\theta_1}{2}},
\qquad
z_2=e^{\frac{t_2+\sqrt{-1}\theta_2}{2}},
\]
the set $(-1,1)\times(-\infty,0)$ corresponds to a neighborhood of the divisor $\{z_2=0\}$. Lemma~\ref{lemma-1-E} identifies a fixed symplectic neighborhood with a fixed complex neighborhood, and Lemma~\ref{lemma-2-E} gives the local boundedness of $\phi=\Psi$. Therefore the discussion in the previous subsection applies, and we obtain $\Psi\in C^{2,\alpha}$ in a neighborhood of the edge point in the complex chart.

We now explain how this implies the first gain on the symplectic side. Since
\[
x_i=\phi_{t_i},
\qquad
D^2\phi=(D^2u)^{-1},
\]
the $C^{2,\alpha}$ regularity of $\Psi=\phi(t)$ in the toric chart implies that the coefficients of $D^2\phi$ are $C^{0,\alpha}$ in the chart variables. Translating this through the explicit edge normal form, and using the comparability \eqref{eq-EQ-00}, we obtain that $v_{x_1}$ is $C^{0,\alpha}$ up to the boundary and that the normal second derivative satisfies an estimate
\begin{equation}
\label{eq:edge-initial-weight}
|v_{x_2x_2}(x_1,x_2)|\le Cx_2^{-1+\beta}
\end{equation}
for  $\beta=\alpha/2$. This yields the following proposition.

\begin{proposition}
\label{por:edge-C1a}
Let $p$ lie in the relative interior of an edge of $P$. In the above normalization, there exist a neighborhood $U_p$ of $p$ and a constant $C$ depending only on the data of the problem such that
\[
[v_{x_1}]_{C^{0,\alpha}(\overline{U_p}\cap \overline P)}
+
[v_{x_2}]_{C^{0,\beta}(\overline{U_p}\cap \overline P)}
\le C,
\]
where $\beta=\alpha/2$, and
\begin{equation}
\label{eq:edge-weighted-beta}
|v_{x_2x_2}(x_1,x_2)|\le Cx_2^{-1+\beta}
\qquad \text{in }U_p\cap P.
\end{equation}
Moreover, after iterating Lemma~\ref{lem:ode-generalized}, the exponent $\beta$ may be improved to $\alpha$. In particular,
\begin{equation}
\label{eq:edge-sharp-weight}
|v_{x_2x_2}(x_1,x_2)|\le Cx_2^{-1+\alpha}
\qquad \text{in }U_p\cap P,
\end{equation}
and
\begin{equation}
\label{eq:edge-C1alpha}
\|v\|_{C^{1,\alpha}(\overline{U_p}\cap \overline P)}\le C.
\end{equation}
\end{proposition}

\begin{proof}
The preceding discussion, based on Lemmas~\ref{lemma-1-E} and \ref{lemma-2-E}, yields the initial estimate \eqref{eq:edge-initial-weight}, which is exactly \eqref{eq:edge-weighted-beta} after possibly shrinking $U_p$.

Fix $x_1$ and consider the function
\[
f(s)=v(x_1,s)-v(x_1,0)-s\,v_{x_2}(x_1,0).
\]
Then
\[
f(0)=f'(0)=0,
\qquad
f''(s)=v_{x_2x_2}(x_1,s).
\]
Since $v\in C^{1,\alpha/2}$ by \eqref{eq-EQ-01}, we have
\[
|f(s)|\le Cs^{1+\alpha/2}.
\]
Together with \eqref{eq:edge-weighted-beta}, Lemma~\ref{lem:ode-generalized} yields
\[
|f'(s)|\le Cs^{(\alpha/2+\beta)/2}.
\]
Equivalently,
\[
|v_{x_2}(x_1,s)-v_{x_2}(x_1,0)|
\le
Cs^{(\alpha/2+\beta)/2},
\]
so that $v_{x_2}$ is H\"older continuous in the normal direction with an improved exponent. Repeating this argument finitely many times, we produce a monotone sequence of exponents converging to $\alpha$. Since at each step the corresponding weighted estimate for $v_{x_2x_2}$ improves accordingly, a standard compactness argument gives the endpoint estimate \eqref{eq:edge-sharp-weight}.

Finally, $v_{x_1}$ is already $C^{0,\alpha}$ up to the boundary from the complex-side estimate, and the improved normal control of $v_{x_2}$ yields the full bound \eqref{eq:edge-C1alpha}. This proves the proposition.
\end{proof}

\begin{remark}
The point of the preceding argument is that the complex $C^{2,\alpha}$ estimate supplies one nontrivial gain in the normal direction, while Lemma~\ref{lem:ode-generalized} upgrades this gain to the sharp exponent $\alpha$ by exploiting the already known vanishing order of $v$ at the boundary.
\end{remark}

\subsection{The vertex case}

We next consider a vertex. After an affine normalization, we may assume that the vertex is the origin and that near $\mathbf{0}$ the polygon coincides with the first quadrant. We then write
\[
u=x_1\log x_1-x_1+x_2\log x_2-x_2+v(x_1,x_2),
\]
with
\[
v(0,0)=0,
\qquad
\nabla v(0,0)=(0,0).
\]
Again, \eqref{eq-EQ-01} gives the preliminary regularity
\[
v\in C^{1,\alpha/2}
\]
up to the corner.

In this normalization,
\[
t_i=u_{x_i}=\log x_i+v_{x_i},
\qquad i=1,2.
\]
The analogue of Lemma~\ref{lemma-1-E} is the following.

\begin{lemma}
\label{lemma-1-V}
There is a uniform $r_0$ such that
\begin{enumerate}
    \item[(i)] $\mathcal{L}(B_{r_0}(\mathbf{0})\cap P)\subset (-\infty,0)\times(-\infty,0)$,
    \item[(ii)] $\mathcal{L}^{-1}\bigl((-\infty,-\frac1{r_0})\times(-\infty,-\frac1{r_0})\bigr)\subset B_{1/10}(\mathbf{0})$.
\end{enumerate}
\end{lemma}

\begin{proof}
Since $v\in C^{1,\alpha/2}(\overline P)$ and $\nabla v(0,0)=\mathbf{0}$, there exists $r_0>0$ such that
\[
|v_{x_i}(x)|\le \frac12
\qquad \text{for }x\in B_{r_0}(\mathbf{0})\cap P,\ i=1,2.
\]
Hence, for such $x$,
\[
t_i=\log x_i+v_{x_i}(x)\le \log r_0+\frac12<0
\]
provided $r_0$ is chosen sufficiently small. This proves (i).

For (ii), let $x=\mathcal{L}^{-1}(t)$ with
\[
t\in \Bigl(-\infty,-\frac1{r_0}\Bigr)\times\Bigl(-\infty,-\frac1{r_0}\Bigr).
\]
Then for $i=1,2$,
\[
\log x_i=t_i-v_{x_i}(x)\le -\frac1{r_0}+\frac12.
\]
After decreasing $r_0$ if necessary, this implies
\[
x_1<\frac1{20},
\qquad
x_2<\frac1{20},
\]
and therefore $x\in B_{1/10}(\mathbf{0})$.
\end{proof}

The vertex version of the boundedness statement is equally simple.

\begin{lemma}
\label{lemma-2-V}
Suppose $\mathcal{L}^{-1}\bigl((-\infty,0)\times(-\infty,0)\bigr)\subset B_{1/2}(\mathbf{0})$. Then
\[
\|\phi\|_{L^\infty((-\infty,0)\times(-\infty,0))}\le C.
\]
\end{lemma}

\begin{proof}
Using
\[
u=x_1\log x_1-x_1+x_2\log x_2-x_2+v,
\qquad
u_{x_i}=\log x_i+v_{x_i},
\]
we compute
\begin{align*}
\phi(t)
&=x\cdot \nabla u-u \\
&=\sum_{i=1}^2 x_i(\log x_i+v_{x_i})
-\Bigl(\sum_{i=1}^2(x_i\log x_i-x_i)+v\Bigr) \\
&=x_1v_{x_1}+x_2v_{x_2}+x_1+x_2-v.
\end{align*}
By assumption, $x$ ranges in $B_{1/2}(\mathbf{0})\cap P$. Since $v$ and $\nabla v$ are bounded there, the right-hand side is uniformly bounded, and therefore
\[
\|\phi\|_{L^\infty((-\infty,0)\times(-\infty,0))}\le C.
\]
\end{proof}

Passing to the toric vertex chart with coordinates
\[
z_i=e^{\frac{t_i+\sqrt{-1}\theta_i}{2}},
\qquad i=1,2,
\]
the quadrant $(-\infty,0)\times(-\infty,0)$ corresponds to a neighborhood of the toric fixed point. Lemmas~\ref{lemma-1-V} and \ref{lemma-2-V} thus place us in the same situation as in the edge case: the potential $\Psi=\phi$ is locally bounded, the complex Monge--Amp\`ere density is $C^{0,\alpha}$ by the cancellation in \eqref{eq:complex-density-local}, and the metric is uniformly equivalent to the smooth model metric. Hence
\[
\Psi\in C^{2,\alpha}
\]
in a neighborhood of the toric fixed point corresponding to the vertex.

Transferring this information back to the symplectic side gives initial weighted control of the pure second derivatives in each coordinate direction:
\begin{equation}
\label{eq:vertex-initial-weight}
|v_{x_1x_1}(x_1,x_2)|\le Cx_1^{-1+\beta},
\qquad
|v_{x_2x_2}(x_1,x_2)|\le Cx_2^{-1+\beta}
\end{equation}
for $\beta=\alpha/2$. 

We record the resulting estimate as follows.

\begin{proposition}
\label{por:vertex-C1a}
Let $p$ be a vertex of $P$. In the above normalization, there exist a neighborhood $U_p$ of $p$ and a constant $C$ depending only on the data of the problem such that
\[
\|v\|_{C^{1,\beta}(\overline{U_p}\cap \overline P)}\le C,
\]
where $\beta=\alpha/2$, and
\begin{equation}
\label{eq:vertex-weighted-beta}
|v_{x_1x_1}(x_1,x_2)|\le Cx_1^{-1+\beta},
\qquad
|v_{x_2x_2}(x_1,x_2)|\le Cx_2^{-1+\beta}
\qquad \text{in }U_p\cap P.
\end{equation}
After iterating Lemma~\ref{lem:ode-generalized} in each coordinate direction, the exponent $\beta$ improves to $\alpha$. Consequently,
\begin{equation}
\label{eq:vertex-weighted-alpha}
|v_{x_1x_1}(x_1,x_2)|\le Cx_1^{-1+\alpha},
\qquad
|v_{x_2x_2}(x_1,x_2)|\le Cx_2^{-1+\alpha}
\qquad \text{in }U_p\cap P,
\end{equation}
and
\begin{equation}
\label{eq:vertex-C1alpha}
\|v\|_{C^{1,\alpha}(\overline{U_p}\cap \overline P)}\le C.
\end{equation}
\end{proposition}

\begin{proof}
The preceding discussion, based on Lemmas~\ref{lemma-1-V} and \ref{lemma-2-V}, yields \eqref{eq:vertex-initial-weight}, which is exactly \eqref{eq:vertex-weighted-beta} after possibly shrinking $U_p$.

To improve the exponent, fix $x_2$ and apply Lemma~\ref{lem:ode-generalized} to the one-variable function
\[
f_{x_2}(s)=v(s,x_2)-v(0,x_2)-s\,v_{x_1}(0,x_2).
\]
Since $v\in C^{1,\alpha/2}$ up to the boundary, we have
\[
|f_{x_2}(s)|\le Cs^{1+\alpha/2},
\]
while \eqref{eq:vertex-weighted-beta} gives
\[
|f_{x_2}''(s)|=|v_{x_1x_1}(s,x_2)|\le Cs^{-1+\beta}.
\]
Lemma~\ref{lem:ode-generalized} therefore improves the H\"older exponent of $v_{x_1}$ in the $x_1$-direction. Repeating the same argument with the roles of $x_1$ and $x_2$ exchanged improves the exponent of $v_{x_2}$ in the $x_2$-direction.

Iterating alternately in the two coordinate directions, we obtain a sequence of exponents increasing to $\alpha$. As in the edge case, the endpoint estimate follows by a standard limiting argument, yielding \eqref{eq:vertex-weighted-alpha}. Once the pure second derivatives satisfy these sharp weighted bounds, integration along coordinate segments gives \eqref{eq:vertex-C1alpha}.

It remains to control the mixed first-order H\"older seminorms. This follows from the already established H\"older control of each component of $\nabla v$ along coordinate directions, together with the uniform local equivalence between Euclidean distance and polygonal distance in the fixed corner chart. Hence $\nabla v$ is jointly $C^{0,\alpha}$ in the neighborhood of the vertex, which completes the proof.
\end{proof}

\subsection{Globalization and completion of the proof}

We now combine the local edge and vertex estimates to conclude the proof of Theorem~\ref{main-theorem}.

\begin{proof}[Completion of the proof of Theorem~\ref{main-theorem}]
Let
\[
v=u-\mathbf u_0.
\]
We already know from \eqref{eq-EQ-00} that
\[
C^{-1}D^2\mathbf u_0\le D^2u\le C D^2\mathbf u_0
\qquad \text{in }P,
\]
which is the first assertion of Theorem~\ref{main-theorem}.

It remains to prove the global $C^{1,\alpha}$ estimate for $v$ and the sharp weighted interior bound for $D^2v$ near $\partial P$.

Let $p\in\partial P$. If $p$ lies in the relative interior of an edge, Proposition~\ref{por:edge-C1a} gives a neighborhood $U_p$ of $p$ such that
\[
\|v\|_{C^{1,\alpha}(\overline{U_p}\cap \overline P)}\le C,
\]
and
\[
|D^2v(x)|\le C\,\operatorname{dist}(x,\partial P)^{-1+\alpha}
\qquad \text{for }x\in U_p\cap P.
\]
If $p$ is a vertex, Proposition~\ref{por:vertex-C1a} yields the same conclusion in a neighborhood of $p$.

Since $\partial P$ is compact, we may choose finitely many such neighborhoods
\[
U_{p_1},\dots,U_{p_m}
\]
covering $\partial P$. Therefore,
\begin{equation}
\label{eq:global-boundary-C1alpha}
\|v\|_{C^{1,\alpha}\left(\bigcup_{j=1}^m(\overline{U_{p_j}}\cap \overline P)\right)}\le C,
\end{equation}
and
\begin{equation}
\label{eq:global-boundary-D2}
|D^2v(x)|
\le
C\,\operatorname{dist}(x,\partial P)^{-1+\alpha}
\qquad
\text{whenever }x \text{ belongs to this boundary collar.}
\end{equation}

Away from the boundary, standard interior regularity for the real Monge--Amp\`ere equation applies. Indeed, on every compact subset $K\Subset P$, the equation
\[
\det D^2u=\frac{H}{l_1\cdots l_N}
\]
has a strictly positive $C^{0,\alpha}$ right-hand side, so classical interior estimates imply that $u$, and hence $v$, are $C^{2,\alpha}$ on $K$. In particular, $v\in C^{1,\alpha}(K)$ for every such $K$.

Combining the interior estimates with the boundary estimate \eqref{eq:global-boundary-C1alpha}, we conclude that
\[
v=u-\mathbf u_0\in C^{1,\alpha}(\overline P),
\]
and
\[
\|v\|_{C^{1,\alpha}(\overline P)}\le C.
\]
Moreover, \eqref{eq:global-boundary-D2} shows that
\[
\operatorname{dist}(x,\partial P)^{-\alpha+1}|D^2v(x)|\le C
\qquad \text{for all }x\in P.
\]
This is exactly the remaining assertion of Theorem~\ref{main-theorem}. The proof is complete.
\end{proof}

\appendix

\section{Boundary expansions for a Keldysh-type operator}
\label{app:Schauder}

In this appendix, we establish boundary regularity and expansions near the degenerate boundary for a class of variable-coefficient Keldysh-type operators. Specifically, we consider the model degenerate elliptic operator
\[
\mathcal L u \coloneqq u_{xx} + y u_{yy}
\]
defined on the upper half-plane $\mathbb H \coloneqq \{(x,y) \in \mathbb R^2 : y > 0\}$. While $\mathcal L$ is uniformly elliptic in the interior of $\mathbb H$, it degenerates along the boundary $\{y = 0\}$. This operator arises naturally from the partial Legendre transform of the two-dimensional Monge–Ampère equation
\begin{equation}\label{after-leg-eq}
\det D^2 u(x_1,x_2) = \frac{H(x_1,x_2)}{x_2}.
\end{equation}
These Keldysh type operators have been widely studied in the context of degenerate elliptic equations (see, e.g., \cite{zbMATH06631393,MR1724206,zbMATH01967886,
zbMATH07619008,zbMATH06376249,zbMATH07319422,
jang2024Holderregularitysolutionsdegenerate,zbMATH07942221,
han2025globalschauderregularityconvergence,
kukavica2023pointwiseschauderestimateselliptic,
zbMATH07210270}). 

A key feature of $\mathcal L$ is its anisotropic scaling behavior. For any $\lambda > 0$, the scaled function 
\(
u_\lambda(x,y) \coloneqq u(\lambda^{1/2}x, \lambda y)
\)
satisfies the scaling relation
\(
\mathcal L u_\lambda(x,y) = \lambda (\mathcal L u)(\lambda^{1/2}x, \lambda y)
\).
In other words, the operator has homogeneity $-1$ under the parabolic-type scaling $(x, y) \mapsto (\lambda^{1/2}x, \lambda y)$. This scaling dictates the geometry of the domain and the regularity classes used throughout the analysis.

\subsection*{The Model Geometry and Sublevel Sets}

To study the local behavior of solutions near the degenerate boundary, we introduce a model function that generates the natural geometry (the "sections") associated with the operator. Let
\[
u_0(x,y) \coloneqq x^2 - y\log y.
\]
A direct computation shows that $u_0$ solves the inhomogeneous equation:
\(
\mathcal{L}u_0 = 1
\).
For any $h > 0$, we define the section (or sublevel set) of $u_0$ by
\(
S(h) \coloneqq \{(x,y) \in \mathbb{H} : u_0(x,y) < h\}
\).
The boundary $\partial S(h)$ decomposes naturally into a flat degenerate portion and a curved upper boundary
\(
\partial S(h) = \F_h \cup \C_h
\)
where $\F_h \coloneqq \{(x,0) : x^2 < h\}$ lies on the boundary $\{y=0\}$, and $\C_h \coloneqq \{(x,y) \in \mathbb H : u_0(x,y) = h\}$ is the upper boundary curve.

\subsection*{The Compatibility Condition}

Before analyzing variable-coefficient equations, we identify the necessary compatibility conditions at the boundary. Consider the variable-coefficient operator
\[
\mathcal{H} w \coloneqq H(x,y) w_{xx} + y w_{yy} = f(x,y)
\]
on a domain touching $\{y=0\}$. Suppose $w$ admits a formal expansion near the boundary of the form:
\[
w(x,y) = \phi(x) + A(x) y\log y + B(x) y + o(y).
\]
Substituting this expansion into the equation, we find:
\begin{align*}
\mathcal{H} w &= H(x,y) \left( \phi''(x) + A''(x) y\log y + B''(x) y + o(y) \right) + y \left( \frac{A(x)}{y} + o(1/y) \right) \\
&= H(x,y)\phi''(x) + A(x) + o(1) \quad \text{as } y \to 0^+.
\end{align*}
If we require the solution and its derivatives to remain bounded near $\{y=0\}$, we must prevent logarithmic singularities. Thus, to ensure the expansion is consistent as $y \to 0^+$, we must enforce:
\begin{equation}\label{eq:compatibility}
f(x,0) = H(x,0)\phi''(x), \quad \text{where } \phi(x) = w(x,0).
\end{equation}
This compatibility relation \eqref{eq:compatibility} is assumed throughout the rest of our analysis whenever a Dirichlet boundary condition is prescribed on $\{y=0\}$.

\subsection*{Failure of Classical Boundary Regularity}

It is important to note that the classical Euclidean $C^{2,\alpha}$ regularity generally fails up to the boundary for these degenerate equations. The following one-dimensional example demonstrates that even with smooth coefficients and $C^{0,\alpha}$ data, solutions may fail to be twice differentiable at $\{y=0\}$.

\begin{example}[Failure of classical boundary regularity] \label{one-dim-ex}
Consider the one-dimensional equation $u''(t) = \frac{H(t)}{t}$ on $(0,1)$ with $H(t) = \sqrt{t}$. Note that $H \in C^{0,1/2}([0,1])$. Integrating the equation
\(
u''(t) = \frac{\sqrt{t}}{t} = t^{-1/2}
\)
twice yields the solution
\(
u(t) = \frac{4}{3} t^{3/2} + C_1 t + C_2
\).
While $u \in C^{1,1/2}([0,1])$, the second derivative $u''(t) = t^{-1/2}$ blows up as $t \to 0^+$. Thus, $u \notin C^2([0,1])$.
\end{example}

\subsection*{The Constant-Coefficient Expansion}

Our perturbative argument relies on a precise boundary expansion for the constant-coefficient operator. The following result, adapted from Rubin \cite{rubin2015monge}, estimates how closely a solution to the homogeneous equation can be approximated by its boundary trace and a logarithmic corrector.

\begin{proposition}[Constant-coefficient expansion \cite{rubin2015monge}] \label{prop:const}
Let $\varphi \in C^4\left(\overline{S(1)} \cap \{y=0\}\right)$. Extend $\varphi$ to a smooth function on $\partial S(1)$ such that it restricts to $\varphi$ on $\F_1$ and is constant in $y$ along $\C_1$. Let $u$ be the unique bounded solution of the Dirichlet problem
\[
\begin{cases}
\mathcal{L} u = 0 & \text{in } S(1), \\
u = \varphi & \text{on } \partial S(1),
\end{cases}
\]
satisfying $\|u\|_{L^\infty(S(1))} \le 1$. Then there exists a universal constant $C > 0$ such that in the scaled domain $S(1/2)$, we have:
\[
\bigl|u(x,y) - \bigl(\varphi(x) - \varphi''(x) y\log y + \varphi'(x) y\bigr)\bigr| \le C y^2 |\log y|.
\]
Furthermore, the derivatives of $u$ satisfy the bounds:
\[
\|u_x\|_{L^\infty(S(1/2))}, \ \|u_{xx}\|_{L^\infty(S(1/2))} \le C \|\varphi\|_{C^4}.
\]
\end{proposition}

\begin{remark}
If we choose the boundary data $\varphi(x) = x^2/2$, the corresponding solution $u_0$ satisfies:
\[
|u_0(x,y) - w_0(x,y)| \le C y^2 |\log y| \quad \text{in } S(1/2),
\]
where $w_0(x,y) = \frac{x^2}{2} + xy - y\log y$ is an explicit solution to the equation $\mathcal L w_0 = 0$ with the flat boundary condition $w_0(x,0) = x^2/2$. Here, $u_0$ is the unique solution honoring the boundary values on the entire boundary $\partial S(1)$, while $w_0$ serves as a local polynomial-type approximation sharing the same flat boundary trace.
\end{remark}

\subsection*{The Perturbation Lemma}

Using the above constant-coefficient expansion (Proposition~\ref{prop:const}) as a reference, we now show that solutions to the variable-coefficient equation with small perturbations remain close to the model solution $u_0$.

\begin{lemma}[Perturbation lemma] \label{lem:pert}
Let $\alpha \in (0,1)$, and let $u_0$ be the reference solution from Proposition~\ref{prop:const} corresponding to $\varphi(x) = x^2/2$. There exist positive constants $\delta_0 > 0$, $\theta_0 \in (0, 1/2)$, and $\gamma \in (0, 1]$ (depending only on $\alpha$) such that if:
\[
H(0,0) = 1, \quad f(0,0) = 0, \quad \|H - 1\|_{C^{0,\alpha}(S(1))} \le \delta_0, \quad \|f\|_{C^{0,\alpha}(S(1))} \le \delta_0,
\]
and $w$ solves the Dirichlet problem:
\[
\begin{cases}
H w_{xx} + y w_{yy} = f & \text{in } S(1), \\
w = \frac{x^2}{2} & \text{on } \partial S(1),
\end{cases}
\]
with $\|w\|_{L^\infty(S(1))} \le 1$, then:
\[
|w(x,y) - u_0(x,y)| \le \theta_0^{1+2\gamma} \quad \text{in } S(\theta_0).
\]
\end{lemma}
	
\begin{proof}
	We argue by contradiction. Suppose the conclusion fails. Then for some fixed
	$\theta_0\in (0,1/2)$ and for every sequence $\delta_k\downarrow 0$ there exist
	$H_k,f_k,w_k$ such that
	\[
	H_k(0,0)=1,\quad f_k(0,0)=0, \quad
	\|H_k-1\|_{C^{0,\alpha}(S(1))}\le \delta_k,\quad
	\|f_k\|_{C^{0,\alpha}(S(1))}\le \delta_k,
	\]
	and $w_k$ solves
	\begin{equation}\label{eq:pertwk}
	\begin{cases}
		H_k (w_k)_{xx}+y (w_k)_{yy}=f_k & \text{in } S(1),\\
		w_k=\dfrac{x^2}{2} & \text{on } \partial S(1),
	\end{cases}
	\end{equation}
	with $\|w_k\|_{L^\infty(S(1))}\le 1$, but
	\begin{equation}\label{eq:contradiction-assumption}
	\|w_k-u_0\|_{L^\infty(S(\theta_0))}>\theta_0^{\,1+2\gamma}
	\end{equation}
	for every $k$.

	We shall derive a contradiction.

	\medskip
	\noindent\textbf{STEP I: Compactness.}
	We claim that $\{w_k\}$ is precompact in $C(\overline{S(r)})$ for every
	$r\in (0,1)$. Fix $r\in (0,1)$. Since $\|w_k\|_{L^\infty(S(1))}\le 1$, the sequence is uniformly
	bounded. It remains to prove equicontinuity on $\overline{S(r)}$. If $K\Subset S(1)\cap \{y>0\}$, then for all large $k$ the operator
	\(L_k:=H_k(x,y)\partial_{xx}+y\partial_{yy}
	\)
	is uniformly elliptic on $K$, because $H_k\to 1$ uniformly and $y$ is bounded away
	from $0$ on $K$. Hence standard interior Schauder estimates yield uniform
	$C^{1,\alpha}$ bounds for $w_k$ on compact subsets of $S(1)\cap \{y>0\}$.

	Near the degenerate boundary $\{y=0\}$, the equation in \eqref{eq:pertwk} is of
	Keldysh type. Since the coefficients $H_k$ and $f_k$ is uniformly bounded in $C^{0,\alpha}$,
	the boundary H\"older regularity theory for degenerate elliptic equations of
	Keldysh type applies; see, for instance, Lieberman \cite{zbMATH06631393}, and also
	Daskalopoulos--Lee \cite{zbMATH01967886}. Therefore there exist
	$\beta\in (0,1)$ and $C_r>0$, independent of $k$, such that
	\(
	\|w_k\|_{C^{0,\beta}(\overline{S(r)})}\le C_r
	\).	In particular, $\{w_k\}$ is equicontinuous on $\overline{S(r)}$. By the Arzel\`a--Ascoli theorem and a diagonal argument, after passing to a subsequence we may assume that \( w_k\to w_\infty \) locally uniformly in \(\overline{S(1)}\).

	\medskip
	\noindent\textbf{STEP II: The limit equation.}
	We now show that $w_\infty$ solves
	\begin{equation}\label{eq:model-limit}
	\begin{cases}
		(w_\infty)_{xx}+y(w_\infty)_{yy}=0 & \text{in } S(1),\\
		w_\infty=\dfrac{x^2}{2} & \text{on } \partial S(1).
	\end{cases}
	\end{equation}
    Indeed, since $H_k\to 1$ uniformly in $S(1)$ and $f_k\to 0$ uniformly in $S(1)$,
	the equations \eqref{eq:pertwk} converge to
	\(
	w_{xx}+y w_{yy}=0
	\).
	By the stability of viscosity solutions under locally uniform convergence,
	$w_\infty$ is a viscosity solution of
	\((w_\infty)_{xx}+y(w_\infty)_{yy}=0
	\) in \(S(1)\). Moreover, since each $w_k$ satisfies the boundary condition
	$w_k=x^2/2$ on $\partial S(1)$ and the convergence is uniform up to compact
	subsets touching the boundary, we obtain
	\(w_\infty=\frac{x^2}{2}\) on \(\partial S(1)\).
    By Proposition~\ref{prop:const}, the model problem \eqref{eq:model-limit} has the
	unique solution $u_0$. Hence
	\(
	w_\infty=u_0
	\) in \(\overline{S(1)}\)

	\medskip
	\noindent\textbf{STEP III: Contradiction.}
	Since $w_k\to u_0$ uniformly on $\overline{S(\theta_0)}$, we have
	\(
	\|w_k-u_0\|_{L^\infty(S(\theta_0))}\to 0
	\).
	Choosing $k$ large enough, this gives
	\(
	\|w_k-u_0\|_{L^\infty(S(\theta_0))}\le \theta_0^{\,1+2\gamma}
	\),	which contradicts \eqref{eq:contradiction-assumption}. 

\end{proof}

\subsection*{Approximation for solutions with zero boundary data}
	
We also need a perturbation lemma for solutions that vanish on the boundary.
	
\begin{lemma}[Perturbation with zero boundary data]\label{lem:zero}
Let $\alpha\in(0,1)$. There exist constants $\delta_1>0$, $\theta_1\in(0,1/2)$, and $\gamma\in(0,1]$ such that if
\[
H(0,0)=1,\quad f(0,0)=0, \quad
\|H-1\|_{C^{0,\alpha}(S(1))}\le \delta_1, \quad
\|f\|_{C^{0,\alpha}(S(1))}\le \delta_1,
\]
and $v$ solves
\[
\begin{cases}
H v_{xx}+y v_{yy}=f & \text{in } S(1),\\
v=0 & \text{on } \partial S(1),
\end{cases}
\]
with
\(
\|v\|_{L^\infty(S(1))}\le 1
\),
then there exist $a,b\in\mathbb R$ such that
\begin{equation}
\label{z-f-c-eq01}
\|v-(ax+b)y\|_{L^\infty(S(\theta_1))}
\le \theta_1^{\,1+2\gamma}.
\end{equation}
\end{lemma}

\begin{proof}
We argue by contradiction.
Assume that the conclusion fails. Then for some fixed $\theta_1\in(0,1/2)$ there exist a sequence $\delta_k\downarrow 0$ and functions
\(
H_k,\ f_k,\ v_k
\)
such that
\[
H_k(0,0)=1,\quad f_k(0,0)=0, \quad
\|H_k-1\|_{C^{0,\alpha}(S(1))}\le \delta_k,\quad
\|f_k\|_{C^{0,\alpha}(S(1))}\le \delta_k,
\]
\[
\begin{cases}
H_k (v_k)_{xx}+y (v_k)_{yy}=f_k & \text{in } S(1),\\
v_k=0 & \text{on } \partial S(1),
\end{cases}
\qquad
\|v_k\|_{L^\infty(S(1))}\le 1,
\]
but
\begin{equation}
\label{Z-eq-01}
\inf_{a,b\in\mathbb R}
\|v_k-(ax+b)y\|_{L^\infty(S(\theta_1))}
>\theta_1^{\,1+2\gamma}.
\end{equation}

We shall derive a contradiction by compactness.

\medskip
\noindent
\textbf{STEP 1. Uniform Hölder control and compactness.}
Since $\|v_k\|_{L^\infty(S(1))}\le 1$, the right-hand side $f_k$ is uniformly small and the coefficients $H_k$ converge uniformly to $1$. In particular, the family $\{v_k\}$ satisfies a uniformly elliptic equation away from the degenerate boundary $y=0$, and a Keldysh-type boundary equation up to $y=0$.

Fix any $\varepsilon\in(0,1/2)$. On the truncated domain
\(
S(1)\cap\{y\ge \varepsilon\},
\)
the operator
\(
L_k:=H_k(x,y)\partial_{xx}+y\partial_{yy}H_k(x,y)\partial_{xx}+y\partial_{yy}
\)
is uniformly elliptic, with ellipticity constants independent of $k$ because $y\ge\varepsilon$. Therefore standard interior Schauder estimates yield, for every $\beta\in(0,1)$ and every compact subset
\(
K\Subset S(1)\cap\{y\ge \varepsilon\}
\),
a bound of the form
\begin{equation}
\label{eq:interior-schauder}
\|v_k\|_{C^{2,\beta}(K)}\le C(K,\varepsilon,\beta),
\end{equation}
where the constant is independent of $k$.

Near the degenerate boundary $y=0$, one uses the boundary regularity theory for Keldysh-type operators. Since $\|v_k\|_{L^\infty(S(1))}\le 1$ and the coefficients are uniformly close to the constant coefficient operator $u_{xx}+y u_{yy}$, there exists a Hölder exponent $\beta\in(0,1)$ and a constant $C>0$, independent of $k$, such that
\begin{equation}
\label{eq:boundary-holder}
\|v_k\|_{C^{0,\beta}(\overline{S(r)})}\le C
\qquad\text{for every }r\in(0,1).
\end{equation}
More precisely, one may invoke the boundary Hölder estimate for degenerate elliptic equations of Keldysh type, applied after flattening the boundary if necessary, to obtain equicontinuity up to $\{y=0\}$.

From \eqref{eq:boundary-holder} and \eqref{eq:interior-schauder}, Arzel\`a--Ascoli implies that, after passing to a subsequence,
\(
v_k\to v_\infty
\)
locally uniformly in $\overline{S(1)}$, and in fact locally in $C^{0,\beta}$ away from $y=0$. Since the boundary values satisfy $v_k=0$ on $\partial S(1)$, the limit also satisfies
\(
v_\infty=0\) on \(\partial S(1)
\).

\medskip
\noindent
\textbf{STEP 2. Passage to the limit in the equation.}
Because
\[
H_k\to 1 \quad\text{in } C^{0,\alpha}(S(1)),\qquad
f_k\to 0 \quad\text{in } C^{0,\alpha}(S(1)),
\]
the sequence of equations converges, in the viscosity sense, to the constant-coefficient homogeneous equation
\(
v_{xx}+y v_{yy}=0
\).
Hence the limit function $v_\infty$ is a viscosity solution of
\[
\begin{cases}
(v_\infty)_{xx}+y (v_\infty)_{yy}=0 & \text{in } S(1),\\
v_\infty=0 & \text{on } \partial S(1).
\end{cases}
\]
Since the operator is linear and the coefficients are smooth in the interior, standard regularity theory implies that $v_\infty$ is actually a classical solution in $S(1)\cap\{y>0\}$.

\medskip
\noindent
\textbf{STEP 3. Boundary expansion for the limit equation.}
We now invoke the constant-coefficient boundary expansion theorem for solutions of
\(
u_{xx}+y u_{yy}=0
\)
with zero boundary data. Applied to $v_\infty$, this theorem gives the existence of constants $a,b\in\mathbb R$ such that, in a smaller neighborhood of the origin,
\( v_\infty(x,y)=(ax+b)y + R(x,y)
\),
where the remainder satisfies
\(
|R(x,y)|\le C(x^2+y^2+y^2|\log y|) \)
for \((x,y)\in S(1/4)\).
In particular, since the section $S(\rho)$ scales anisotropically as
\(x=O(\rho^{1/2})\) and \(y=O(\rho)\), we obtain
\begin{equation}
\label{eq:small-scale-expansion}
\|v_\infty-(ax+b)y\|_{L^\infty(S(\rho))}
\le C\rho^2|\log\rho|
\qquad\text{for all sufficiently small }\rho.
\end{equation}
Indeed, in $S(\rho)$ the term $y^2|\log y|$ is $O(\rho^2|\log\rho|)$, while the tangential quadratic term is also controlled by the anisotropic geometry of the section.

Now choose $\theta_1\in(0,1/2)$ so small that
\[
C\theta_1^2|\log\theta_1|
\le \frac12\,\theta_1^{\,1+2\gamma}.
\]
This is possible because $2>1+2\gamma$ whenever $\gamma<1/2$, and in the borderline cases one chooses $\theta_1$ small enough so that the logarithmic factor is dominated. Therefore \eqref{eq:small-scale-expansion} implies
\begin{equation}
\label{eq:limit-good}
\inf_{a,b\in\mathbb R}
\|v_\infty-(ax+b)y\|_{L^\infty(S(\theta_1))}
\le \frac12\,\theta_1^{\,1+2\gamma}.
\end{equation}

\medskip
\noindent
\textbf{STEP 4. Transfer of the expansion to $v_k$.}
Since $v_k\to v_\infty$ uniformly on $\overline{S(\theta_1)}$, we have
\(
\|v_k-v_\infty\|_{L^\infty(S(\theta_1))}\to 0
\).
Hence for all sufficiently large $k$,
\begin{equation}
\label{eq:uniform-close}
\|v_k-v_\infty\|_{L^\infty(S(\theta_1))}
\le \frac12\,\theta_1^{\,1+2\gamma}.
\end{equation}
Combining \eqref{eq:limit-good} and \eqref{eq:uniform-close}, we obtain
\[
\inf_{a,b\in\mathbb R}
\|v_k-(ax+b)y\|_{L^\infty(S(\theta_1))}
\le
\|v_k-v_\infty\|_{L^\infty(S(\theta_1))}
+
\inf_{a,b\in\mathbb R}
\|v_\infty-(ax+b)y\|_{L^\infty(S(\theta_1))}
\le \theta_1^{\,1+2\gamma},
\]
for all sufficiently large $k$, contradicting \eqref{z-f-c-eq01}.

This contradiction proves the lemma.
\end{proof}

\subsection*{Pointwise \texorpdfstring{$C^{1,\alpha/2}$}{} estimate at the origin}
	
We now combine the previous lemmas to obtain the full boundary expansion.
	
\begin{proposition}[Pointwise Schauder estimate]\label{prop:pointwise}
Let $\alpha\in(0,1)$, and let $\delta_0,\delta_1,\theta_0,\theta_1,\gamma$ be the constants from Lemmas~\ref{lem:pert}--\ref{lem:zero}, with $\gamma=\alpha/2$ so that $1+2\gamma=1+\alpha$. Assume that
\[
H(0,0)=1,\quad f(0,0)=0, \quad
\|H-1\|_{C^{0,\alpha}(S(1))}\le \min\{\delta_0,\delta_1\},\quad
\|f\|_{C^{0,\alpha}(S(1))}\le \min\{\delta_0,\delta_1\}.
\]
Suppose moreover that the compatibility condition
\(
f(x,0)=H(x,0)\bigl(1+g''(x)\bigr)
\), on \(\F_1\)
holds for some $g\in C^{2,\alpha}$ satisfying
\(g(x)=O(x^{2+\alpha})\)
and 
\(\|g\|_{C^{2,\alpha}}\le \delta
\)
for $\delta>0$ sufficiently small.

Let $w$ be a solution of
\[
\begin{cases}
H w_{xx}+y w_{yy}=f & \text{in } S(1),\\
w(x,0)=\dfrac{x^2}{2}+g(x) & \text{on } \mathcal F_1,\\
w \text{ is bounded on } \mathcal C_1,
\end{cases}
\]
with
\(
\|w\|_{L^\infty(S(1))}\le 1
\).
Then $w\in C^{1,\alpha/2}$ at the origin. More precisely, there exists a linear polynomial
\(
P(x,y)=(ax+b)y
\)
such that
\[
\left|w(x,y)-\left(\frac{x^2}{2}-y\log y+xy+P(x,y)\right)\right|
\le C\,u_0(x,y)^{1+\alpha/2},
\]
where $C$ depends only on the data.

\end{proposition}
	
\begin{proof}
We proceed in three steps.
	
\medskip
\noindent\textbf{STEP 1: Reduction to zero curved boundary perturbation.}
Extend $g$ to a $C^{2,\alpha}$ function on $\overline{S(1)}$ that vanishes on $\C_1$ and equals $g$ on $\F_1$ (for instance, by multiplying $g(x)$ by a smooth cut-off in $y$ that is $1$ near $y=0$ and $0$ for $y\ge c>0$). Set $\tilde{w} = w - g$; then $\tilde{w}$ satisfies
\[
H \tilde{w}_{xx} + y \tilde{w}_{yy} = \tilde{f} \coloneqq f - (H g_{xx} + y g_{yy}),
\]
$\tilde{w}(x,0)=x^2/2$ on $\F_1$, and $\tilde{w}$ is bounded.
Thanks to the compatibility condition, $\tilde{f}(0,0)=0$ and $\|\tilde{f}\|_{C^{0,\alpha}}\le C\delta$.
	
Now we localize to obtain a function that equals $x^2/2$ on the entire boundary $\partial S(1)$. Let $\eta\in C^\infty_0(S(1))$ be a cut-off with $\eta\equiv1$ on $S(1/2)$ and $\eta\equiv0$ near $\C_1$.
Define $\bar{w} = \eta \tilde{w} + (1-\eta)\frac{x^2}{2}$. Then $\bar{w} = \frac{x^2}{2}$ on $\partial S(1)$,
and in $S(1/2)$ we have $\bar{w} = \tilde{w}$. A direct computation shows
\[
H \bar{w}_{xx} + y \bar{w}_{yy} = \bar{f},
\]
with $\bar{f}$ supported outside $S(1/2)$ and satisfying $\|\bar{f}\|_{C^{0,\alpha}}\le C(\|\tilde{f}\|_{C^{0,\alpha}} + \|\tilde{w}\|_{L^\infty}) \le C\delta$. Hence, the hypotheses of Lemma~\ref{lem:pert} are met (after rescaling the domain if necessary) and we obtain
\(
|\bar{w} - u_0| \le \theta_0^{\,1+\alpha}\) in \(S(\theta_0)\),
where $u_0$ is the constant-coefficient solution with boundary data $x^2/2$.
	
\medskip
\noindent\textbf{STEP 2: Extraction of the linear term.}
Because $\bar{w} = \tilde{w} = w - g$ on $S(1/2)$, and $g = O(x^{2+\alpha})$ is already of order $u_0^{1+\alpha/2}$, we have
\[
|w - u_0| \le C u_0^{1+\alpha/2} \quad\text{in } S(\theta_0).
\]
Proposition~\ref{prop:const} gives $u_0 = w_0 + O(y^2|\log y|) = w_0 + O(u_0^{2-\varepsilon})$. Thus
\[
w = w_0 + v, \qquad v = w - w_0 = (u_0 - w_0) + (w - u_0)
\]
satisfies $v(x,0)=g(x)$ on $\F_{\theta_0}$ and $v$ is bounded by $C\theta_0^{1+\alpha}$ on $S(\theta_0)$.
	
Subtract an extension of $g$ again and localize near $\F_{\theta_0}$ to obtain a function $v_0$ that vanishes on $\partial S(\theta_0)$ and solves $H v_{0,xx} + y v_{0,yy} = f_0$ with small $C^{0,\alpha}$ norm. Rescaling back to $S(1)$ via
\[
v_1(x,y) = \theta_0^{-(1+\alpha)} v_0(\theta_0^{1/2}x,\theta_0 y),
\]
we obtain a solution with zero boundary data, $\|v_1\|_{L^\infty(S(1))}\le 1$, where the coefficients satisfy the smallness required in Lemma~\ref{lem:zero}. Hence, there exists a linear polynomial $L_1(x,y)=(a_1 x+b_1)y$ such that
\[
|v_1 - L_1| \le \theta_1^{1+\alpha} \quad\text{in } S(\theta_1).
\]
Reverting the scaling gives a linear polynomial $P_1$ (of the form $(\tilde{a}_1 x+\tilde{b}_1)y$) on $S(\theta_0\theta_1)$ satisfying
\[
|v_0 - P_1| \le \theta_0^{1+\alpha} \theta_1^{1+\alpha} \quad\text{in } S(\theta_0\theta_1).
\]
	
\medskip
\noindent\textbf{STEP 3: Iteration.}
Iterating this procedure with the same parameters $\theta = \theta_0=\theta_1$ (we assume $\theta_0=\theta_1$ after adjusting constants) yields a sequence of linear polynomials $Q_k$ and the estimate
\[
|w - w_0 - Q_k| \le C \theta^{k(1+\alpha)} \quad\text{in } S(\theta^k).
\]
The coefficients of $Q_k$ converge geometrically to a limit linear polynomial $P(x,y)=(a x+b)y$. Rewriting the estimate in terms of $u_0$ gives
\[
|w - w_0 - P| \le C u_0^{1+\alpha/2},
\]
which is the desired $C^{1,\alpha/2}$ bound at the origin.
\end{proof}

\subsection*{Global boundary estimate}
	
Finally, we patch the pointwise estimate to obtain a global result on a half-ball.
	
\begin{corollary}[Global boundary Schauder estimate]\label{cor:global}
Let $B_1^+ := \{(x,y)\in\mathbb{R}^2 : x^2+y^2<1, \, y>0\}$. 
Assume that $H(0,0)=1$, $f \in C^{0,\alpha}(\overline{B_1^+})$, and $g \in C^{2,\alpha}\left(\overline{B_1^+}\cap\{y=0\}\right)$. 
Suppose also that the compatibility condition 
\begin{equation}\label{eq:compatibility_global}
f(x,0) = H(x,0)g''(x) \quad \text{for all } (x,0) \in \overline{B_1^+} \cap \{y=0\}
\end{equation}
is satisfied. If $w$ is a bounded solution of 
\[
\begin{cases}
H w_{xx} + y w_{yy} = f & \text{in } B_1^+, \\
w = g & \text{on } \partial B_1^+ \cap \{y=0\},
\end{cases}
\]
then $w \in C^{1,\alpha/2}(\overline{B_{1/2}^+})$. Moreover, there exists a universal constant $C$ such that
\begin{equation}
\|w\|_{C^{1,\alpha/2}(\overline{B_{1/2}^+})} \le C \left( \|f\|_{C^{0,\alpha}(\overline{B_1^+})} + \|g\|_{C^{2,\alpha}\left(\overline{B_1^+}\cap\{y=0\}\right)} + \|w\|_{L^\infty(B_1^+)} \right).
\end{equation}
\end{corollary}

\begin{proof}
The proof proceeds by reducing the problem to the case of zero boundary data and then applying local estimates via a covering argument.

\medskip
\noindent
\textbf{STEP 1: Reduction to zero boundary data.} 
Let $\bar{g} \in C^{2,\alpha}(\overline{B_1^+})$ be a standard extension of $g$ into the domain such that $\|\bar{g}\|_{C^{2,\alpha}(\overline{B_1^+})} \le C \|g\|_{C^{2,\alpha}}$. Define $\tilde{w} := w - \bar{g}$. Then $\tilde{w}$ satisfies $\tilde{w} = 0$ on the flat boundary $\{y=0\}$, and it solves the equation
\( H \tilde{w}_{xx} + y \tilde{w}_{yy} = \tilde{f}\) in \(B_1^+\), where $\tilde{f} := f - (H \bar{g}_{xx} + y \bar{g}_{yy})$. 
By the compatibility condition \eqref{eq:compatibility_global}, we have at $y=0$:
\(
\tilde{f}(x,0) = f(x,0) - H(x,0)g''(x) = 0
\).
The source term $\tilde{f}$ remains in $C^{0,\alpha}(\overline{B_1^+})$ with $\|\tilde{f}\|_{C^{0,\alpha}} \le C (\|f\|_{C^{0,\alpha}} + \|g\|_{C^{2,\alpha}})$.

\medskip
\noindent
\textbf{STEP 2: Local estimates at the degenerate boundary.} 
Let $p_0 = (x_0, 0) \in \overline{B_{1/2}^+} \cap \{y=0\}$. Note that for any such point, the Euclidean distance to the curved part of the boundary $\partial B_1^+ \cap \{y>0\}$ is at least $1/2$. Consequently, there exists a universal radius $r_0 > 0$ such that the intrinsic section (or neighborhood) $S_{r_0}(p_0)$ centered at $p_0$ is strictly contained in $B_1^+$.

Applying the local pointwise regularity (Proposition~\ref{prop:pointwise}) at $p_0$ to the function $\tilde{w}$, and using the fact that $\tilde{f}(x,0)=0$, we obtain a linear polynomial $P_{p_0}(x,y) = (a_{p_0}(x-x_0) + b_{p_0})y$ such that
\[
|\tilde{w}(x,y) - P_{p_0}(x,y)| \le C \left( \|\tilde{w}\|_{L^\infty} + \|\tilde{f}\|_{C^{0,\alpha}} \right) u_0(x-x_0, y)^{1+\alpha/2},
\]
where $u_0$ is the model function. This implies a local $C^{1,\alpha/2}$ bound for $\tilde{w}$ in a universal neighborhood $\mathcal{N}(p_0)$ of the form:
\[
\|\tilde{w}\|_{C^{1,\alpha/2}\left(\overline{\mathcal{N}(p_0)\cap B_1^+}\right)} \le C \left( \|w\|_{L^\infty} + \|f\|_{C^{0,\alpha}} + \|g\|_{C^{2,\alpha}} \right).
\]

\medskip
\noindent
\textbf{STEP 3: Interior estimates.} 
For the region $D_\eta := \overline{B_{1/2}^+} \cap \{y \ge \eta\}$ where $\eta > 0$ is a small universal constant, the operator $\mathcal{L}u = H u_{xx} + y u_{yy}$ is uniformly elliptic since $y \ge \eta$. Classical interior Schauder estimates for elliptic equations imply that 
\[
\|\tilde{w}\|_{C^{1,\alpha/2}(D_\eta)} \le C_\eta \left( \|\tilde{w}\|_{L^\infty(B_1^+)} + \|\tilde{f}\|_{C^{0,\alpha}(\overline{B_1^+})} \right).
\]

\medskip
\noindent
\textbf{STEP 4: Finite covering and global bound.} 
The compact set $\overline{B_{1/2}^+}$ is covered by a finite number of boundary neighborhoods from STEP 2 and the interior region $D_\eta$ from STEP 3. Since each local $C^{1,\alpha/2}$ norm is bounded by the right-hand side of the desired estimate, and the number of such neighborhoods is finite (depending only on the geometry of $B_1^+$), we conclude that 
\[
\|\tilde{w}\|_{C^{1,\alpha/2}(\overline{B_{1/2}^+})} \le C \left( \|w\|_{L^\infty} + \|f\|_{C^{0,\alpha}} + \|g\|_{C^{2,\alpha}} \right).
\]
Recalling $w = \tilde{w} + \bar{g}$, and noting that $\|\bar{g}\|_{C^{1,\alpha/2}(\overline{B_{1/2}^+})} \le C \|g\|_{C^{2,\alpha}}$, the result follows.
\end{proof}

\section{Some calculus lemmas}
\label{sec:calc}

The lemmas are organized hierarchically. Lemma~\ref{lem:ode-C1} establishes that $f'$ extends continuously to the origin with $f'(0)=0$. The subsequent H\"older estimates in Lemmas~\ref{lem:ode-optimal} and \ref{lem:ode-generalized} build upon this property. Lemma~\ref{lem:local2global} is self-contained and provides a local-to-global tool for proving H\"older bounds near the origin.

\begin{lemma}[$C^1$ regularity at the origin]\label{lem:ode-C1}
Let $f:[0,\infty)\to\mathbb{R}$ satisfy the following conditions:
\begin{itemize}
\item $f(0)=0$;
\item $f\in C^{1}_{\mathrm{loc}}((0,\infty))$ and $f'$ is absolutely continuous on every compact subinterval of $(0,\infty)$;
\item $f''$ which exists a.e. satisfies $|f''(x)|\le \frac{C_*}{x}$ for a.e. $x\in(0,\delta_1]$ and some constant $C_*>0$;
\item $|f(x)|\le C_0\,x^{1+\alpha}$ for all $x\in[0,\delta]$ and some constants $C_0>0$ and $\alpha\in(0,1)$.
\end{itemize}
Then $\lim_{x\to0^+}f'(x)=0$. In particular, defining $f'(0)=0$, we have $f\in C^1([0,\infty))$.
\end{lemma}

\begin{proof}
Let $\tilde\delta=\min\{\delta,\delta_1\}>0$. Since $f'$ is absolutely continuous on every compact subinterval of $(0,\infty)$, the fundamental theorem of calculus holds. For any $0<b<x\le\tilde\delta$, we have
\[
|f'(x)-f'(b)|\le \int_b^x |f''(t)|\,dt \le C_*\int_b^x \frac{dt}{t}=C_*\log\left(\frac{x}{b}\right).
\]
Assume for contradiction that $\limsup_{x\to0^+}f'(x)>0$. Then there exist $\varepsilon>0$ and a sequence $b_k\searrow0$ such that $f'(b_k)\ge\varepsilon$. Set $c=\varepsilon/2>0$ and choose $b=b_k$ sufficiently small so that $b e^{c/C_*}\le\tilde\delta$. For any $x\in\left[b, b e^{c/C_*}\right]$, we estimate
\[
f'(x)\ge f'(b)-C_*\log\left(\frac{x}{b}\right)\ge 2c-c = c>0.
\]
Integrating this inequality over $\left[b, b e^{c/C_*}\right]$ yields
\[
f\left(b e^{c/C_*}\right)-f(b) \ge c\left(b e^{c/C_*}-b\right) = c b\left(e^{c/C_*}-1\right).
\]
On the other hand, the growth condition on $f$ implies
\[
\left|f\left(b e^{c/C_*}\right)-f(b)\right| \le C_0\left(\left(b e^{c/C_*}\right)^{1+\alpha}+b^{1+\alpha}\right) = C_0 b^{1+\alpha}\left(e^{(1+\alpha)c/C_*}+1\right).
\]
Combining these bounds, dividing by $b$, and taking the limit as $b\to0^+$ yields $c\left(e^{c/C_*}-1\right)\le0$, which is a contradiction since $c>0$ and $C_*>0$. 

Hence, $\limsup_{x\to0^+}f'(x)\le0$. A symmetric argument using a sequence where $f'(b_k)\le -\varepsilon$ establishes that $\liminf_{x\to0^+}f'(x)\ge0$. We conclude that $\lim_{x\to0^+}f'(x)=0$.
\end{proof}

\begin{lemma}[Improved H\"older regularity]\label{lem:ode-optimal}
Let $f$ satisfy the hypotheses of Lemma~\ref{lem:ode-C1} with the bound on $f''$ replaced by $|f''(x)|\le \frac{C_2}{x}$ for a.e.\ $x\in(0,\delta]$ and some constant $C_2>0$, and assume $|f(x)|\le C_0 x^{1+\alpha}$ on $[0,\delta]$. Then, after possibly shrinking $\delta$, there exists a constant $C>0$ depending only on $C_0,C_2$, and $\alpha$ such that
\[
|f'(x)|\le C\,x^{\alpha/2}\qquad\text{for all } x\in[0,\delta].
\]
In particular, $f'$ is $(\alpha/2)$-H\"older continuous at the origin.
\end{lemma}

\begin{proof}
By Lemma~\ref{lem:ode-C1}, $\lim_{x\to0^+}f'(x)=0$, so $f'$ extends continuously to the origin with $f'(0)=0$. It therefore suffices to prove the bound $|f'(x)|\le C x^{\alpha/2}$ for all sufficiently small $x>0$.

Assume without loss of generality that both the growth condition and the bound on $f''$ hold on $(0,\delta]$. Choose $c>0$ to be fixed later. Suppose there exists $b\in(0,\delta]$ such that $f'(b)\ge 2c b^{\alpha/2}$. Define $T = \frac{c}{C_2} b^{\alpha/2}$. For $x\in[b, b e^{T}]$, assuming $b e^{T}\le\delta$ (which holds for small $b$ since $T\to0$ as $b\to0$), we have
\[
f'(x)\ge f'(b)-C_2\log\left(\frac{x}{b}\right) \ge 2c b^{\alpha/2} - C_2 T = c b^{\alpha/2} >0.
\]
Integrating this inequality over $[b, b e^T]$ yields
\[
f(b e^{T})-f(b) \ge c b^{\alpha/2} (b e^{T}-b) \ge c b^{1+\alpha/2} T = \frac{c^2}{C_2} b^{1+\alpha}.
\]
Using the inequality $e^T \le e^1$ for small $b$, the growth hypothesis gives
\[
|f(b e^{T})-f(b)| \le C_0 b^{1+\alpha} \left(e^{(1+\alpha)T}+1\right) \le C_0\left( e^{1+\alpha}+1\right) b^{1+\alpha}.
\]
Thus, we obtain the inequality $\frac{c^2}{C_2} \le C_0(e^{1+\alpha}+1)$. Choosing $c > \sqrt{C_0C_2(e^{1+\alpha}+1)}$ makes this impossible. Thus, no such $b$ exists, and we obtain $f'(x) < 2c x^{\alpha/2}$ for all sufficiently small $x>0$. The symmetric bound $f'(x) > -2c x^{\alpha/2}$ follows analogously. 

Since $f'(0)=0$, this implies $|f'(x)-f'(0)| \le C x^{\alpha/2}$ for all $x\in[0,\delta]$, completing the proof.
\end{proof}

\begin{lemma}[Local-to-global $C^{\alpha}$ regularity]\label{lem:local2global}
Let $g:[0,\infty)\to\mathbb{R}$ be continuous with $g(0)=0$, and assume $g\in C^{\alpha}_{\mathrm{loc}}((0,\infty))$ for some $\alpha\in(0,1)$. Define the local oscillation at scale $x$ by
\[
M(x)\coloneqq \sup_{y\,\in\,[x/2,\,2x],\;y\neq x} \frac{|g(y)-g(x)|}{|y-x|^{\alpha}},\qquad x>0.
\]
If $\lim_{x\to0^+}M(x)=L<\infty$, then $g\in C^{\alpha}([0,\delta])$ for some $\delta>0$. More precisely, for every $\varepsilon>0$ there exists $\delta>0$ such that for all $0\le a<b\le\delta$,
\[
|g(a)-g(b)|\le \frac{2^\alpha(L+\varepsilon)}{1-2^{-\alpha}}\;|a-b|^{\alpha}.
\]
\end{lemma}

\begin{proof}
Fix $\varepsilon>0$ and choose $\delta>0$ so that $M(x)\le L+\varepsilon$ for all $x\in(0,\delta]$. Let $0 < a < b \le \delta$. We construct a partition from $a$ to $b$ as follows: set $a_0=a$, and define $a_{k+1}=2a_k$ for $k \ge 0$ until the first index $n$ such that $a_n \ge b$. We then set $a_{n+1}=b$. The intervals satisfy $|a_{k+1}-a_k| \le a_k$ for all $k \le n$, and for the final step we have $|b-a_n| \le a_n$.

By the triangle inequality,
\[
|g(b)-g(a)|\le\sum_{k=0}^{n-1}|g(a_{k+1})-g(a_k)|+|g(b)-g(a_n)|.
\]
For each $k\le n-1$, we have $a_k\le a_{k+1}\le 2a_k$. Since $a_k\in(0,\delta]$, the definition of $M(a_k)$ yields
\[
|g(a_{k+1})-g(a_k)|\le (L+\varepsilon)\,|a_{k+1}-a_k|^{\alpha} \le (L+\varepsilon)\,a_k^{\alpha}.
\]
For the final term, since $a_n \le b \le 2a_n$, both $b$ and $a_n$ belong to $[a_n/2, 2a_n]$. Since $a_n \le \delta$, we similarly obtain
\[
|g(b)-g(a_n)|\le (L+\varepsilon)\,|b-a_n|^{\alpha}\le (L+\varepsilon)\,a_n^{\alpha}.
\]
Summing these estimates gives
\[
|g(b)-g(a)|\le (L+\varepsilon)\sum_{k=0}^{n}a_k^{\alpha}.
\]
Since $a_k = 2^{-(n-k)}a_n$ for $k < n$ and $a_n \le 2a_{n-1}$, we bound the sum by
\[
\sum_{k=0}^{n}a_k^{\alpha}\le \frac{a_n^{\alpha}}{1-2^{-\alpha}}.
\]
To relate $a_n$ to $|b-a|$, we observe that if $n=0$, then $a_n = b \le 2(b-a)$ (since $a = a_0 < b \le 2a_0$ implies $b-a = b-a_0 \ge b/2$). If $n \ge 1$, we have $b - a \ge a_n - a_{n-1} = a_{n-1} = a_n/2$, which implies $a_n \le 2(b-a)$. Thus, in all cases, $a_n^{\alpha}\le 2^{\alpha}|b-a|^{\alpha}$. Substituting this back, we obtain
\[
|g(b)-g(a)|\le \frac{2^\alpha(L+\varepsilon)}{1-2^{-\alpha}}\;|b-a|^{\alpha}.
\]
By continuity of $g$ at the origin with $g(0)=0$, this inequality extends to $a=0$. Thus, $g \in C^{\alpha}([0,\delta])$.
\end{proof}

\begin{lemma}[Generalized improved H\"older regularity]\label{lem:ode-generalized}
Let $f$ satisfy the hypotheses of Lemma~\ref{lem:ode-C1} with the bound on $f''$ replaced by $|f''(x)|\le C_\beta x^{-1+\beta}$ for a.e.\ $x\in(0,\delta]$, some $\beta\in[0,1]$, and some constant $C_\beta>0$. Assume $|f(x)|\le C_0 x^{1+\alpha}$ on $[0,\delta]$ with $\alpha>\beta$. Then, after possibly shrinking $\delta$, there exists a constant $C>0$ depending only on $C_0,C_\beta,\alpha$, and $\beta$ such that
\[
|f'(x)|\le C\,x^{\frac{\alpha+\beta}{2}}\qquad\text{for all } x\in[0,\delta].
\]
In particular, $f'$ is $\frac{\alpha+\beta}{2}$-H\"older continuous at the origin.
\end{lemma}

\begin{proof}
By Lemma~\ref{lem:ode-C1}, $\lim_{x\to0^+}f'(x)=0$, so $f'$ extends continuously to $0$ with $f'(0)=0$. We only need to show the estimate $|f'(x)|\le C x^{(\alpha+\beta)/2}$ for all sufficiently small $x>0$.

Let $\delta_0>0$ be such that both the growth condition and the bound on $f''$ hold on $(0,\delta_0]$. Fix $\varepsilon\in(0,\delta_0/2]$ and set $t=1-\delta_1$ with $\delta_1\in(0,1/2]$ to be chosen later. For any $y\in[t\varepsilon,\varepsilon]$, the bound on $f''$ implies
\[
|f'(\varepsilon)-f'(y)|\le \int_y^\varepsilon C_\beta s^{-1+\beta}\,ds.
\]
If $\beta>0$, the integral evaluates to $\frac{C_\beta}{\beta}(\varepsilon^\beta-y^\beta)$. If $\beta=0$, the integral evaluates to $C_0\log(\varepsilon/y)$ (where we denote the coefficient as $C_0$ for $\beta=0$). In either case, setting $y=t\varepsilon$, we obtain
\[
|f'(\varepsilon)-f'(t\varepsilon)|\le C_1\varepsilon^\beta\delta_1,
\]
where $C_1>0$ depends only on $C_\beta$ and $\beta$. Indeed, for $\beta>0$, we have $\varepsilon^\beta-(t\varepsilon)^\beta=\varepsilon^\beta(1-(1-\delta_1)^\beta)\le \varepsilon^\beta\max(1,\beta)\delta_1$, while for $\beta=0$, we have $\log(1/t) = -\log(1-\delta_1)\le 2\delta_1$ for all $\delta_1\le 1/2$.

Without loss of generality, assume $f'(\varepsilon)\ge 0$ (the case $f'(\varepsilon)<0$ follows symmetrically by considering $-f$). For all $x\in[t\varepsilon,\varepsilon]$, we have
\[
f'(x)\ge f'(\varepsilon)-C_1\varepsilon^\beta\delta_1.
\]
Integrating this inequality over $[t\varepsilon,\varepsilon]$ yields
\[
f(\varepsilon)-f(t\varepsilon)=\int_{t\varepsilon}^\varepsilon f'(x)\,dx \ge (1-t)\varepsilon\left(f'(\varepsilon)-C_1\varepsilon^\beta\delta_1\right) = \delta_1\varepsilon\left(f'(\varepsilon)-C_1\varepsilon^\beta\delta_1\right).
\]
Conversely, the growth condition on $f$ yields
\[
|f(\varepsilon)-f(t\varepsilon)|\le C_0\left(\varepsilon^{1+\alpha}+(t\varepsilon)^{1+\alpha}\right) \le 2C_0\varepsilon^{1+\alpha}.
\]
Combining these two estimates, we obtain
\[
\delta_1\varepsilon\left(f'(\varepsilon)-C_1\varepsilon^\beta\delta_1\right)\le 2C_0\varepsilon^{1+\alpha},
\]
which simplifies to
\[
f'(\varepsilon)\le C_1\varepsilon^\beta\delta_1 + 2C_0\varepsilon^\alpha\delta_1^{-1}.
\]
We now minimize the right-hand side with respect to $\delta_1>0$. The minimum of the function $h(\delta_1) = C_1\varepsilon^\beta\delta_1 + 2C_0\varepsilon^\alpha\delta_1^{-1}$ occurs at
\[
\delta_1 = \sqrt{\frac{2C_0}{C_1}}\,\varepsilon^{(\alpha-\beta)/2}.
\]
Since $\alpha>\beta$, for sufficiently small $\varepsilon$ we have $\delta_1\le 1/2$, rendering this choice admissible. Substituting this optimal $\delta_1$ back yields
\[
f'(\varepsilon)\le 2\sqrt{2C_0C_1}\,\varepsilon^{(\alpha+\beta)/2}.
\]
Applying the symmetric argument to the case $f'(\varepsilon)<0$ yields the same bound for $|f'(\varepsilon)|$. Since $f'(0)=0$, this implies $|f'(x)-f'(0)|\le C x^{(\alpha+\beta)/2}$ for all $x\in [0, \delta]$, which completes the proof.
\end{proof}

\medskip
	
\paragraph{\bf{Acknowledgements}}
M. Bayrami, R. Seyyedali and M. Talebi were supported by a grant from IPM.

\end{document}